\documentclass[a4paper, draft, 12pt]{amsart}
\usepackage{amsthm,amssymb}
\usepackage{dsfont}
\usepackage{mathrsfs} 
\usepackage[all,cmtip]{xy}
\usepackage{patchcmd}
\input{diagxy}
\usepackage{url}
\usepackage{bm}
\usepackage{mathtools}
\usepackage{enumerate}
\usepackage{enumitem}
\usepackage{pbox}
\usepackage[top=2.4cm, bottom=2.4cm, left=2cm, right=2cm]{geometry}
\usepackage[normalem]{ulem}
\usepackage{tikz}
\usetikzlibrary{arrows,positioning}
\allowdisplaybreaks

\makeatletter
\patchcommand\@starttoc{\begin{quote}}{\end{quote}}
\def\@tocline#1#2#3#4#5#6#7{\relax
  \ifnum #1>\c@tocdepth 
  \else
    \par \addpenalty\@secpenalty\addvspace{#2}%
    \begingroup \hyphenpenalty\@M
    \@ifempty{#4}{%
      \@tempdima\csname r@tocindent\number#1\endcsname\relax
    }{%
      \@tempdima#4\relax
    }%
    \parindent\z@ \leftskip#3\relax \advance\leftskip\@tempdima\relax
    \rightskip\@pnumwidth plus4em \parfillskip-\@pnumwidth
    #5\leavevmode\hskip-\@tempdima
      \ifcase #1
       \or\or \hskip 1em \or \hskip 2em \else \hskip 3em \fi%
      #6\nobreak\relax
    \dotfill\hbox to\@pnumwidth{\@tocpagenum{#7}}\par
    \nobreak
    \endgroup
  \fi}
\makeatother

 \theoremstyle{plain}
 \newtheorem{thm}{Theorem}[section]
 \newtheorem{cor}[thm]{Corollary}
 \newtheorem{lem}[thm]{Lemma}
 
 \newtheorem{prop}[thm]{Proposition}

\theoremstyle{definition}
 \newtheorem{defn}[thm]{Definition}

\theoremstyle{remark}
 \newtheorem{rem}[thm]{Remark}
 \newtheorem{ques}[thm]{Question}
 
 \newtheorem{ter}[thm]{Terminology}
 \newtheorem{nota}[thm]{Notation}
 \newtheorem{conv}[thm]{Convention}

 \numberwithin{equation}{section}

\theoremstyle{plain}

\DeclareMathOperator{\VF}{VF}

\DeclareMathOperator{\RV}{RV}
\DeclareMathOperator{\DC}{DC}
\DeclareMathOperator{\MM}{\mathcal{M}}

\DeclareMathOperator{\OO}{\mathcal{O}}

 \DeclareMathOperator{\ran}{ran}
 \DeclareMathOperator{\dom}{dom}

 \DeclareMathOperator{\id}{id}

 \DeclareMathOperator{\lh}{lh}

 \DeclareMathOperator{\dcl}{dcl}
 \DeclareMathOperator{\pr}{pr}

 \DeclareMathOperator{\mgl}{GL}

\DeclareMathOperator{\jcb}{Jcb}

\DeclareMathOperator{\K}{\Bbbk}

\DeclareMathOperator{\res}{res}  

\def\Xint#1{\mathchoice
{\XXint\displaystyle\textstyle{#1}}%
{\XXint\textstyle\scriptstyle{#1}}%
{\XXint\scriptstyle\scriptscriptstyle{#1}}%
{\XXint\scriptscriptstyle\scriptscriptstyle{#1}}%
\!\int}
\def\XXint#1#2#3{{\setbox0=\hbox{$#1{#2#3}{\int}$}
\vcenter{\hbox{$#2#3$}}\kern-.5\wd0}}

\newcommand{\Z}{\mathds{Z}}

\newcommand{\F}{\mathds{F}}

\newcommand{\KKK}{\mathds{K}}

\newcommand{\N}{\mathds{N}}

\newcommand{\R}{\mathds{R}}

\newcommand{\omin}{$o$\nobreakdash}

\newcommand{\T}{$T$\nobreakdash}

\newcommand{\ga}{\mathfrak{a}}
\newcommand{\gb}{\mathfrak{b}}
\newcommand{\gc}{\mathfrak{c}}

\newcommand{\gh}{\mathfrak{h}}

\newcommand{\go}{\mathfrak{o}}
\newcommand{\gp}{\mathfrak{p}}
\newcommand{\gq}{\mathfrak{q}}

\newcommand{\0}{\emptyset}

\DeclareMathAlphabet{\mathpzc}{OT1}{pzc}{m}{it}

 \DeclarePairedDelimiter\abs{\lvert}{\rvert}

 \newcommand{\set}[1]{\left\{#1\right\}}

 \newcommand{\usub}[2]{#1_{\textup{#2}}}

 \newcommand{\lan}[3]{\mathcal{L}_{#1 \textup{#2} #3}}

\newcommand{\mdl}[1]{\mathcal{#1}}  
\newcommand{\bb}[1]{\mathbb{#1}}

\newcommand{\limplies}{\rightarrow}

\newcommand{\ex}[1]{\exists #1 \;} 

\newcommand{\rest}{\upharpoonright}

\newcommand{\fun}{\longrightarrow}
\newcommand{\efun}{\longmapsto}
\newcommand{\sub}{\subseteq}

\newcommand{\mi}{\smallsetminus}

\newcommand{\la}{\langle}
\newcommand{\ra}{\rangle}

\DeclareMathOperator{\mVF}{{\mu}{\VF}}
\DeclareMathOperator{\mgVF}{{\mu_{\Gamma}} {\VF}}
\DeclareMathOperator{\mRV}{{\mu}{\RV}}
\DeclareMathOperator{\mgRV}{{\mu_{\Gamma}}{\RV}}
\DeclareMathOperator{\mG}{{\mu}{\Gamma}}
\DeclareMathOperator{\mgRES}{{\mu_{\Gamma}}{\RES}}
\DeclareMathOperator{\mRES}{{\mu}{\RES}}

\DeclareMathOperator{\misp}{\mu I_{sp}}
\DeclareMathOperator{\mgisp}{{\mu_{\Gamma}}{I_{sp}}}
\DeclareMathOperator{\mispdb}{\mu D_{sp}}
\DeclareMathOperator{\mgispdb}{{\mu_{\Gamma}}{D_{sp}}}
\DeclareMathOperator{\mgD}{{\mu_{\Gamma}}{\bb D}}
\DeclareMathOperator{\mgE}{{\mu_{\Gamma}}{\bb E}}
\DeclareMathOperator{\mgL}{{\mu_{\Gamma}}{\bb L}}
\DeclareMathOperator{\mL}{{\mu}{\bb L}}
\DeclareMathOperator{\vgRES}{{\vol_{\Gamma}}{\RES}}

\DeclareMathOperator{\rv}{rv}

\DeclareMathOperator{\vv}{val}

\DeclareMathOperator{\RES}{RES}

\DeclareMathOperator{\gsk}{\mathbf{K}_+}
\DeclareMathOperator{\ggk}{\mathbf{K}}

\DeclareMathOperator{\fn}{FN}
\DeclareMathOperator{\mfn}{{\mu} FN}
\DeclareMathOperator{\mgfn}{{\mu_{\Gamma}} FN}
\DeclareMathOperator{\fib}{fib}

\DeclareMathOperator{\vol}{vol}
\DeclareMathOperator{\db}{d \! b}

\DeclareMathOperator{\rad}{rad}

\DeclareMathOperator{\vrv}{vrv}

\DeclareMathOperator{\RVH}{RVH}

\DeclareMathOperator{\can}{\mathbf{c}}

\DeclareMathOperator{\ito}{int}

\newbox\gnBoxA
\newdimen\gnCornerHgt
\setbox\gnBoxA=\hbox{$\ulcorner$}
\global\gnCornerHgt=\ht\gnBoxA
\newdimen\gnArgHgt

\def\code #1{%
        \setbox\gnBoxA=\hbox{$#1$}%
        \gnArgHgt=\ht\gnBoxA%
        \ifnum \gnArgHgt<\gnCornerHgt
                \gnArgHgt=0pt%
        \else
                \advance \gnArgHgt by -\gnCornerHgt%
        \fi
        \raise\gnArgHgt\hbox{$\ulcorner$} \box\gnBoxA %
                \raise\gnArgHgt\hbox{$\urcorner$}}

\DeclareMathOperator{\TCVF}{TCVF}

\newcommand{\dand}{\quad \text{and} \quad}
\newcommand{\LT}{$\lan{T}{}{}$\nobreakdash}

\newcommand{\pard}[1]{\tfrac{\partial}{\partial x_{#1}}}

\newcommand{\npard}[1]{\tfrac{\partial^*}{\partial x^*_{}}}
\newcommand{\ddx}{\tfrac{d}{d x}}

\DeclareMathOperator{\sgn}{sgn}

\DeclareMathOperator{\abval}{\abs{\vv}}
\DeclareMathOperator{\abvrv}{\abs{\vrv}}
\DeclareMathOperator{\RVV}{RV_0}
\DeclareMathOperator{\GAA}{\Gamma_0}
\DeclareMathOperator{\absG}{\abs{\Gamma}}
\DeclareMathOperator{\mmdl}{\mdl R_{\rv}}

\DeclareMathOperator{\dsh}{\sharp\sharp}

\begin{document}

\title[Integration in $\TCVF$]{Integration in power-bounded $T$-convex valued fields}

\author[Y. Yin]{Yimu Yin}
\address{Santa Monica, California}
\email{yimu.yin@hotmail.com}

\thanks{The research leading to the true claims in this paper has been partially supported by the ERC Advanced Grant NMNAG, the grant ANR-15-CE40-0008 (D\'efig\'eo), and the SYSU grant 11300-18821101}


\begin{abstract}
This is the second installment of a series of papers aimed at developing a theory of Hrushovski-Kazhdan style motivic integration for certain types of nonarchimedean \omin-minimal fields, namely power-bounded \T-convex valued fields, and closely related structures. The main result in the first installment is a canonical isomorphism between the Grothendieck rings of certain categories of definable sets, which is understood as  a universal additive invariant or a generalized Euler characteristic  because  the categories do not carry volume forms. Here we introduce two types of volume forms into each of the relevant categories, one takes values in the value group and the other in the finer $\RV$-sort. The resulting isomorphisms respect Jacobian transformations --- that is, the change of variables formula holds --- and hence are regarded as motivic integrals. As in the classical theory of integration, one is often led to consider locally constant functions with bounded support in various situations. For the space of such functions, the construction may be fine-tuned so as to become more amenable to applications. The modifications are nevertheless substantial and constitute the bulk of the technical work.
\end{abstract}

\maketitle

\tableofcontents

\section{Introduction}\label{section:intro}

This paper is a sequel to \cite{Yin:tcon:I}, whose notational conventions, particularly those in \cite[\S~2.1]{Yin:tcon:I}, will be used throughout; reminders will be provided along the way. To begin with, let $T$ stand for a complete power-bounded \omin-minimal \LT-theory extending the theory $\usub{\textup{RCF}}{}$ of real closed fields and $\mdl R \coloneqq (R, <, \ldots)$ a sufficiently saturated model of $T$. Let $\OO$ be a \emph{proper} \T-convex subring of $\mdl R$ in the sense of \cite{DriesLew95}, that is, $\OO$  is a convex subring of $\mdl R$ such that, for every definable (no parameters allowed) continuous function $f : R \fun R$, we have $f(\OO) \sub \OO$. According to \cite{DriesLew95}, the theory $T_{\textup{convex}}$ of the pair $(\mdl R, \OO)$, suitably axiomatized in the language $\lan{}{convex}{}$ that extends $\lan{T}{}{}$ with a new unary relation symbol, is complete. We assume that $T$ admits quantifier elimination and is universally axiomatizable, which can always be arranged through definitional extension. Then $T_{\textup{convex}}$ admits quantifier elimination too.

To construct Hrushovski-Kazhdan style integrals, however, we need to work with a different language. Let $\vv : R^{\times} \fun \Gamma$ be the valuation map induced by $\OO$, $\K$ the corresponding residue field, and $\res : \OO \fun \K$ the residue map. There is a canonical way of turning the ordered field $\K$ into a \T-model as well. Let $\MM$ be the maximal ideal of $\OO$, $\RV$ the set $R^{\times} / (1 + \MM)$ of equivalence classes, and $\rv : R^{\times} \fun \RV$ the quotient map. For each $a \in R$, the  map $\vv$ is constant on the set $a + a\MM$, and hence there is an induced map $\vrv : \RV \fun \Gamma$. All of this structure can be expressed in a  two-sorted first-order language $\lan{T}{RV}{}$, in which $R$ is referred to as the $\VF$-sort and $\RV$ is taken as a new sort. The resulting theory $\TCVF$ (see \cite[Definition~2.7]{Yin:tcon:I}) is complete and weakly \omin-minimal, and admits quantifier elimination. Informally and for all practical purposes, the language $\lan{T}{RV}{}$ may be viewed as an extension of the language $\lan{}{convex}{}$.

Henceforth we shall primarily work in a sufficiently saturated $\TCVF$-model $\mmdl$, which is, up to isomorphism, the unique $\lan{T}{RV}{}$-expansion of the $T_{\textup{convex}}$-model $(\mdl R, \OO)$, together with a fixed small substructure $\mdl S$.

The category $\VF_*$ essentially consists of the definable subsets of $\VF^n$, $n \geq 0$, as objects and the definable bijections between them as morphisms. The category $\RV[k]$ essentially consists of the definable subsets of $\RV^k$ as objects and the definable bijections between them as morphisms. The category  $\RV[*]$  is the coproduct of $\RV[k]$, $k \geq 0$, and hence is equipped with a gradation by ambient dimensions. The  main construction of \cite{Yin:tcon:I} is a canonical isomorphism between the two  Grothendieck rings:
\[
 \int :  \ggk \VF_* \to \ggk \RV[*] / (\bm P - 1),
\]
where the (nonhomogenous) principal ideal $(\bm P - 1)$ is generated by the difference of $1 \in \ggk \RV[0]$, short for the element $[\{1\}]$, and $\bm P \in \ggk \RV[1]$, short for the element $[\rv(1 + \MM)] - [\rv(\MM \mi 0)]$. This isomorphism is understood as a universal additive invariant, and is also referred to as a generalized Euler characteristic when the target ring $\ggk \RV[*] / (\bm P - 1)$ is recast via natural isomorphism (or homomorphism) into something that is structurally more explicit.

A classical integrand of the form $\int_A \omega$ requires the following data: a differentiable manifold $M$,  a measurable subset $A$ of $M$, and a volume form $\omega$ on $M$ (a nowhere vanishing top-dimensional differential form). In the setting of this paper, we can represent such an integrand directly as a pair $(A, \omega)$, where  $A \sub \VF^k$ is an object of $\VF_*$ and the $\RV$-volume form $\omega : A \fun \RV$ is an arbitrary definable function (so the ambient manifold is in effect always the affine space $\VF^k$). Taking values in the $\RV$-sort corresponds to the identification of a volume form $\omega$ with another one $f\omega$ when $f -1 \in \MM$. The motivic integral that shall be constructed is intended as an invariant associated with a category (actually a groupoid) and hence should  not depend on the representatives of an isomorphism class. This means that, on the one hand, a  morphism between two such pairs needs to bear the right relation between the volume forms and the Jacobian in question, and, on the other hand, should reflect the fact that removing a measurable subset of dimension less than $k$ does not  bring about a change of the integral. The category equipped with such morphisms is denoted by $\mVF[k]$ (see Definition~\ref{defn:VF:mu} for a precise formulation). The coproduct of $\mVF[k]$, $k \geq 0$, is denoted by $\mVF[*]$; observe that gradation by ambient dimensions is a necessity in the presence of volume forms. The category $\mRV[*]$ is obtained in a similar fashion.  As in \cite{Yin:tcon:I}, one of the main results of this paper is  a canonical isomorphism
\[
 \int : \ggk \mVF[*] \fun \ggk  \mRV[*] /  (\bm P),
\]
which respects the gradation since the principal ideal $(\bm P)$ is now homogenous.

We may coarsen the data and consider $\Gamma$-volume forms instead of $\RV$-volume forms, that is, definable functions into the value group $\Gamma$. Taking values in $\Gamma$ corresponds to the identification of a volume form $\omega$ with another one $f\omega$ when $\vv \circ f = 0$. Consequently, for instance, any definable bijection between two sets in the residue field, such as a dilation, is $\Gamma$-measure preserving. The resulting categories are denoted by $\mgVF[*]$, $\mgRV[*]$ and, in complete analogy with the integral above, there is a canonical isomorphism between the two Grothendieck rings:
\[
\int : \ggk \mgVF[*] \fun \ggk  \mgRV[*] /  (\bm P).
\]
In a sense, this integral with  $\Gamma$-volume forms is a halfway point between the integral with $\RV$-volume forms and the universal additive invariant above.

One advantage of $\Gamma$-volume forms, though, is that  the structure of the Grothendieck ring $\ggk \mgRV[*]$ can be significantly elucidated. To wit, it can be expressed as a tensor product of two other Grothendieck rings $\ggk \mgRES[*]$ and $\ggk \mG[*]$, that is, there is an isomorphism of graded rings
\[
\mgD: \ggk \mgRES[*] \otimes_{\ggk \mG^{c}[*]} \ggk \mG[*] \fun \ggk \mgRV[*],
\]
where $\mgRES[*]$ is essentially the category of definable sets in the residue field $\K$ (as a \T-model) with  $\Gamma$-volume forms and $\mG[*]$ is essentially the category of definable sets in the value group $\Gamma$ (as an \omin-minimal group) with volume forms (there is only one type of volume forms for such sets), both are graded by ambient dimensions, and $\mG^{c}[*]$ is the full subcategory of $\mG[*]$ of finite objects, whose Grothendieck ring admits a natural embedding into  $\ggk \mgRES[*]$ as well. This isomorphism, when combined with the integral $\int$ and the two Euler characteristics in \omin-minimal groups (one is a truncated version of the other), yields two homomorphisms of graded rings:
\[
\Xint{\textup{e}}^g, \Xint{\textup{e}}^b: \ggk \mgVF[*] \to^{\int} \ggk \mgRV[*] / (\bm P) \two^{\mgE_{g}}_{\mgE_{b}} \F_3\Gamma[X].
\]
Here $\F_3\Gamma$ is the group ring of $\Gamma(\mdl S)$ over the finite field $\F_3$.

Thus it appears that the construction of these two homomorphisms is much ado about nothing, as essentially only information about the ``sign'' survives. The problem is that the Grothendieck ring $\ggk \mgRES[*]$ is much simpler than its counterpart in, say, \cite[Theorem~10.11]{hrushovski:kazhdan:integration:vf} and cannot afford the further reduction demanded by the vanishing of $\bm P$. To remedy this, we can trim down the categories $\mgRV[*]$, $\mG[*]$ as follows.

A set is bounded if, after applying the maps $\vv$ and $\vrv$ in the corresponding coordinates, it is contained in a box of the form $[\gamma, \infty]^n$, and doubly bounded if the box is of the form $[- \gamma, \gamma]^n$. The full subcategories of $\mgRV[*]$, $\mG[*]$ of doubly bounded objects are denoted by $\mgRV^{\db}[*]$, $\mG^{\db}[*]$. The corresponding restriction of $\mgD$ is indeed an isomorphism
\[
\mgD^{\db}: \gsk \mgRES[*] \otimes_{\gsk \mG^{c}[*]} \gsk \mG^{\db}[*] \fun \gsk \mgRV^{\db}[*].
\]
A set $A \sub \VF^k$ is proper invariant if it is bounded and its characteristic function is locally constant. The subcategory $\mgVF^{\diamond}[k]$ of $\mgVF[k]$ consists of the proper invariant objects and certain morphisms between them; the defining conditions for these morphisms are too involved to afford a summary description here, see Definitions~\ref{defn:binv}, \ref{rela:unary}, and \ref{defn:diamond} for detail. Again, there is a canonical isomorphism $\int^\diamond$  and it, together with $\mgD^{\db}$, induces a graded ring homomorphism:
\[
 \Xint{\textup{e}}^\diamond : \ggk \mgVF^\diamond[*] \to^{\int^\diamond} \ggk \mgRV^{\db}[*] / (\bm P_\Gamma) \to^{\mgE^{\db}} \Z[X].
\]
There is only one such homomorphism because the two Euler characteristics in \omin-minimal groups agree on doubly bounded sets. The homogenous ideal $(\bm P_\Gamma)$ is not principal; it is generated by the elements $\bm P_\gamma \in \ggk \mgRV^{\db}[1]$, one for each $\gamma \in \Gamma^+(\mdl S)$, defined as follows. Let $\MM_\gamma$ be the open  disc around $0$ with radius $\gamma$, $t_\gamma \in \vrv^{-1}(\gamma)$ a definable element, and $\RV^{\circ \circ}_\gamma$ denote the set $\rv(\MM_{\gamma} \mi \{0\})$. Then $\bm P_\gamma$ is the element
\[
[\RV^{\circ \circ}_0 \mi \RV^{\circ \circ}_{\gamma}] + [\{t_\gamma\}] - [\{1\}]
\]
with the constant volume form $0$, and it does not depend on the choice of $t_\gamma$.

The reader should keep in mind that the construction of the isomorphism $\int^\diamond$ encompasses significant technical complications, which constitute the bulk of the work below, but it is still possible to run it with the other ones in parallel so that we do not have to do everything twice with slight variations. This is the reason why some of the definitions are not as concise as they could be and many statements and proofs carry two cases.

\section{Definability in $T$-convex valued fields}

\subsection{Preliminaries}
We begin by reviewing a few fundamental facts about $T$-convex valued fields. Some of the definitions and a portion of the notation  from \cite{Yin:tcon:I} that will be used throughout are also reiterated for the reader's convenience.

\begin{nota}[Coordinate projections]\label{indexing}
For each $n \in \N$, let $[n]$ abbreviate the set $\{1, \ldots, n\}$. For $E \sub [n]$, we write $\pr_E(A)$, or even $A_E$ when there is no danger of confusion, for the projection of $A$ into the coordinates contained in $E$. It is often more convenient to use simple standard descriptions as subscripts. For example, if $E$ is a singleton $\{i\}$ then we shall always write $E$ as $i$ and $\tilde E \coloneqq [n] \mi E$ as $\tilde i$; similarly, if $E = [i]$, $\{k:  i \leq k \leq j\}$, $\{k:  i < k < j\}$, $\{\text{all the coordinates in the sort $S$}\}$, etc., then we may write $\pr_{\leq i}$, $\pr_{[i, j]}$, $A_{(i, j)}$, $A_{S}$, etc.; in particular, we shall frequently write $A_{\VF}$ and $A_{\RV}$ for the projections of $A$ into the $\VF$-sort and $\RV$-sort coordinates.

Unless otherwise specified, by writing $a \in A$ we shall mean that $a$ is a finite tuple of elements (or ``points'') of $A$, whose length is not always explicitly indicated. If $a = (a_1, \ldots, a_n)$ then for all $1 \leq i < j \leq n$, following the notational scheme above, $a_i$, $a_{\tilde i}$, $a_{\leq i}$, $a_{[i, j]}$, $a_{\VF}$, etc., are shorthand for the corresponding subtuples of $a$.

The sets $\{t\} \times A$, $\{t\} \cup A$, $A \mi \{t\}$, etc., shall be simply written as $t \times A$, $t \cup A$, $A \mi t$, etc., when it is clearly understood that $t$ is an element and hence must be interpreted as a singleton in these expressions.

For $a \in A_{\tilde E}$, the fiber $\{b : ( b, a) \in A \} \sub A_E$ over $a$ is often denoted by $A_a$. Note that, in the discussion below, the distinction between the two sets $A_a$ and $A_a \times a$ is usually immaterial and hence they may and  shall be tacitly identified. In particular, given a function $f : A \fun B$ and $b \in B$, the pullback $f^{-1}(b)$ is sometimes written as $A_b$ as well. This is a special case since functions are routinely identified with their graphs. This notational scheme is especially useful when the function $f$ has been clearly understood in the context and hence there is no need to spell it out all the time.
\end{nota}

The definitions of the language $\lan{T}{RV}{}$ and the $\lan{T}{RV}{}$-theory $\TCVF$ will not be repeated here, see \cite[\S~2.2]{Yin:tcon:I} for detail.

\begin{rem}\label{signed:Gam}
Recall from \cite[Definition~2.7]{Yin:tcon:I} that, although the behavior of the valuation map $\abval$ in the traditional sense is coded in $\TCVF$, we use the \emph{signed valuation map}
\[
\vv : \VF \to^{\rv} \RVV \to^{\vrv} \GAA,
\]
since it is more natural in the present setting; here the subscripts indicate the inclusion of the ``middle element'' $0$. The \emph{signed value group} $\Gamma_0$, as a quotient of $\RV_0$, is multiplicatively written and the induced ordering $\leq$ on it no longer needs to be inverted. It is also tempting to use this ordering $\leq$ on the \emph{value group} $\absG_{\infty}$ instead of its inverse, but this makes citing results in the literature a bit awkward. We shall actually abuse the notation and denote the ordering $\leq^{-1}$ on $\abs{\Gamma}_{\infty}$ also by $\leq$; this should not cause confusion since the ordering on $\Gamma_0$ will rarely be used (we will indicate so explicitly when it is used).

The axioms in \cite[Definition~2.7]{Yin:tcon:I} guarantee that the ordered abelian group ${\GAA} /{\pm 1}$ (here $\vv(\pm 1)$ is just written as $\pm 1$) with the bottom element $0$ is isomorphic to $\abs{\Gamma}_{\infty}$ if either one of the orderings is inverted. So $\abval$ may be thought of as the composition $\vv/{\pm 1} : \VF \fun {\GAA} /{\pm 1}$.

The \emph{sign function} $\sgn : \Gamma^n \fun \pm$ is given by $\gamma \efun +$ if $\Pi \gamma \in \Gamma^+$ and $\gamma \efun -$ if $\Pi \gamma \in \Gamma^-$, where $\Pi (\gamma_1, \ldots, \gamma_n) = \gamma_1 \cdot  \ldots  \cdot \gamma_n$. We may identify $\pm$ with $\pm 1\in \Gamma$.
\end{rem}

\begin{thm}\label{theos:qe}
The theory $\TCVF$ admits quantifier elimination and is weakly \omin-minimal.
\end{thm}
\begin{proof}
See \cite[Theorem~2.16, Corollary~2.18]{Yin:tcon:I} and \cite[Corollary~3.14]{DriesLew95}.
\end{proof}

Recall from \S~\ref{section:intro} that $\mmdl$ is a sufficiently saturated $\TCVF$-model. Occasionally, when we work in the \LT-reduct $\mdl R$ of $\mmdl$, or just wish to emphasize that a set is definable in $\mdl R$ instead of $\mmdl$, the symbol ``$\lan{T}{}{}$'' or ``$T$'' will be inserted into the corresponding places in the terminology.

\begin{conv}\label{topterm}
Since, apart from the binary relation $\leq$, the language $\lan{T}{}{}$ only has function symbols (see \cite[Remark~2.3]{Yin:tcon:I}), we may and shall assume that, in any $\lan{T}{RV}{}$-formula, every \LT-term occurs in the scope of an instance of the function symbol $\rv$. For example, if $f(x)$, $g(x)$ are \LT-terms then the formula $f(x) < g(x)$ is equivalent to $\rv(f(x) - g(x)) < 0$. The \LT-term $f(x)$ in $\rv(f(x))$ is referred to as a \emph{top \LT-term}.
\end{conv}

The default topology on  $\VF$ is of course the order topology and the default topology on $\VF^n$ is the corresponding product topology, similarly for $\RV$, $\Gamma$, etc.

\begin{rem}
We refer to \cite{mac:mar:ste:weako} for the general theory of weak \omin-minimality. Its dimension theory  gives rise to the notion of the \emph{$\VF$-dimension} of a definable set $A$,  denoted by $\dim_{\VF}(A)$, which is the largest natural number $k$ such that, possibly after re-indexing of the $\VF$-coordinates, $\pr_{\leq k}(A_t)$ has nonempty interior for some $t \in A_{\RV}$. There are alternative characterizations of the  operator $\dim_{\VF}$, see \cite[Lemma~2.27]{Yin:tcon:I}.

Weakly \omin-minimal dimension is also applicable in the $\RV$-sort; we call it  the \emph{$\RV$-dimension} and the corresponding operator is denoted by $\dim_{\RV}$, see \cite[Remark~2.28]{Yin:tcon:I}. Similarly, we use \omin-minimal dimension in the $\Gamma$-sort (the value group $\Gamma$ is thought of informally as a definable sort, similarly for the residue field $\K$, etc.) and call it the \emph{$\Gamma$-dimension}, the corresponding operator is denoted by $\dim_{\Gamma}$.
\end{rem}

\begin{lem}[{\cite[Lemma~2.29]{Yin:tcon:I}}]\label{dim:cut:gam}
Suppose that $U \sub \RV^n$ is a definable set with $\dim_{\RV}(U) = k$. Then $\dim_{\RV}(U_{\gamma}) = k$ for some $\gamma \in \vrv(U)$.
\end{lem}

\begin{ter}[Sets and subsets]\label{nota:sub}
By a definable set in $\VF$ we mean a definable subset in $\VF$, by which we just mean a subset of $\VF^n$ for some $n$, unless indicated otherwise; similarly for other sorts or structures  in place of $\VF$ that have been clearly understood in the context.
\end{ter}

\begin{rem}\label{rem:hyp}
By \cite[Lemma~3.19]{Yin:tcon:I}, the substructure $\mdl S$ is definably closed. If $\mdl S$ is $\VF$-generated --- that is, the map $\rv$ is surjective in $\mdl S$ --- and $\Gamma(\mdl S)$ is nontrivial then $\mdl S$ is indeed an elementary substructure and hence every definable set contains a definable point (see \cite[Hypothesis~5.11]{Yin:tcon:I}).
\end{rem}

Note that  $\mdl S$ is actually regarded as a part of the language now and hence, contrary to the usual convention in the model-theoretic literature, ``$\0$-definable'' or ``definable'' only means ``$\mdl S$-definable'' instead of ``parametrically definable'' if no other qualifications are given. To simplify the notation, we shall not mention $\mdl S$ and its extensions explicitly unless necessary. For example, the definable closure operator $\dcl_{\mdl S}$, etc., will simply be written as $\dcl$, etc.

\begin{rem}\label{pillars}
Two pillars at the very foundation of our construction are \cite[Theorems~A,~B]{Dries:tcon:97}. The former theorem says that the structure of definable sets in the $\K$-sort is precisely that given by the theory $T$. The latter theorem says that the structure of definable sets in the $\abs \Gamma$-sort is precisely that given by the \omin-minimal theory of nontrivially ordered vector spaces over the field of exponents of the theory $T$ (denoted by $\KKK$); thus, modulo the sign, this is also the structure of definable sets in the $\Gamma$-sort. In particular, every definable function in the $\Gamma$-sort is definably piecewise $\KKK$-linear modulo the sign (see Lemma~\ref{gam:pulback:mono} below for a more precise statement, and for simplicity we shall henceforth omit the modifier ``modulo the sign'').
\end{rem}

\begin{lem}[{\cite[Proposition~5.8]{Dries:tcon:97}}]\label{gk:ortho}
If $f : \Gamma \fun \K$ is a definable function then $f(\K)$ is finite. Similarly, if $g : \K \fun \Gamma$ is a definable function then $g(\Gamma)$ is finite.
\end{lem}

\begin{rem}[\omin-minimal sets in $\RV$]\label{omin:res}
The theory of \omin-minimality, in particular its terminologies and notions, may be applied to a set $U \sub \RV^n$ if $\vrv(U)$ is a singleton or, more generally, is finite. For example, we shall say that $U$ is a \emph{cell} if the multiplicative translation $U / u \sub (\K^+)^n$ of $U$ by some $u \in U$ is an \omin-minimal cell (see \cite[\S~3]{dries:1998}); this definition does not depend on the choice of $u$. Similarly, the \emph{\omin-minimal Euler characteristic} $\chi(U)$ of such a set $U$ is the \omin-minimal Euler characteristic of $U / u$ (see \cite[\S~4.2]{dries:1998}). This definition may be extended to disjoint unions of finitely many (not necessarily disjoint) sets $U_i \sub \RV^n \times \Gamma^m$ such that each $\vrv(U_i)$ is finite.
\end{rem}

Taking disjoint union of finitely many definable sets of course will introduce extra bookkeeping coordinates, but we shall suppress this in notation.

\begin{thm}[{\cite[\S~8.2.11]{dries:1998}}]\label{groth:omin}
Let $U$, $V$ be definable sets in $\RV$ such that $\vrv(U)$, $\vrv(V)$ are both finite. Then there is a definable bijection between $U$ and $V$ if and only if
$\dim_{\RV}(U) = \dim_{\RV}(V)$ and $\chi(U) = \chi(V)$.
\end{thm}

\begin{nota}\label{gamma:what}
We shall write $\gamma$ as $\gamma^\sharp$ when we want to emphasize that it is the set $\vrv^{-1}(\gamma)$ or $\abvrv^{-1}(\gamma)$ in $\mmdl$ that is being considered. More generally, if $I$ is a set in $\Gamma$ or $\absG$ then we write $I^\sharp = \bigcup\{\gamma^\sharp: \gamma \in I\}$. Similarly, if $U$ is a set in $\RV$ then $U^\sharp$ stands for the set $\bigcup\{\rv^{-1}(t): t \in U\}$, and hence if $I$ is a set in $\Gamma$ then $I^{\dsh}$ is a set in $\VF$.
\end{nota}

\begin{defn}[Valuative discs]\label{defn:disc}
A set $\gb \sub \VF$ is an \emph{open disc} if there is a $\gamma \in \absG_{\infty}$ and a $b \in \gb$ such that $a \in \gb$ if and
only if $\abval(a - b) > \gamma$; it is a \emph{closed disc} if $a \in \gb$ if and only if $\abval(a - b) \geq \gamma$. The point $b$ is a \emph{center} of $\gb$. The value $\gamma$ is the \emph{valuative radius} or simply the \emph{radius} of $\gb$, which is denoted by $\rad (\gb)$. A set of the form $t^\sharp$, where $t \in \RV$, is called an \emph{$\RV$-disc}.

A closed disc with a maximal open subdisc removed is called a \emph{thin annulus}.

A set $\gp \sub \VF^n \times \RV_0^m$ of the form $\prod_{i \leq n} \gb_i \times t$ is called an (\emph{open, closed, $\RV$-}) \emph{polydisc} if each $\gb_i$ is an (open, closed, $\RV$-) disc. The \emph{radius} $\rad(\gp)$ of $\gp$ is the tuple $(\rad(\gb_1), \ldots, \rad(\gb_n))$. The open and the closed polydiscs centered at a point $(a,t) \in \VF^n \times \RV_0^m$ with radius $\gamma \in \absG_\infty^n$ are denoted by $\go((a,t), \gamma)$, $\gc((a,t), \gamma)$, respectively. The subdiscs $\go(0, \gamma)$, $\gc(0, \gamma)$ of $\VF$ are more suggestively denoted by $\MM_\gamma$, $\OO_\gamma$, respectively.


The \emph{$\RV$-hull} of a set $A$, denoted by $\RVH(A)$, is the union of all the $\RV$-polydiscs whose intersections with $A$ are nonempty. If $A$ equals $\RVH(A)$ then $A$ is called an \emph{$\RV$-pullback}.
\end{defn}

\begin{defn}[$\vv$-intervals]
Let $\ga$, $\gb$ be discs, not necessarily disjoint. The subset $\{\ga < x < \gb\}$ of $\VF$, if it is not empty, is called an \emph{open $\vv$-interval} and is denoted by $(\ga, \gb)$, whereas the subset
\[
\{a \in \VF : \ex{x \in \ga, y \in \gb} ( x \leq a \leq y) \},
\]
if it is not empty, is called a \emph{closed $\vv$-interval} and is denoted by $[\ga, \gb]$. The other $\vv$-intervals $[\ga, \gb)$, $(-\infty, \gb]$, etc., are defined in the obvious way, where $(-\infty, \gb]$ is a closed (or half-closed) $\vv$-interval that is unbounded from below.

Let $A$ be such a $\vv$-interval. The discs $\ga$, $\gb$ are called the \emph{end-discs} of $A$. If $\ga$, $\gb$ are both points in $\VF$ then of course we just say that $A$ is an interval and if $\ga$, $\gb$ are both $\RV$-discs then we say that $A$ is an $\RV$-interval. If $A$ is of the form $(\ga, \gb]$ or $[\gb, \ga)$, where $\ga$ is an open disc and $\gb$ is the smallest closed disc containing $\ga$, then $A$ is called a \emph{half thin annulus} and the \emph{radius} of $A$ is $\rad(\gb)$.

Two $\vv$-intervals are \emph{disconnected} if their union is not a $\vv$-interval.
\end{defn}

\begin{rem}[Holly normal form]\label{rem:HNF}
By the valuation property (see \cite[Proposition~9.2]{DriesSpei:2000} and \cite{tyne}) and \cite[Proposition~7.6]{Dries:tcon:97}, we have an important tool called \emph{Holly normal form}, see \cite[Theorem~4.8]{holly:can:1995} (henceforth abbreviated as HNF); that is, every definable subset of $\VF$ is a unique union of finitely many definable pairwise disconnected $\vv$-intervals. This is a generalization of the \omin-minimal condition.
\end{rem}

\subsection{Further auxiliaries}
We gather here some technicalities delineating the landscape of definite sets in $\mmdl$  that will be directly cited below, most of which are  from \cite{Yin:tcon:I}.

A definable function $f$ is \emph{quasi-\LT-definable} if it is a restriction of an \LT-definable function (with parameters in $\VF(\mdl S)$, of course).

\begin{lem}[{\cite[Lemma~3.3]{Yin:tcon:I}}]\label{fun:suba:fun}
Every definable function $f : \VF^n \fun \VF$ is piecewise quasi-\LT-definable, that is, there are a definable finite partition $(A_i)_i$ of $\VF^n$ and \LT-definable functions $f_i: \VF^n \fun \VF$ such that $f \rest A_i = f_i \rest A_i$ for all $i$.
\end{lem}

\begin{cor}[Monotonicity]\label{mono}
Let $A \sub \VF$ and $f : A \fun \VF$ be a definable function. Then there is a definable finite partition of $A$ into $\vv$-intervals $A_i$ such that every $f \rest A_i$ is quasi-\LT-definable, continuous, and monotone (constant or strictly increasing or strictly decreasing). Consequently, each $f(A_i)$ is a $\vv$-interval.
\end{cor}
\begin{proof}
This follows from Lemma~\ref{fun:suba:fun}, \omin-minimal monotonicity, and HNF.
\end{proof}

\begin{lem}[{\cite[Lemma~3.6]{Yin:tcon:I}}]\label{RV:no:point}
Given a tuple $t$ in $\RV$, if $a \in \VF$ is $t$-definable then $a$ is actually definable. Similarly, for any tuple $\gamma$ in $\Gamma$, if $t \in \RV$ is $\gamma$-definable then $t$ is definable.
\end{lem}

\begin{lem}[{\cite[Lemma~3.20]{Yin:tcon:I}}]\label{clo:disc:bary}
Every definable closed disc $\gb$ contains a definable point.
\end{lem}

\begin{lem}[{\cite[Corollary~3.21]{Yin:tcon:I}}]\label{open:disc:def:point}
Let $\ga \sub \VF$ be a disc and $A$ a definable subset of $\VF$. If $\ga \cap A$ is a nonempty proper subset of $\ga$ then $\ga$ contains a definable point.
\end{lem}

\begin{ter}\label{rvfiber}
A \emph{$\VF$-fiber} of a set $A$ is a set of the form $A_t$, where $t \in A_{\RV}$ (recall Notation~\ref{indexing}); in particular,  a $\VF$-fiber of a function $f : A \fun B$ is a set of the form $f_t$ for some  $t \in f_{\RV}$ (here $f$ also stands for its own graph) which is indeed (the graph of) a function between two sets in $\VF$.  We say that $A$ is open if every one of its $\VF$-fibers is, $f$ is continuous if every one of its $\VF$-fibers is, and so on.
\end{ter}

\begin{defn}[Disc-to-disc]\label{defn:dtdp}
Let $f : A \fun B$ be a bijection between two sets $A$ and $B$, each with exactly one $\VF$-coordinate. We say that $f$ is \emph{concentric} if, for every $\VF$-fiber $f_{t}$ of $f$ and all open polydiscs $\ga \sub \dom(f_{t})$, $f_{t}(\ga)$ is also an open polydisc; if both $f$ and $f^{-1}$ are concentric then $f$ has the \emph{disc-to-disc property} (henceforth abbreviated as ``dtdp'').
\end{defn}

\begin{lem}[{\cite[Lemma~3.24]{Yin:tcon:I}}]\label{open:pro}
Let $f : A \fun B$ be a definable bijection between two sets $A$ and $B$, each with exactly one $\VF$-coordinate. Then there is a definable finite partition $(A_i)_i$ of $A$ such that every restriction $f \rest A_i$ has dtdp.
\end{lem}

\begin{defn}
Let $A$ be a subset of $\VF^n$. The \emph{$\RV$-boundary} of $A$, denoted by $\partial_{\RV}A$, is the definable subset of $\rv(A)$ such that $t \in \partial_{\RV} A$ if and only if $t^\sharp \cap A$ is a proper nonempty subset of $t^\sharp$. The definable set $\rv(A) \mi \partial_{\RV}A$, denoted by $\ito_{\RV}(A)$, is called the \emph{$\RV$-interior} of $A$.
\end{defn}

Obviously, $A \sub \VF^n$ is an $\RV$-pullback if and only if $\partial_{\RV} A$ is empty. Note that $\partial_{\RV}A$ is in general different from the topological boundary $\partial(\rv(A))$ of $\rv(A)$ in $\RV^n$ and neither one of them includes the other.

\begin{lem}[{\cite[Lemma~3.26]{Yin:tcon:I}}]\label{RV:bou:dim}
If $A$ is a definable subset of $\VF^n$ then $\dim_{\RV}(\partial_{\RV} A) < n$.
\end{lem}

\begin{defn}[Contractions]\label{defn:corr:cont}
A function $f : A \fun B$ is \emph{$\rv$-contractible} if there is a (necessarily unique) function $f_{\downarrow} : \rv(A) \fun \rv(B)$, called the \emph{$\rv$-contraction} of $f$, such that
\[
(\rv \rest B) \circ f = f_{\downarrow} \circ (\rv \rest A).
\]
Similarly, it is \emph{$\res$-contractible} (respectively, \emph{$\vv$-contractible}) if the same holds in terms of $\res$ (respectively, $\vv$ or $\vrv$, depending on the coordinates) instead of $\rv$.
\end{defn}

 More general forms of contractions will be introduced in \S~\ref{defn:binv}.

\begin{lem}[{\cite[Lemma~3.28]{Yin:tcon:I}}]\label{fn:alm:cont}
For every definable function $f : \VF^n \fun \VF$ there is a definable set $U \sub \RV^n$ with $\dim_{\RV}(U) < n$ such that $f \rest (\VF^n \mi  U^\sharp)$ is $\rv$-contractible.
\end{lem}

\begin{lem}[{\cite[Lemma~4.12]{Yin:tcon:I}}]\label{gam:pulback:mono}
Let $I, J \sub \Gamma^k$ be definable sets. Then every definable bijection $g : I \fun J$ is  a piecewise $\mgl_k(\KKK) \times \Z_2$-transformation (with definable constant terms). Consequently, $g$ is a $\vrv$-contraction.
\end{lem}

\begin{ter}\label{bounded}
We say that a set $I \sub \absG_{\infty}^n$ is \emph{$\gamma$-bounded}, where $\gamma \in \absG$, if it is contained in the box $[\gamma, \infty]^n$, and is \emph{doubly $\gamma$-bounded} if it is contained in the box $[-\gamma, \gamma]^n$. More generally, let $A$ be a subset of $\VF^n \times \RV_0^m \times \Gamma_0^l$ and
\[
\abs A_{\Gamma} = \{(\abval(a), \abvrv(t), \abs \gamma) : (a, t, \gamma) \in A\} \sub \abs\Gamma_{\infty}^{n+m+l};
\]
then we say that $A$ is \emph{$\gamma$-bounded} if $\abs A_{\Gamma}$ is, and so on.
\end{ter}

\begin{lem}\label{db:to:db}
Let $U \sub \RV^k$ be a doubly bounded set and $f: U \fun \RV$ a definable function. Then $f(U)$ is also doubly bounded. The same holds if the codomain of $f$ is $\Gamma$.
\end{lem}
\begin{proof}
By induction on $k$, both claims are immediately reduced to showing that if $k=1$  and $g : U \fun \Gamma$ is a definable function then $g(U)$ is doubly bounded. Then, by weak \omin-minimality and Lemma~\ref{gk:ortho}, we may assume that $g$  $\vrv$-contracts to a function $g_{\downarrow} : \vrv(U) \fun \Gamma$. By Remark~\ref{pillars}, $g_{\downarrow}$ is piecewise $\KKK$-linear  and hence its range must be doubly bounded.
\end{proof}

As usual in model theory, some properties concerning definable sets can be more conveniently formulated in the expansion $\mdl R_{\rv}^{\textup{eq}}$ of $\mmdl$ by all definable sorts. For our purpose here, a much simpler expansion $\mdl R_{\rv}^{\bullet}$ suffices, see \cite[Notation~3.24]{Yin:tcon:I} for details.  Several lemmas in  \cite{Yin:tcon:I} are stated for definable sets in $\mdl R_{\rv}^{\bullet}$ instead of $\mmdl$, among which only the following one, a more general version of Lemma~\ref{RV:no:point}, is directly cited:

\begin{lem}[{\cite[Lemma~3.8]{Yin:tcon:I}}]\label{ima:par:red}
For $\code \ga  \in \DC$, if $a \in \VF$ is $\code \ga$-definable then $a$ is definable.
\end{lem}

Here $\DC$ is the definable sort of  discs and, for a disc $\ga \sub \VF$, the corresponding imaginary element in $\DC$ is denoted by $\code{\ga}$, which may be heuristically thought of as the ``name'' of $\ga$.

\begin{lem}\label{invar:db}
Let $A \sub \VF^n \times \RV^m$ be a definable set such that $A_{\VF}$ is bounded and $A_{\RV}$ is doubly bounded. Let $f : A \fun \absG$ be a definable function  that is constant on $A \cap \ga$ for all open polydiscs $\ga$ of radius $\gamma \in \absG$. Then $f(A)$ is doubly bounded.
\end{lem}
\begin{proof}
By Lemma~\ref{db:to:db} and \omin-minimality in the $\Gamma$-sort, it is enough to consider the case $A \sub \VF^n$. It is easy to see that, by induction on $n$, this may be further reduced to the case $n=1$. By the assumption, we may assume that, for every open disc $\ga$ of radius $\gamma$, either $\ga \cap A = \0$ or $\ga \sub A$, and moreover $A \cap \MM_\gamma = \0$, that is, $A$ is doubly bounded.

Suppose for contradiction that $f(A)$ contains a definable interval $I$ of the form $(\beta, \infty)$. By HNF, for  each $\alpha \in I$, if $\delta^{\dsh} \cap f^{-1}(\alpha)$ is a nonempty proper subset of $\delta^{\dsh}$ for some  $\delta \in \Gamma$ then at least one of the end-discs of $f^{-1}(\alpha)$, say $\gb_\alpha$, is contained in $\delta^{\dsh}$. Therefore, shrinking $I$ if necessary, we can construct a definable function $g$ on $I$ such that
\begin{itemize}
  \item either $g(\alpha)  = \code{\gb_\alpha}$ for all $\alpha \in I$ or
  \item $g(\alpha) \in \Gamma$ and $g(\alpha)^{\dsh} \sub f^{-1}(\alpha)$ for all $\alpha \in I$.
\end{itemize}
In the latter case, we may assume that $g$ is $\KKK$-linear, but then its image cannot be doubly bounded since $g$ is injective, contradicting the choice of $I$. The former case requires further argument.

Consider the function $\bar g$ on $I$ given by $\alpha \efun \rad(\gb_\alpha)$. Since $\rad(\gb_\alpha) \leq \gamma$ for all $\alpha \in I$ and $A$ is bounded,  $\bar g(I)$ must be doubly bounded. By \omin-minimality in the $\Gamma$-sort, shrinking $I$ if necessary, we may assume that $\bar g$ is constant, say $\bar g(I) = \delta$, and either every $\gb_\alpha$ is a closed disc or every $\gb_\alpha$ is an open disc. In the former case, by Lemmas~\ref{clo:disc:bary} and \ref{RV:no:point}, $g(I)$ must be finite; a moment of reflection shows that this is impossible since either eventually every $\gb_\alpha$ is an end-disc at a closed end or eventually every $\gb_\alpha$ is an end-disc at an open end. In the latter case, for the same reason, we may assume that $g(I)$ is infinite. But then there would be a restriction of $g$ whose image form a closed disc $\gb$ of radius $\delta$, which is easily seen to contradict Lemma~\ref{gk:ortho}.
\end{proof}

\begin{defn}
Let $p : A \fun \absG$ be a definable function. We say that $p$ is an \emph{$\go$-partition} of $A$ if, for every $a \in A$, the function $p$ is constant on $\go(a, p(a)) \cap A$.
\end{defn}

\begin{lem}\label{vol:par:bounded}
Let $p$ be an $\go$-partition of $A$. Suppose that $A_t$ is closed and bounded for every $t \in A_{\RV}$ and $A_{\RV}$ is doubly bounded. Then $p(A)$ is doubly bounded.
\end{lem}
\begin{proof}
We first handle the case that $A$ has no $\RV$-coordinates. To begin with, observe that parameters have no bearing on whether  $p(A)$ is doubly bounded or not, and hence, without loss of generality, we may assume that the substructure $\mdl S$ is $\VF$-generated. Since $A$ is bounded, $p(A)$ must be bounded too. Suppose for contradiction that $p(A)$ is not doubly bounded. For $\gamma \in \absG$, let $A_{\gamma} = \{a \in A : p(a) > \gamma \}$. For all $c \in \MM \mi \{0\}$, since the \T-model $\mdl S \la c \ra_T$ generated by $c$ may be expanded to a (unique) elementary $\TCVF$-submodel, we see that $A_{\abval(c)}$ contains a point in $\mdl S \la c \ra_T$. So, by compactness, there is a definable function $f : \MM \mi \{0\} \fun A$ such that $p(f(c)) > \abval(c)$ for every $c \in \MM \mi \{0\}$. By Lemma~\ref{fun:suba:fun} and \omin-minimal monotonicity,
\[
 a^+ \coloneqq \lim_{x \rightarrow 0^+} f(x) \dand a^- \coloneqq \lim_{x \rightarrow 0^-} f(x)
\]
both exist and, since $A$ is closed, they are contained in $A$. Thus, there is a $c \in \MM^+ \mi \{0\}$ such that $\abval(c) > p(a^+)$ and $f(c) \in \go(a^+, p(a^+))$. Since $p$ is an $\go$-partition, this implies $p(f(c)) = p(a^+)$, which is a contradiction.

The general case follows from the case above and Lemma~\ref{db:to:db}.
\end{proof}

This lemma, in its various formulations, is crucial for the good behavior of motivic Fourier transform (see \cite[\S~11]{hrushovski:kazhdan:integration:vf} and \cite{yin:hk:part:3}). For essentially the same reason, the specialized version of the main construction for proper invariant sets below depends heavily on  it.

\subsection{Of continuity and differentiability}

For the next two lemmas, as in the proof of Lemma~\ref{vol:par:bounded}, since parameters do not affect whether the conclusions hold or not, we may assume that  every definable set contains a definable point (see Remark~\ref{rem:hyp}). In fact, for convenience, we shall enlarge  $\mdl S$ so that all the relevant parametrically definable sets become definable.

\begin{lem}\label{cb:to:cb}
Let $f : A \fun \VF^m$ be a definable continuous function. Suppose that $A \sub \VF^n$ is closed and bounded. Then $f(A)$ is closed and bounded.
\end{lem}
\begin{proof}
Suppose for contradiction that $f(A)$ is not bounded. Then, by compactness, there is a definable function $g : \VF^+ \fun A$ such that $\abs{f(g(x))} > x$ for every $x \in \VF^+$. Since $A$ is closed and bounded, applying monotonicity (Corollary~\ref{mono}) to each coordinate of $g$, we see that $\lim_{x \limplies \infty} g(x)$ exists and belongs to $A$. But this implies that $\lim_{x \limplies \infty} f(g(x))$ also exists and belongs to $f(A)$, which is impossible.

The argument for closedness is similar: Otherwise, we would have a point $b \in \VF^m \mi f(A)$ and a definable function $g : (0,1) \fun A$ such that $\lim_{x \limplies 0} f(g(x)) = b$, which is impossible since $\lim_{x \limplies 0} g(x)$ exists and belongs to $A$.
\end{proof}

\begin{cor}\label{conti:homeo}
If the function $f$ in the above lemma is injective then  it is a homeomorphism from $A$ onto $f(A)$.
\end{cor}

This corollary is just a more general version of \cite[Corollary~6.1.12]{dries:1998}. The following lemma is analogous to \cite[Corollary~6.2.4]{dries:1998}.

\begin{lem}\label{fiber:conti}
Let $A \sub \VF^n$ and $f : A \fun \VF$ be a definable function. Suppose that, for every $a \in A_1 \coloneqq \pr_1(A)$, $A_a \sub \VF^{n-1}$ is open and the induced function $f_a : A_a \fun \VF$ is continuous. Then there is a definable finite set $B \sub A_1$ such that $f$ is continuous away from $\bigcup_{a \in B} a \times A_a$.
\end{lem}
\begin{proof}
Note that the case $n=1$ is meaningful and it follows from monotonicity. So assume $n > 1$. Let $A' \sub A$ be the definable set of points at which  $f$ fails to be continuous. Since $\dim_{\VF}(A') < n$, we may assume that $A'' \coloneqq A \mi A'$ is open.  Suppose for contradiction that $\pr_1(A')$ contains an open interval $I$. By monotonicity, shrinking $I$ if necessary,  there is a continuous  \LT-definable function $g: I \fun \VF^{n-1}$ whose graph, also denoted by $g$, is contained in $A'$ and the obvious function on $I$ induced by $f$, $g$ is continuous  as well. For every $a \in g$ there is a continuous \LT-definable function $h_a : (0, 1) \fun A''$ with  $h_a \limplies a$ but $f \circ h_a \limplies b_a \neq f(a)$; here we use the notation in \cite[\S~6.4]{dries:1998} and the argument for the claim is exactly as in the proof of \cite[Lemma~6.4.2]{dries:1998}. By compactness, we may assume that the functions $h_a$ are given uniformly by an \LT-definable function $h : I \times (0, 1) \fun A''$ and there is an $\epsilon \in \VF^+$ such that $\abs{f(a) - b_a} > \epsilon$ for all $a \in g$. By \cite[Lemma~1.5]{DriesLew95}, we may further assume that $h$ is continuous. Then it is not hard to see that, for all $a \in g$, there is a continuous \LT-definable function $h'_a : (0, 1) \fun \ran(h)$ with $h'_a \limplies a$ and $\pr_1(\ran(h'_a)) = a_1 \coloneqq \pr_1(a)$. Since both $f_{a_1}$ and $f \circ h$ are continuous, we have $f_{a_1} \circ h'_a \limplies f(a)$ and $f_{a_1} \circ h'_a \limplies b_a$, which is a contradiction.
\end{proof}

Given a definable set $A$, we say that a property holds \emph{almost everywhere} in $A$ or \emph{for almost every point} in $A$ if it holds away from a definable subset of $A$ of a smaller $\VF$-dimension. This terminology will also be used with respect to other notions of dimension.

By Lemma~\ref{fun:suba:fun} and \omin-minimal differentiability, every definable function $f : \VF^n \fun \VF^m$ is $C^p$ almost everywhere (with respect to the operator $\dim_{\VF}$) for all $p$ (see \cite[\S~7.3]{dries:1998}).
The Jacobian (determinant) of $f$ at $a \in \VF^n$, if it exists, is denoted by $\jcb_{\VF} f (a)$.

There is also an alternative definition of differentiability:

\begin{defn}\label{defn:diff}
For any $a \in \VF^n$, we say that $f$ is \emph{differentiable at $a$} if there is a linear map $\lambda : \VF^n \fun \VF^m$ (of $\VF$-vector spaces) such that, for any $\epsilon \in \absG$, if $b \in \VF^n$ and $\abval(b)$ is sufficiently large then
\[
\abval(f(a + b) - f(a) - \lambda(b)) - \abval(b) > \epsilon.
\]
\end{defn}

It is straightforward to check that if such a linear function $\lambda$ (represented by a matrix with entries in $\VF$) exists then it is unique. Thus this notion of differentiability agrees with the one defined via \omin-minimality.

\begin{lem}
Let $f : A \fun \VF$ be a definable function, where $A \sub \VF^n$ is open. Then $f$ is differentiable if and only if it is $C^1$.
\end{lem}
\begin{proof}
This is  \cite[\S~7.2.6]{dries:1998} if $f$ is \LT-definable, but  the classical proof alluded there, which only depends on the mean value theorem, goes through even if $f$ is not \LT-definable.
\end{proof}

\begin{lem}\label{diff:dtdp}
Let $f : A \fun B$ be a continuous definable bijection between two open subsets of $\VF$. Suppose that every $a \in A$ is contained in an open interval $B_a \sub A$ such that $\vv \circ \ddx f$ is constant on $B_a \mi a$. Then $A$ admits an $\go$-partition such that $f \rest \ga$ has dtdp for every open disc $\ga \sub A$ in question.
\end{lem}

Note that, compared with Lemma~\ref{open:pro}, the assumption of this lemma is stronger and its conclusion weaker, but the point is that we do not need a partition to achieve it. Of course $f$ is $C^1$ almost everywhere, so the assumption is not vacuous only for finitely many points in $A$.

\begin{proof}
It is enough to show that  each $a \in A$ is contained in an open disc $\ga \sub A$ such that $f \rest \ga$ has dtdp. By monotonicity (Corollary~\ref{mono}), we may assume that $B_a$ is such that $f_a \coloneqq f \rest B_a $ is a parametrically \LT-definable monotone function (it is monotone because $f$ is continuous and bijective). Write $\alpha = \vv( \ddx f(B_a \mi a))$. Since $b < \ddx f(B_a \mi a) < b'$ for some $b, b' \in \alpha^{\sharp\sharp}$, a simple estimate argument shows that, for any open disc $\ga \sub B_a$, $f(\ga)$ is an open disc of radius $\rad(\ga) + \abs \alpha$.
\end{proof}

\begin{lem}\label{adjust:C1}
Let $f : A \fun B$ be as in Lemma~\ref{diff:dtdp}. Then there are finitely many definable open discs $\ga_i \sub A$ and a definable $C^1$ bijection $f^* : A \fun B$ with $f = f^*$ on $A \mi \bigcup_i \ga_i$ such that $\vv \circ \ddx f^*(a) = \vv \circ \ddx f(a)$ whenever the righthand side exists and, on each $\ga_i$, $\ddx f^*$ is constant.
\end{lem}
\begin{proof}
There are finitely many definable open discs $\ga_i \sub A$ and a definable point $a_i \in \ga_i$ for each $i$ such that $f$ is $C^1$ on $A \mi \bigcup_i \ga_i$, $\vv \circ \ddx f$ is constant on each $\ga_i \mi a_i$, and $\gb_i = f(\ga_i)$ is an open disc; this last clause is by Lemma~\ref{diff:dtdp}. Let $d_i = \lim_{x \rightarrow a_i^+} \ddx f(x)$. Then there is a linear bijection $f_i : \ga_i \fun \gb_i$ with slope $d_i$ and $f_i(a_i) = f(a_i)$. The claim follows.
\end{proof}

In classical analysis, a function $f : A \fun \R$ is called a Darboux function if it has the intermediate value property, that is, for any $a, b \in \dom(f)$ and any $x \in [f(a),  f(b)]$, there is a $c \in [a, b]$ with $f(c) = x$. By the intermediate value theorem, every continuous function on a real interval is a Darboux function. There are discontinuous Darboux functions, but every discontinuity of such a function is essential, that is, at least one of the one-sided limits does not exist (in $\R \cup \pm \infty$). Darboux's theorem states that if $f$ is a real-valued differentiable function on an open interval then $\ddx f$ is a  Darboux function. This theorem is an easy consequence of the extreme value theorem and hence also holds in any \omin-minimal field.

The following easy lemma must have been observed before, but we are unable to find a reference in the literature on \omin-minimality.


\begin{lem}\label{darb}
Let $(R, <, \ldots)$ be an \omin-minimal field. Let $f : A \fun R$ be a definable Darboux function, where $A\sub R$ is open. Then $f$ is continuous. In particular, if $f$ is differentiable then $\ddx f$ is continuous.
\end{lem}
\begin{proof}
By \omin-minimality, one-sided limits always exist at any $a \in A$. So the first claim holds as remarked above. The second claim follows from Darboux's theorem.
\end{proof}

This lemma will not be used later. It is stated here merely to showcase a natural scenario in which the conclusion of Lemma~\ref{diff:dtdp} holds: If  $f : A \fun B$ is a differentiable definable bijection between two open subsets of $\VF$ then, by Lemma~\ref{darb}, $\ddx f$ is continuous and hence every $a \in A$ is contained in an open interval $B_a \sub A$ such that $\vv \circ \ddx f$ is constant on $B_a$.

Next, we would like to differentiate functions between arbitrary definable sets. The simplest way to do this is to ``forget'' the $\RV$-coordinates. More precisely, let
$f : \VF^n \times \RV^m \fun \VF^{n'} \times \RV^{m'}$
be a definable function. By compactness, for every $t \in \RV^m$ there is an $s \in \RV^{m'}$ such that $\dim_{\VF}(\dom(f_{(t, s)})) = n$ and hence $\dom(f_{(t, s)})$ has an open subset. For such an $s \in \RV^{m'}$ and each $a$ contained in an open subset of $\dom(f_{(t, s)})$, we define the directional derivatives of $f$ at $(a, t)$ to be those of $f_{(t, s)}$ at $a$. Then every partial derivative of $f$ is defined almost everywhere.

Let $f : U \fun V$ be a definable function between two definable sets in $\K$. Suppose that $U$ contains an open set. Then, by Remark~\ref{pillars} and \omin-minimal differentiability, $f$ is $C^p$ almost everywhere  for all $p$. (Derivatives can actually be defined for definable functions between $\K$-torsors, see \cite[Definition~2.37]{Yin:tcon:I}.) More generally, if $U$, $V$ are sets in $\RV$ then the \emph{normalized partial derivative} $\npard j f_i(u)$ of $f$ at $u \in U$ is defined as $\pard j f^*_i(1)$, if it exists, where $f^*_i$ is the $i$th component of the induced function $f^*$ between the multiplicative translations $U / u$ and $V / f(u)$; the partial derivative $\pard j f_i(u)$ of $f$ is then  given by $u_j^{-1} \cdot f_i(u)  \cdot \npard j f_i(u)$. It follows from \cite[Lemma~2.29]{Yin:tcon:I} and compactness that every (normalized) partial derivative of $f$ is defined almost everywhere. The normalized partial derivatives of $f$ at $u \in U$ give rise to the \emph{normalized Jacobian} $\jcb_{\K} f (u)$; the partial derivatives of $f$ give rise to the \emph{$\RV$-Jacobian} $\jcb_{\RV} f (u)$. We have
\[
\jcb_{\RV} f (u) = (\Pi u)^{-1} \cdot \Pi f(u) \cdot \jcb_{\K} f (u),
\]
where $\Pi (u_1, \ldots, u_n) = u_1 \cdot \ldots \cdot u_n$.

Even more generally, let $F$ be a subset of $\RV^n \times \RV^{n'}$. We say that $F$ is a \emph{local  function at $(u,v) \in F$} if  there is a function $F_{(u,v)} : U \fun V$ such that $U \sub \RV^n$, $V \sub \RV^{n'}$ are open and $F \cap (U \times V) = F_{(u,v)}$. In that case, we define the (normalized) partial derivatives of $F$ at $(u,v)$ to be those of $F_{(u,v)}$ at $(u,v)$ and also write $\jcb_{\K} F (u, v)$, $\jcb_{\RV} F (u, v)$ for the Jacobians, if they exist.

We may further coarsen the data and define the \emph{$\Gamma$-Jacobian} of a (not necessarily definable) correspondence. Let $U \sub \RV^n \times \Gamma^m$, $V \sub \RV^{n'} \times \Gamma^{m'}$, and $C \sub U \times V$. The \emph{$\Gamma$-Jacobian} of $C$ at $((u, \alpha), (v, \beta)) \in C$, written as $\jcb_{\Gamma} C((u, \alpha), (v, \beta))$, is the element
\[
(\Pi (\vrv(u), \alpha))^{-1} \Pi (\vrv(v), \beta)  \in \Gamma.
\]
 Also set
\[
\jcb_{\absG} C((u, \alpha), (v, \beta)) =  - \Sigma (\abvrv(u), \abs \alpha) + \Sigma (\abvrv(v), \abs \beta),
\]
which equals $\abs{\jcb_{\Gamma} C((u, \alpha), (v, \beta))} \in \absG$; here $\Sigma (\gamma_1, \ldots, \gamma_n) = \gamma_1 + \ldots + \gamma_n$. If $C$ is the graph of a function then we just write $C(u, \alpha)$ instead of $C((u, \alpha), (v, \beta))$.

\begin{lem}[{\cite[Corollary~3.32]{Yin:tcon:I}}]\label{rv:op:comm}
Let $U^\sharp \sub (\OO^\times)^n$  be a definable set with $\dim_{\RV}(U) = n$. Let $P(x_1, \ldots, x_m)$ be a partial differential operator with definable $\res$-contractible coefficients $a_i : U^\sharp \fun \OO$ and $P_{\downarrow_{\res}}(x_1, \ldots, x_m)$ the corresponding operator with $\res$-contracted coefficients $a_{i\downarrow_{\res}} : \res(U^\sharp) \fun \K$. Let
\[
f = (f_1, \ldots, f_m) : U^\sharp \fun \OO
\]
be a sequence of definable $\res$-contractible functions. Then, for almost all $t \in U$ and all $a \in t^\sharp$,
\[
\res(P(f)(a)) = P_{\downarrow_{\res}}(f_{\downarrow_{\res}})(\res(a)).
\]
\end{lem}

\section{Grothendieck semirings}\label{section:groth}

From now on we require that the substructure $\mdl S$ be $\VF$-generated and $\Gamma(\mdl S)$ be nontrivial (see Remark~\ref{rem:hyp}). These two conditions are imposed at later stages in \cite{Yin:tcon:I} for the sake of greater generality, which is not a concern here.

\subsection{The categories of definable sets with volume forms}

\begin{defn}
An \emph{$\RV$-fiber} of a definable set $A$ is a set of the form $A_a$, where $a \in A_{\VF}$. The \emph{$\RV$-fiber dimension} of $A$ is the maximum of the $\RV$-dimensions of its $\RV$-fibers and is denoted by $\dim^{\fib}_{\RV}(A)$.
\end{defn}

\begin{lem}[{\cite[Lemma~4.2]{Yin:tcon:I}}]\label{RV:fiber:dim:same}
Let $f : A \fun A'$ be a definable bijection. Then $\dim^{\fib}_{\RV}(A) = \dim^{\fib}_{\RV} (A')$.
\end{lem}

\begin{defn}[$\VF$-categories]\label{defn:VF:cat}
The objects of the category $\VF[k]$ are the definable sets of $\VF$-dimension less than or equal to $k$ and $\RV$-fiber dimension $0$ (that is, all the $\RV$-fibers are finite). Any definable bijection between two such objects is a morphism of $\VF[k]$. Set $\VF_* = \bigcup_k \VF[k]$.
\end{defn}

As soon as one considers adding volume forms to definable sets in $\VF$, the question of ambient dimension arises and, consequently, one has to take ``essential bijections'' as morphisms.

\begin{defn}[$\VF$-categories with volume forms]\label{defn:VF:mu}
An object of the category $\mVF[k]$ is a definable pair $(A, \omega)$, where $A \in \VF[k]$, $A_{\VF} \sub \VF^k$, and $\omega : A \fun \RV$ is a function, which is understood as a \emph{definable $\RV$-volume form} on $A$. A \emph{morphism} between two such objects $(A, \omega)$, $(B, \sigma)$ is a definable \emph{essential bijection} $F : A \fun B$, that is, a bijection that is defined outside definable subsets of $A$, $B$ of $\VF$-dimension less than $k$, such that, for every $x \in \dom(F)$,
\[
\omega(x) = \sigma(F(x)) \rv(\jcb_{\VF} F(x)).
\]
We say that such an $F$ is \emph{$\RV$-measure-preserving}, or simply measure-preserving.

The category $\mgVF[k]$ is similar, except that the \emph{definable $\Gamma$-volume form} $\omega$ is of the form $A \fun \Gamma$ and the \emph{$\Gamma$-measure-preserving} essential bijection $F$ needs to satisfy a weaker condition
\[
\omega(x) = \sigma(F(x)) \vv(\jcb_{\VF} F(x)).
\]
\end{defn}

In the definition above and other similar ones below, for the cases $k =0$, the reader should interpret things such as $\VF^0$ and how they interact with other things in a natural way. For instance, $\VF^0$ may be treated as the empty tuple, the only definable set of $\VF$-dimension less than $0$ is the empty set, and $\jcb_{\VF}$ is always $1$ on sets that have no $\VF$-coordinates. Thus $(A, \omega) \in \mVF[0]$ if and only if $A$ is a finite definable subset of $\RV_0^n$ for some $n$.

Set $\mVF[{\leq} k] = \bigoplus_{i \leq k} \mVF[i]$ and $\mVF[*] = \bigoplus_{k} \mVF[k]$; similarly for the other categories below (with or without volume forms).

\begin{rem}\label{mor:equi}
Let $F : (A, \omega) \fun (B, \sigma)$ be a $\mVF[k]$-morphism. Our intension is that such an $F$ should identify the two objects. However, if $F$ is not defined everywhere in $A$ then obviously it does not admit an inverse. We remedy this by introducing the following obvious congruence relation $\sim$ on $\mVF[k]$. Let $G : (A, \omega) \fun (B, \sigma)$ be another $\mVF[k]$-morphism. Then $F \sim G$ if $F(a) = G(a)$ for all $a \in A$  outside a definable subset  of $\VF$-dimension less than $k$. The morphisms of the quotient category $\mVF[k] /{\sim}$ have the form $[F]$, where $F$ is a $\mVF[k]$-morphism. Clearly every $(\mVF[k] / {\sim})$-morphism is an isomorphism and hence $\mVF[k] / {\sim}$ is a groupoid. In fact, all the categories of definable sets we shall work with should be and are groupoids.

It is certainly more convenient to work with representatives than equivalence classes. In the discussion below, this quotient category $\mVF[k] /{\sim}$ will almost never be needed except when it comes to forming the Grothendieck semigroup or, by abuse of terminology, when we speak of two objects of $\mVF[k]$ being isomorphic.
\end{rem}

%

\begin{defn}[$\RV$-categories]\label{defn:c:RV:cat}
The objects of the category $\RV[k]$ are the pairs $(U, f)$ with $U$ a set in $\RVV$ and $f : U \fun \RV^k$ a definable finite-to-one function. Given two such objects $(U, f)$, $(V, g)$, any definable bijection $F : U \fun V$ is a \emph{morphism} of $\RV[k]$.
\end{defn}

Note that the categories $\VF[0]$, $\RV[0]$ are equivalent, similarly for other such categories.

\begin{nota}\label{0coor}
We emphasize that if $(U, f)$ is an object of $\RV[k]$ then $f(U)$ is a subset of $\RV^k$ instead of $\RV_0^k$, while $0$ can occur in any coordinate of $U$. An object of  $\RV[*]$ of the form $(U, \id)$ is often just written as $U$.

More generally, if $f : U \fun \RV_0^k$ is a definable finite-to-one function then $(U, f)$ denotes the obvious object of $\RV[{\leq} k]$. Often $f$ will just be a coordinate projection (every object in $\RV[*]$ is isomorphic to an object of this form). In that case, the object $(U, \pr_{\leq k})$ is simply denoted by $U_{\leq k}$ and its class in $\gsk \RV[{\leq} k]$ by $[U]_{\leq k}$, etc.
\end{nota}

\begin{defn}[$\RES$-categories]\label{defn:RES:cat}
The category $\RES[k]$ is the full subcategory of $\RV[k]$ such that $(U, f) \in \RES[k]$ if and only if $\vrv(U)$ is finite.
\end{defn}

\begin{rem}\label{fin:dec}
Every $\RV[k]$-morphism $F: (U, f) \fun (V, g)$  induces a definable finite-to-finite correspondence $F^\dag \sub f(U) \times g(V)$. Since $F^\dag$ can be decomposed into finitely many definable bijections $(F^\dag_i)_i$ such that each $F^\dag_i$ is differentiable everywhere if $\dom(F^\dag_i)$ is of $\RV$-dimension $k$,  we see that $F^\dag$ is a local bijection outside a definable subset of $\RV$-dimension less than $k$.
\end{rem}

\begin{defn}[$\RV$- and $\RES$-categories with volume forms]\label{defn:RV:mu}
An object of the category $\mRV[k]$ is a definable triple $(U, f, \omega)$, where $(U, f)$ is an object of $\RV[k]$ and $\omega : U \fun \RV$ is a function, which is understood as a \emph{definable $\RV$-volume form} on $(U, f)$. A \emph{morphism} between two such objects $(U,f, \omega)$, $(V, g, \sigma)$ is an $\RV[k]$-morphism $F: (U, f) \fun (V, g)$ such that
\begin{itemize}
  \item $\omega(u) = \sigma(F(u)) \cdot \jcb_{\RV} F^\dag(f(u), g \circ F(u))$ for every $u \in U$ outside a definable subset  of $\RV$-dimension less than $k$,
  \item $\vrv(\omega(u)) = (\vrv \circ \sigma \circ F)(u) \cdot \jcb_{\Gamma}  F^\dag(f(u), g \circ F(u))$ for every $ u \in U$.
\end{itemize}

The category $\mgRV[k]$ is similar, except that the \emph{definable $\Gamma$-volume form} $\omega$ is of the form $U \fun \Gamma$ and a morphism only needs to satisfy the second condition.

The categories $\mRES[k]$, $\mgRES[k]$ are the obvious full subcategories of $\mRV[k]$, $\mgRV[k]$.
\end{defn}

\begin{defn}[$\RV$- and $\RES$-categories with volumes]\label{defn:RV:cat:1}
The category ${\vol}{\RV[k]}$ is the full subcategory of $\mRV[k]$ such that $(U, f, \omega) \in {\vol}{\RV[k]}$ if and only if $\omega = 1$.

The categories ${\vol_{\Gamma}}{\RV[k]}$, ${\vol}{\RES[k]}$, and $\vgRES[k]$ are similar, but with $\mRV[k]$ replaced by $\mgRV[k]$, $\mRES[k]$, and $\mgRES[k]$.
\end{defn}

We may view ${\vol}{\RV[k]}$ as a subcategory of $\RV[k]$ and hence will use the same notation for their objects, similarly for the other pairs.

\begin{rem}\label{G:vol}
Let $k > 0$ and $(U, \omega)$ be an object of $\mgRES[k]$ with $\vrv(U) = \gamma \in \Gamma^k$ and $\omega(U) = \alpha \in \Gamma$. Let $t \in \alpha^\sharp$ be a definable point and $tU$ the multiplicative translation of $U$ in the first coordinate by $t$. Then $(U, \omega)$ is isomorphic to $(tU, 1)$. It follows that
\[
\gsk \mgRES[k] \cong \gsk \vgRES[k].
\]
This does not hold for the other pairs, though.
\end{rem}

\begin{defn}[$\Gamma$-categories]\label{def:Ga:cat}
The objects of the category $\Gamma[k]$ are the finite disjoint unions of definable subsets of $\Gamma^k$. Any definable bijection between two such objects is a \emph{morphism} of $\Gamma[k]$.

The category $\Gamma^{c}[k]$ is the full subcategory of $\Gamma[k]$ such that $I \in \Gamma^{c}[k]$ if and only if $I$ is finite.
\end{defn}

\begin{defn}[$\Gamma$-categories with volume forms]\label{def:Ga:cat:mu}
An object of the category $\mG[k]$ is a definable pair $(I, \omega)$, where $I \in \Gamma[k]$ and $\omega : I \fun \Gamma$ is a function. A $\mG[k]$-morphism between two objects $(I, \omega)$, $(J, \sigma)$ is a definable bijection $F : I \fun J$ such that, for all $\alpha  \in I$,
\[
\omega(\alpha) = \sigma(F(\alpha)) \jcb_{\Gamma} F(\alpha).
\]

The category $\mG^{c}[k]$ is the obvious full subcategory of $\mG[k]$.
\end{defn}

\begin{nota}\label{nota:RV:short}
We introduce the following shorthand for distinguished elements in the various Grothendieck semigroups and their groupifications (and closely related constructions):
\begin{gather*}
\bm 1_{\K} = [\{1\}] \in \gsk \RES[0], \quad [1] = [\{1\}] \in \gsk \RES[1],\\
[\bm T] = [\K^+] \in \gsk \RES[1], \quad [\bm A] = [\bm T] + [- \bm T] + [1] \in \gsk \RES[1],\\
\bm 1_{\Gamma} = [\Gamma^0] \in \gsk \Gamma[0], \quad [e] = [\{1\}] \in \gsk \Gamma[1], \quad [\bm H] = [(0, 1)] \in \gsk \Gamma[1],\\
\bm P = [1] - [\RV^{\circ \circ}] \in \ggk \RV[1].
\end{gather*}
Here $\RV^{\circ\circ} = \rv(\MM \mi 0)$  and $- \bm T$ is the object $\K^-$. Also note that the interval $\bm H$ is formed in the signed value group $\Gamma$, whose ordering is inverse to that of the value group $\abs \Gamma_\infty$ (recall Remark~\ref{signed:Gam}). The interval $(1,\infty) \sub  \Gamma$ is denoted by $\bm H^{-1}$. These symbols also stand for the corresponding elements with the constant volume form $1$.

As in~\cite{hrushovski:kazhdan:integration:vf}, the element $\bm P$, viewed as an element in $\ggk \mRV[*]$ or $\ggk \mgRV[*]$,  plays a  special role in the main construction (see Propositions~\ref{kernel:L} and the remark thereafter).
\end{nota}


\begin{lem}[{\cite[Lemma~4.17]{Yin:tcon:I}}]\label{resg:decom}
Let $A \sub \RV^k \times \Gamma^l$ be an $\alpha$-definable set, where $\alpha \in \Gamma$. Set $A_{\RV} = U$ and assume $U \in \RES[*]$. Then there is an $\alpha$-definable finite partition $(U_i)_i$ of $U$ such that, for each $i$ and all $t, t' \in U_i$, we have $A_t = A_{t'}$.
\end{lem}


If $A$ is the graph of a function $\RV^k \fun \Gamma^l$ then this is just one half of Lemma~\ref{gk:ortho}.

We need a more general version of \cite[Lemma~4.19]{Yin:tcon:I}:

\begin{lem}\label{RV:decom:RES:G}
Let $A \sub \RV^k \times \Gamma^l$ be a definable set and $A_{\RV} = U$. Then there are finitely many definable injections $\rho_i : V_i \fun U$  with $V_i \sub \RV^k$ such that
\begin{itemize}
  \item the sets $\rho_i(V_i)$ form a partition of $U$,
  \item each set $\bigcup_{t \in V_i} t \times A_{\rho_i(t)}^\sharp$ is of the form $U_i \times I_i^\sharp$, where
      \[
      U_i \sub (\K^+)^{k_i}, \quad I_i \sub \Gamma^{l_i}, \quad k_i + l_i = k + l, \quad k_i \leq k.
      \]
\end{itemize}
\end{lem}

\begin{proof}
We proceed by induction on $k$. For the base case $k =1$,  by Lemma~\ref{resg:decom}, we see that the claim holds for every $A^{\alpha} \coloneqq (\alpha^\sharp \times \Gamma^l) \cap A$. By weak \omin-minimality, there are only finitely many $\alpha \in \vrv(U)$ for which the partition of $\pr_1(A^{\alpha})$ in question does not just contain one piece $\alpha^\sharp$. This readily implies the base case (the desired injections $\rho_i$ are just identity functions).

For the inductive step $k > 1$,  by the argument above and compactness, there are an $n \in \N$ and a definable function $U' \coloneqq \pr_{<k}(U) \fun \Gamma^n$ given by $t \efun (\alpha_{ it})_i$ such that every set $A_{ t}^* \coloneqq A_{ t} \mi \bigcup_i (\alpha_{ it}^\sharp \times \Gamma^l)$
is a union of sets of the form $\alpha^\sharp \times D_\alpha$, where $\alpha \times D_\alpha \sub \Gamma \times \Gamma^l$. Since $\bigcup_{t \in U'} t \times \vrv(A_{t}^*)$ is a subset of $\RV^{k-1} \times \Gamma^{l+1}$, by the inductive hypothesis, we may assume that it is empty. Now, without loss of generality, we may further assume $n =1$ and thereby simply write $t \efun \alpha_{t}$. There is a definable function $f$ on $U'$ such that $f( t) \in \alpha_{t}^\sharp$. This means that, after translating each  $U_{ t}$ by $1 /f( t)$, we may take $\pr_k(U) \sub \K^+$ for all $ t \in U'$. For each $t \in \pr_k(U)$, by the inductive hypothesis again, there are $t$-definable sets $U_{it} \times I_{it}^\sharp$ as desired for $A_t$. So, by compactness, we are reduced to Lemma~\ref{resg:decom}, and this concludes the proof.
\end{proof}

\begin{rem}
The assignment $(I, \omega) \efun (I^\sharp, \id, \omega^\sharp)$, where $\omega^\sharp = \omega \circ \vrv$, determines a map $\mG[*] \fun \mgRV[*]$. By Lemma~\ref{gam:pulback:mono}, this map induces a graded semiring homomorphism
$\gsk \mG[*] \fun \gsk \mgRV[*]$. By Remark~\ref{pillars} and Theorem~\ref{groth:omin}, this homomorphism restricts to an injective homomorphism $\gsk \mG^{c}[*] \fun \gsk \mgRES[*]$ of graded semirings. The map
\[
\gsk \mgRES[*] \times \gsk \mG[*] \fun \gsk \mgRV[*]
\]
 determined by the assignment
\[
([(U, f, \omega)], [(I, \sigma)]) \efun [(U \times I^\sharp, f \times \id, \omega \sigma^\sharp)],
\]
where $\omega \sigma^\sharp$ is the $\Gamma$-volume form on $U \times I^\sharp$ given by $(u, v) \efun \omega(u)\sigma^\sharp(v)$, is well-defined and is clearly $\gsk \mG^{c}[*]$-bilinear. Hence it induces a $\gsk \mG^{c}[*]$-linear map
\begin{equation}\label{def:mgd}
\mgD: \gsk \mgRES[*] \otimes_{\gsk \mG^{c}[*]} \gsk \mG[*] \fun \gsk \mgRV[*],
\end{equation}
which is a homomorphism of graded semirings. We will abbreviate ``$\otimes_{\gsk \mG^{c}[*]}$'' as ``$\otimes$'' below.
\end{rem}

Recall that there is a similar $\gsk \Gamma^{c}[*]$-linear map $\bb D$ without volume forms and  \cite[Proposition~4.21]{Yin:tcon:I} states that it is indeed an isomorphism. This also holds for $\mgD$.

\begin{prop}\label{mgD:iso}
The map $\mgD$ is an isomorphism of graded semirings.
\end{prop}
\begin{proof}
We first show that the map $\mgD$ is surjective. Let $(U, f, \omega) \in \mgRV[n]$ and $\omega_f : U \fun \Gamma$ be the function given by $u \efun \omega(u) \Pi \vrv(f(u))$. Applying Lemma~\ref{RV:decom:RES:G} to (the graph of) $\omega_f$, we see that there are finitely many definable injections $\rho_i : U_i \times I_i^\sharp \fun U$, where $U_i$ is a set in $\K^+$ and $I_i$ a set in $\Gamma$, such that the sets $\rho_i(U_i \times I_i^\sharp)$ form a partition of $U$ and $\omega_f \rest \rho_i(U_i \times \gamma^\sharp)$ is constant for every $\gamma \in I_i$. Let $\omega_i = \omega_f \circ \rho_i$ if $I_i$ is empty and $\omega_i = 1$ otherwise. Let $\sigma_i : I_i \fun \Gamma$ be the function given by
\[
\gamma \efun \omega_f (\rho_i(U_i \times \gamma^\sharp)) /\Pi \gamma.
\]
Then the sum $\sum_i [(U_i, \omega_i)] \otimes [(I_i, \sigma_i)]$ is mapped to $[(U, f, \omega)]$ by $\mgD$ and hence $\mgD$ is surjective.

Injectivity  can be established as in the proof of \cite[Proposition~4.21]{Yin:tcon:I}, with only slight modifications to accommodate the presence of volume forms.
\end{proof}

\subsection{Retractions}\label{subsection:retract}

\begin{nota}\label{euler}
We view $\Gamma$ as a double cover of $\abs \Gamma$ via the identification $\Gamma / {\pm 1} = \abs \Gamma$. Consequently there are two Euler characteristics $\chi_{\Gamma,g}$, $\chi_{\Gamma, b}$ in the $\Gamma$-sort, induced by those on $|\Gamma|$ (see \cite{kage:fujita:2006} and also~\cite[\S~9]{hrushovski:kazhdan:integration:vf}). They are distinguished by
\[
\chi_{\Gamma, g}(\bm H) = \chi_{\Gamma, g}(\bm H^{-1}) = -1 \quad \text{and} \quad \chi_{\Gamma, b}(\bm H) = \chi_{\Gamma, b}(\bm H^{-1}) = 0.
\]
Similarly, there is an Euler characteristic $\chi_{\K}$ in the $\K$-sort (there is only one). Also see Remark~\ref{omin:res}.  All of these Euler characteristics shall be simply denoted by $\chi$ if no confusion can arise.
\end{nota}

Combining the semiring homomorphisms induced by the forgetful functors $\mG[*] \fun \Gamma[*]$, $\mgRES[*] \fun \RES[*]$ with the ones in \cite[Lemma~4.22]{Yin:tcon:I}, we obtain graded ring homomorphisms $\ggk \mG[*] \fun \Z[X]$, $\ggk \mgRES[*] \fun \Z[X]$.
However, to construct more meaningful retractions $\ggk \mgRV[*] \fun \ggk \mgRES[*]$ as in \cite[Proposition~4.24]{Yin:tcon:I}, we need a graded ring richer than $\Z[X]$.

Let $\fn(\Gamma, \Z)$ be the set of functions $f: \Gamma \fun \Z$ such that  $\ran(f)$ is finite and $f^{-1}(m)$ is definable for every $m \in \Z$. Addition in $\fn(\Gamma, \Z)$ is defined in the obvious way. Multiplication in $\fn(\Gamma, \Z)$ is given by convolution product as follows. For $f, g \in \fn(\Gamma, \Z)$ and $\gamma \in \Gamma$, let $h_\gamma$ be the function $\Gamma \fun \Z$ given by $\alpha \efun f(\alpha)g(\gamma/\alpha)$. The range of $h_\gamma$ is obviously finite. For each $m \in \Z$, $h_\gamma^{-1}(m)$ is a finite disjoint union of sets of the form $f^{-1}(m') \cap \gamma / g^{-1}(m'')$, where $m' m'' = m$, and hence is $\gamma$-definable. Let
\[
 h_\gamma^* = \sum_{m} m \chi(h_\gamma^{-1}(m)) \in \Z
\]
and $f * g : \Gamma \fun \Z$ be the function  given by $\gamma \efun h_\gamma^*$. It is not hard to see that, by \omin-minimality in the $\Gamma$-sort, $f * g = g * f$ and it belongs to $\fn(\Gamma, \Z)$. Thus $\fn(\Gamma, \Z)$ is a commutative ring.

\begin{rem}
Let $r_\pm$ be the characteristic functions of the intervals $(0, \pm \infty)$. For each definable element $\gamma \in \Gamma$, let $p_{\gamma}$,  $q_{\gamma}$ be the characteristic functions of $\gamma$, $(0, \gamma)$, respectively. By \omin-minimality  again, as a $\Z$-module, $\fn(\Gamma, \Z)$ is generated by elements of the forms $r_\gamma$, $p_\gamma$, $q_\gamma$. We have the following equalities (recall the function $\sgn$ from Remark~\ref{signed:Gam}):
\[
r_\alpha p_\beta   = r_{\sgn(\alpha \beta)}, \quad p_\alpha p_\beta = p_{\alpha \beta}, \quad p_\alpha q_\beta  = q_{\alpha \beta}, \quad q_\alpha q_\beta = - q_{\alpha \beta}.
\]
In addition,
\begin{itemize}
  \item if $\chi = \chi_{\Gamma, g}$ then $r_\alpha r_\beta = - r_{\sgn(\alpha \beta)}$ and $r_\alpha q_\beta = - r_{\sgn(\alpha \beta)}$,
  \item if $\chi = \chi_{\Gamma, b}$ then $r_\alpha r_\beta =  r_{\sgn(\alpha \beta)}$ and $r_\alpha q_\beta = 0$.
\end{itemize}
\end{rem}

For each $\bm I = (I, \mu) \in \mG[k]$, let $\lambda_{\bm I} : \Gamma \fun \Z$ be the function given by $\gamma \efun \chi(\mu^{-1}_I(\gamma))$, where $\mu_I : I \fun \Gamma$ is the function given by $\gamma \efun \mu(\gamma)\Pi\gamma$. The range of $\lambda_{\bm I}$ is finite and, for every $m \in \Z$, $\lambda_{\bm I}^{-1}(m)$ is definable. So $\lambda_{\bm I} \in \fn(\Gamma, \Z)$. If $F :(I, \mu) \fun (J, \sigma)$ is a $\mG[k]$-morphism then $F$ restricts to a bijection from $\mu^{-1}_I(\gamma)$ to $\sigma_J^{-1}(\gamma)$ for every $\gamma \in \Gamma$ and hence $\lambda_{\bm I} = \lambda_{\bm J}$ and we may write $\lambda_{\bm I}$ as $\lambda_{[\bm I]}$. Let $\bm J = (J, \sigma) \in \mG[l]$. If $k=l$ then clearly $\lambda_{[\bm I]} + \lambda_{[\bm J]} = \lambda_{[\bm I]+[\bm J]}$. An easy computation shows that, for every $\gamma \in \Gamma$,
\[
(\lambda_{[\bm I]} * \lambda_{[\bm J]})(\gamma) = \chi( (\mu \cdot \sigma)^{-1}_{I \times J}(\gamma))
\]
and hence $\lambda_{[\bm I]} * \lambda_{[\bm J]} = \lambda_{[\bm I][\bm J]}$. Let $\Z\Gamma$ be the group ring of $\Gamma(\mdl S)$ over $\Z$, which is viewed as the subring of $\fn(\Gamma, \Z)$ generated, as a $\Z$-module, by $p_{\pm \gamma}$, $\gamma \in \Gamma(\mdl S)$. We have $\ggk \mG^c[*] \cong \Z\Gamma[X]$. It follows that the map
\[
\lambda : \ggk \mG[*] \fun \Z\Gamma \oplus X \fn(\Gamma, \Z)[X]
\]
determined by the assignment  $[\bm I] \efun \lambda_{[\bm I]}X^k$, $\bm I \in \mG[k]$, is indeed a graded ring homomorphism. Of course there are two such homomorphisms $\lambda_g$ and $\lambda_b$, corresponding to the two cases $\chi = \chi_{\Gamma, g}$ and $\chi = \chi_{\Gamma, b}$.

Denote the element $[(\bm H, (\cdot)^{-1})] \in \ggk \mG[1]$ by $[\bm H_*]$. For later use, we compute a few examples:
\begin{equation}\label{half:compu}
  \begin{aligned}
    \lambda([\bm H]) & = q_1 X,  & \lambda([-\bm H])  &= q_{-1} X, \\
 \lambda_g([\pm \bm H_*]) & = -p_{\pm 1} X, & \lambda_b([\pm \bm H_*]) & = 0.
  \end{aligned}
\end{equation}

\begin{nota}\label{vres:to:zg}
Recall from Remark~\ref{G:vol} that, for $k > 0$,  $\ggk \mgRES[k]$ may be identified with $\ggk \vgRES[k]$. On the other hand, $\ggk \mgRES[0] \cong \Z\Gamma$ and $\ggk \vgRES[0] \cong \Z$. As far as the Grothendieck ring is concerned, we prefer to work with objects in the subcategories $\vgRES[k]$, $k > 0$, because the notation will be much simpler.

For $U \in \mgRES[k]$ with $\vrv(U) = \gamma$, let $\phi_U = \chi(U) \Pi\gamma\in \Z\Gamma$. By summation over disjoint union, this assignment is extended to all objects of $\mgRES[k]$. Clearly if $[U] = [V]$ then $\phi_{U} = \phi_{V}$ and hence we may write $\phi_{ U}$ as $\phi_{[U]}$. It is routine to check that the map
\[
\phi: \ggk \mgRES[*] \fun \Z\Gamma[X]
\]
determined by the assignment $[U] \efun \phi_{[U]}X^k$, $U \in \mgRES[k]$, is a graded ring isomorphism.
\end{nota}

Let $I!$ be the ideal of $\ggk \mgRES[*]$ generated by the element $[\bm T] + [-\bm T]- [1]$, which may also be written as $2[\bm T] + [-\bm T]$. We have
\[
\phi^{-1}(p_{-1})([\bm T] + [-\bm T]- [1]) = [- \bm T] + [\bm T]- [-1].
\]
Thus, modulo $I!$,
\[
[1] - [-1] = [-\bm T]-[\bm T] = [\bm T] + [-\bm T]- [1] = - 3[1] = 0.
\]
Choose a definable $t_\gamma \in \gamma^\sharp$ for each definable $\gamma \in \Gamma$. Then the same relations hold for $[\gamma^\sharp]$, $[- \gamma^\sharp]$, $[\{t_\gamma\}]$, $[\{-t_\gamma\}]$ modulo $I!$. The image of $I!$ under $\phi$ is still denoted by $I!$. Since $\phi([\bm T]) = -X$ and $\phi([1]) = X$,  we see that there is an isomorphism
\[
\Z\Gamma[X] / I! \cong \Z\Gamma \oplus X (\F_3\abs\Gamma[X]),
\]
where $\F_3\abs\Gamma$ is the group ring of $\abs\Gamma(\mdl S)$ over the finite field $\F_3$. It and $\phi$ induce a homomorphism $\ggk \mgRES[*]/ I! \fun \F_3\abs\Gamma[X]$, which is still denoted by $\phi$.

We claim that the assignments
\[
p_\gamma X^k \efun  [\bm T]^{k-1}[\gamma^\sharp], \quad q_\gamma X^k \efun  [\bm T]^{k-1}[\{t_\gamma\}], \quad r_\pm X^k \efun 0
\]
induce a graded ring homomorphism
\[
\psi : \Z\Gamma \oplus X\fn(\Gamma, \Z)[X] \fun \ggk \mgRES[*].
\]
It suffices to check that $\psi$ obeys the equational constraints for the generators above, which are all straightforward (the equality $\psi(q_\alpha) \psi( q_\beta) = \psi(- q_{\alpha \beta})$ holds since $[\bm T] = -[1]$).  We also note that, for the case $\chi = \chi_{\Gamma, g}$, there seems to be the possibility of $r_\pm X^k \efun - [\bm T]^k$ or $r_\pm X^k \efun \mp [\bm T]^k$,
but this does not yield a homomorphism, at least not without drastic collapse in the target ring (see Remark~\ref{psi:constraint} below for further explanation).


Let $\gh$ denote the element $[(\RV^{\circ \circ},  1/{\vrv})] \in \ggk \mgRV[1]$.

\begin{prop}\label{prop:eu:retr:k}
There are two graded ring homomorphisms
\[
  \mgE_{g}, \mgE_{b}: \ggk \mgRV[*] \fun  \ggk \mgRES[*] / I! \to^\phi \F_3\abs\Gamma[X]
\]
such that
\begin{itemize}
  \item $\bm P \in \ggk \mgRV[1]$ vanishes under both of them,
  \item for all $x \in \ggk \mgRES[k]$, $\mgE_{g}(x) = \mgE_{b}(x) = \phi(x / I!)$,
  \item they are distinguished by $\mgE_{g}(\gh) = \phi( - [1]/I!)$ and $\mgE_{b}(\gh) = 0$.
\end{itemize}
\end{prop}

\begin{proof}
Let $\iota$  be the obvious graded ring involution on $\ggk \mgRES[*]$, that is, $\iota(x) = (-1)^k x$ for $x \in \ggk \mgRES[k]$. Its composition with the injective homomorphism $\gsk \mG^{c}[*] \fun \gsk \mgRES[*]$  gives rise to a tensor product $\gsk \mgRES[*] \otimes^\iota \gsk \mG[*]$ over $\gsk \mG^{c}[*]$. There is an isomorphism between this tensor product and the other tensor product in (\ref{def:mgd}), determined by the assignments $x \otimes^\iota y \efun x \otimes y$.  Proposition~\ref{mgD:iso} then yields an isomorphism
\[
\mgD^\iota : \gsk \mgRES[*] \otimes^\iota \gsk \mG[*] \fun \ggk \mgRV[*].
\]
Now, the map
\[
\iota \times (\psi \circ \lambda) : \ggk \mgRES[*] \times  \ggk \mG[*] \fun   \ggk \mgRES[*]
\]
is $\ggk \mG^c[*]$-bilinear. Since
\[
(\mgD^\iota)^{-1}(\bm P) = [\bm H] + [-\bm H] - [1] \dand (\mgD^\iota)^{-1}(\gh) = [\bm H_*] + [-\bm H_*],
\]
the existence of the desired homomorphisms follow from (\ref{half:compu}).
\end{proof}

\begin{rem}\label{psi:constraint}
We would not have $\mgE(\bm P) = 0$ without forcing $3[1] = 0$, which is the main reason why the ideal $I!$ appears in the expression. The goal of obtaining such homomorphisms dictates almost entirely how the graded ring homomorphism $\psi$ must be constructed. The image of $p_\gamma X^k$ is determined by the tensor $\otimes_{\gsk \mG^{c}[*]}$. The image of $q_\gamma X^k$ is determined by the equalities $p_\alpha q_\beta = q_{\alpha \beta}$, $q_\alpha q_\beta = - q_{\alpha \beta}$ and the vanishing of $\bm P$, which also bring about the formation of the ideal $I!$. The involuted tensor $\otimes^\iota$ is needed since, otherwise, the vanishing of $\bm P$ would force $[1] = 0$, which is obviously undesirable.

Now, given these constraints, suppose that we insist on having, say, $\psi(r_\pm X^k)= -[\bm T]^k$. Then the conditions
\[
\psi(r_\alpha) \psi(p_\beta) = \psi( r_{\sgn(\alpha \beta)}), \quad \psi(r_\alpha) \psi(q_\beta) = \psi( - r_{\sgn(\alpha \beta)})
\]
imply, for all definable $\beta \in \Gamma$,
\[
[\bm T]([\bm T] - [\beta^\sharp]) = 0, \quad [\bm T]([\bm T] + [\{t_\beta\}]) = 0.
\]
Adding the elements $[\bm T] - [\beta^\sharp]$, $[\bm T] + [\{t_\beta\}]$ to the ideal $I!$ causes the codomain of $\mgE$ to collapse to the ring  $\F_3[X]$ (Proposition~\ref{prop:eu:retr:k} still goes through). The collapse caused by $\psi(r_\pm X^k)= \mp[\bm T]^k$ is even more drastic (in effect forcing $[1] = 0$ again).
\end{rem}

\begin{rem}\label{involution}
There is an alternative construction, retracting $\ggk \mgRV[*]$ to the other factor $\ggk \mG[*]$. Let $\lambda^\iota$ be the involuted version of $\lambda$, that is, $\lambda^\iota([\bm I]) = (-1)^k \lambda([\bm I])$, $\bm I \in \mG[k]$.
Let $!I$ be the ideal of the ring
$\Z\Gamma \oplus X \fn(\Gamma, \Z)[X]$
generated by the element $(q_1 + q_{-1} - p_1)X$. Regarding $\Z\Gamma[X]$ as its subring, we see that $\phi \times \lambda^\iota$ is again $\ggk \mG^c[*]$-bilinear (we can also take the involuted version of $\phi$ instead of $\lambda$). This yields two  graded ring homomorphisms
\[
\ggk \mgRV[*] \fun \Z\Gamma \oplus X \fn(\Gamma, \Z)[X] / !I
\]
as in Proposition~\ref{prop:eu:retr:k}. In light of the equalities
\begin{align*}
(p_{1} + q_1)(q_1 + q_{-1} - p_1)  &= - q_{1} - p_{1}, \\
(p_{1} + q_{-1})(q_1 + q_{-1} - p_1) &= - q_{-1} - p_{1},
\end{align*}
we see that $3p_1X = 0$ modulo $!I$ and hence the two homomorphisms are really just those obtained in Proposition~\ref{prop:eu:retr:k}.
\end{rem}
%


%

Thus it appears that the foregoing  discussion is much ado about almost nothing, as essentially only information about the ``sign'' survives  the homomorphisms $\mgE_{g}$, $\mgE_{b}$. The problem is that the ring $\ggk \mgRES[*]$ is much simpler than its counterpart over algebraically closed valued fields (see \cite[Theorem~10.11]{hrushovski:kazhdan:integration:vf}) and cannot afford the further reduction demanded by the vanishing of the class $\bm P$. To remedy this, we can trim down the categories $\mgRV[*]$, $\mG[*]$ as follows.

\begin{defn}
We say that an object $(U, f) \in \RV[k]$ is \emph{doubly bounded} if $f(U)$ is, and so on. The full subcategory of $\RV[k]$ of doubly bounded objects is denoted by $\RV^{\db}[k]$; similarly for the categories $\mRV^{\db}[k]$, $\mgRV^{\db}[k]$, $\mG^{\db}[k]$.
\end{defn}

\begin{rem}\label{tensor:bd}
By inspection of the proofs of Proposition~\ref{mgD:iso} and Lemma~\ref{RV:decom:RES:G}, we see that $\mgD$ restricts to an isomorphism
\[
\mgD^{\db}: \gsk \mgRES[*] \otimes_{\gsk \mG^{c}[*]} \gsk \mG^{\db}[*] \fun \gsk \mgRV^{\db}[*].
\]
\end{rem}

Let $\fn(\abs\Gamma, \Z)$ be the obvious quotient of $\fn(\Gamma, \Z)$ and $\fn^{\db}(\abs\Gamma, \Z)$ its subring of functions with doubly bounded domains. Recall that the (unsigned) value group $\abs\Gamma$ is additively written. For each definable element $\gamma \in \abs\Gamma \mi 0$, let $o_{\gamma}$ be the characteristic function of $(0, \gamma)$; also set $o_0 = -p_0$. As a $\Z$-module, $\fn^{\db}(\abs\Gamma, \Z)$ is generated by elements of the forms $p_\gamma$, $o_\gamma$. An easy  computation shows that
\begin{itemize}
  \item $p_\alpha o_\beta$ equals $o_{\alpha + \beta} - o_\alpha - p_\alpha$ or $- o_{\alpha + \beta} + o_\alpha - p_{\alpha + \beta}$ or $ o_{\alpha + \beta} + o_\alpha + p_{0}$,
  \item $o_\alpha o_\beta$ equals $- o_{\alpha + \beta}$ or $- o_{\alpha} - o_\beta - p_{0}$.
\end{itemize}

\begin{nota}
Let $!P$ denote the ideal of the ring
$
\Z\abs\Gamma \oplus X \fn^{\db}(\abs\Gamma, \Z)[X]
$
generated by the elements $p_\gamma - p_0$, $(o_\gamma + p_0) X$. By the computation above, the quotient by $!P$ may be naturally identified with $\Z [X]$.
\end{nota}

\begin{nota}\label{db:G:int}
The homomorphism
\begin{equation*}
\lambda : \ggk \mG^{\db}[*] \fun \Z\abs\Gamma \oplus X \fn^{\db}(\abs\Gamma, \Z)[X]
\end{equation*}
is constructed as before (there is only one such homomorphism now since  $\chi_{\Gamma,g}$, $\chi_{\Gamma, b}$ agree on doubly bounded sets). Recall the homomorphism $\phi$ from Notation~\ref{vres:to:zg}. The homomorphism $\ggk \mgRES[*] \fun \Z\abs\Gamma[X]$ induced by $\phi$ is still denoted by $\phi$.
\end{nota}

\begin{nota}\label{pgamma}
For each $\gamma \in \absG_\infty$, denote the set $\rv(\MM_{\gamma} \mi 0)$ by $\RV^{\circ \circ}_\gamma$.  For each definable $\gamma \in \Gamma$ with $\abs \gamma > 0$, let
\[
\bm P_\gamma = [\RV^{\circ \circ} \mi \RV^{\circ \circ}_{\abs\gamma}] + [\{t_\gamma\}] - [1] \in \ggk \RV^{\db}[1],
\]
where $t_\gamma \in \gamma^\sharp$ is any definable point. As in Notation~\ref{nota:RV:short}, it also stands for the corresponding element  in $\ggk \mRV^{\db}[1]$ or $\ggk \mgRV^{\db}[1]$ with the constant volume form $1$. Of course $[\{t_\gamma\}] = [1]$ in $\ggk \RV^{\db}[1]$, but $[\{t_\gamma\}] \neq [1]$ in $\ggk \mRV^{\db}[1]$ or $\ggk \mgRV^{\db}[1]$ unless $\gamma = 1$.
\end{nota}

\begin{prop}\label{prop:eu:retr:k:db}
There is a graded ring homomorphism
\[
  \mgE^{\db}: \ggk \mgRV^{\db}[*] \fun \Z\abs\Gamma \oplus X \fn^{\db}(\abs\Gamma, \Z)[X] / !P \to^{\sim} \Z[X]
\]
such that every $\bm P_\gamma$ vanishes and, for all $x \in \ggk \mgRES[k]$, $\mgE^{\db}(x) = \phi(x) / !P$.
\end{prop}
\begin{proof}
The homomorphism $\phi \times \lambda^\iota$ is still $\ggk \mG^c[*]$-bilinear. We have
\[
((\phi \otimes \lambda^\iota) \circ (\mgD^{\db})^{-1})(\bm P_\gamma) = (- 2o_\gamma - 2p_\gamma + p_\gamma - p_0)X,
\]
which vanishes modulo $!P$.
\end{proof}

\section{Special covariant bijections}

Below we shall primarily concentrate on the finer categories with $\RV$-volume forms. It is a routine matter to translate the results and the arguments for the categories with $\Gamma$-volume forms, often simply by forgetting the conditions that involve the $\RV$-Jacobian operator $\jcb_{\RV}$. Nevertheless, there are shortcuts through forgetful functors, and we will supply more details after the key propositions in Remarks~\ref{gL:sur:c} and \ref{kernel:gL}.

Our  task in this section is to connect $\gsk \mVF[*]$ with $\gsk \mRV[*]$, or more precisely, to establish a surjective homomorphism $\gsk \mRV[*] \fun \gsk \mVF[*]$ and a specialized version of it. Notice the direction of the arrow. The main instrument in this endeavor  is special covariant transformations.


%
%

\subsection{Covariance and invariance}

\begin{defn}\label{defn:binv}
For each element $\gamma \in \absG_\infty$, let
$\pi_\gamma: \VF \fun \VF / \MM_{\gamma}$
be the natural  map. If $\gamma = (\gamma_1, \ldots, \gamma_n) \in \absG_\infty^n$ then $\pi_{\gamma}$ denotes the product of the maps $\pi_{\gamma_i}$. Let $\alpha \in \absG_\infty^n$ and $\beta \in \absG_\infty^m$. We say that a function $f :A \fun B$ with $A \sub \VF^n$ and $B \sub \VF^m$  is \emph{$(\alpha, \beta)$-covariant} if it $(\pi_{\alpha}, \pi_{\beta})$-contracts to a function $f_\downarrow : \pi_{\alpha}(A) \fun \pi_{\beta}(B)$, that is, $\pi_{\beta} \circ f = f_\downarrow \circ \pi_{\alpha}$. For simplicity, we shall often suppress mention of parameters and refer to $(\alpha, \beta)$-covariant functions as $(\alpha, -)$-covariant or $(-,\beta)$-covariant or just  covariant functions. A set $A \sub \VF^n$ is \emph{$\alpha$-invariant} if its characteristic function is $(\alpha, 0)$-covariant.

More generally, for sets $A$, $B$ with $\RV$-coordinates, the function $f : A \fun B$ is \emph{covariant} if every one of its $\VF$-fibers $f_{t}$   is $(\alpha_{t}, \beta_{t})$-covariant  for some $(\alpha_{t}, \beta_{t}) \in \absG_\infty$  (this is in line with Terminology~\ref{rvfiber}). Accordingly, a set is \emph{invariant} if (every $\VF$-fiber of) its characteristic function is $(-, 0)$-covariant.
\end{defn}

\begin{ter}\label{prop:pseu}
Of course  $(\infty, \infty)$-covariance is a vacuous requirement on a definable function, as $\infty$-invariance on a definable set; we shall call them \emph{pseudo} covariance and \emph{pseudo} invariance. On the other hand, we say that a definable function $f$ on $A$ is \emph{proper covariant} if
\begin{itemize}
  \item the sets $A_{\VF}$, $f(A)_{\VF}$ are bounded and the sets $A_{\RV}$, $f(A)_{\RV}$ are  doubly bounded,
  \item for each $\VF$-fiber $f_{t}$ of $f$ there is a $t$-definable tuple $(\alpha_t, \beta_t) \in \absG$ such that $f_t$ is $(\alpha_{t}, \beta_{t})$-covariant, $\dom(f_{t})$ is $\alpha_t$-invariant, and $\ran(f_{t})$ is $\beta_t$-invariant.
\end{itemize}

Actually, if $A_t$ is closed for every $t \in A_{\RV}$ then  there is no need to demand that $f(A)_{\RV}$ be doubly bounded, since it is implied by the other conditions: There is an $\go$-partition $p: A \fun \absG$ such that, for each open polydisc $\gb$ in question, $f(\gb)$ lies in the range of some $\VF$-fiber of $f$. This, by Lemma~\ref{vol:par:bounded}, yields a definable surjection $A \fun \abvrv(f(A)_{\RV})$ that satisfies the assumption of Lemma~\ref{invar:db}, and hence the set $f(A)_{\RV}$ is doubly bounded.

Observe also that if $f$ is proper covariant then, since $(A \times f(A))_{\RV}$ is  doubly bounded, by Lemma~\ref{db:to:db}, there is a definable element $\gamma \in \absG$ such that every $\VF$-fiber $f_{t}$ of $f$ is $(\gamma, \beta_{t})$-covariant (but not necessarily $(\gamma, \gamma)$-covariant) and $\dom(f_{t})$,  $\ran(f_{t})$ are both $\gamma$-invariant.

We are only interested in proper and pseudo covariant functions, and no covariant functions other than these shall be considered. Accordingly, an invariant set $A$ is either pseudo or proper, where $A$ is \emph{proper} if its projection into the $\RV$-coordinates is proper covariant, in particular, $A_{\VF}$ is bounded and $A_{\RV}$ is doubly bounded (there are various ways to interpret properness when the range of a function has no $\VF$-coordinates, and the task of choosing one is left to the reader). It follows from the discussion above that if $A$ is proper invariant then $A$ is $\gamma$-invariant for all sufficiently large definable elements $\gamma \in \absG$.
\end{ter}

\begin{lem}\label{inv:crit}
Let $A$ be a definable set such that $A_{\VF}$ is bounded and $A_{\RV}$ is doubly bounded. Then $A$ is proper invariant if and only of $A$ is clopen (recall Terminology~\ref{rvfiber}).
\end{lem}
\begin{proof}
The ``only if'' direction is clear. For the ``if'' direction, since every $\VF$-fiber of $A$ is open, there is an  $\go$-partition $p: A \fun \absG$. Since every $\VF$-fiber of $A$ is also closed, by Lemma~\ref{vol:par:bounded}, $p(A)$ is doubly bounded. The claim follows.
\end{proof}

\begin{lem}\label{int:inv}
Let $A$, $B$ be proper invariant subsets of $\VF^n \times \RV^m$. Then $A \cap B$, $\RVH(A) \mi A$, and $B \mi A$ are also proper invariant.
\end{lem}
\begin{proof}
Clearly $(A \cap B)_{\VF}$ is bounded, $(A \cap B)_{\RV}$ is doubly bounded, and every $\VF$-fiber of $A \cap B$ is clopen; similarly for the other cases. So this follows from Lemma~\ref{inv:crit}.
\end{proof}

Note that, however, a restriction of a proper covariant function is not necessarily a proper covariant function, even if the restricted domain and range are both proper invariant. In a sense, this behavior requires continuity.

\begin{lem}\label{conti:covar}
Let $f : A \fun B$ be a definable continuous surjection between two proper invariant sets. Suppose that  the range of every $\VF$-fiber of $f$ is clopen.  Then $f$ is proper covariant. If $f$ is bijective then the same holds under the weaker assumption that  the range of every $\VF$-fiber of $f$ is open.
\end{lem}
\begin{proof}
For each $(b, s) \in B_{\VF} \times B_{\RV}$ and every $\VF$-fiber $f_{(t,s)}$ of $f$, let $\beta_t \in \absG$ be a $(b,s,t)$-definable element such that
\begin{equation*}
\begin{cases}
  \go((b,s), \beta_t) \sub \ran(f_{(t,s)}), & \text{if } (b, s) \in \ran(f_{(t,s)}),\\
  \go((b,s), \beta_t) \cap \ran(f_{(t,s)}) = \0, & \text{if } (b, s) \notin \ran(f_{(t,s)}).
\end{cases}
\end{equation*}
Since $A_{\RV}$ is doubly bounded, by Lemma~\ref{db:to:db}, there is a $(b,s)$-definable $\beta \in \absG$ which satisfies this condition for every $t \in A_{\RV}$; let $\beta_{(b,s)} \in \absG$ be the smallest of such elements, which exists since the $\Gamma$-sort is \omin-minimal. Let $q: B \fun \absG$ be the definable function given by $(b,s) \efun \beta_{(b,s)} $, which is clearly an $\go$-partition. Since $B$ is proper invariant, by Lemmas~\ref{vol:par:bounded} and \ref{inv:crit}, we see that $q(B)$ is doubly bounded by, say, $\beta \in \absG^+$. Then, by continuity, there is another $\go$-partition $p: A \fun \absG$ such that, for each open polydisc $\gb$ in question, $f(\gb)$ lies in an open polydisc of radius $\beta$, and $p(A)$ is doubly bounded as well. It follows that $f$ is  proper covariant.

For the second claim, since $f$ is bijective, there is only one $\VF$-fiber $f_{(t,s)}$ of $f$ with $(b, s) \in \ran(f_{(t,s)})$. If we let $\beta_{(b,s)} \in \absG$ be the smallest element such that $\go((b,s), \beta_{(b,s)}) \sub \ran(f_{(t,s)})$ then the resulting function $q$ is still an $\go$-partition. The rest of the argument is the same.
\end{proof}

\begin{cor}\label{cov:rest}
Let $f : A \fun B$ be a continuous proper covariant bijection. If $A'$ is a proper invariant subset of $A$ then $f \rest A'$ is proper covariant.
\end{cor}
\begin{proof}
Since the domain and the range of every $\VF$-fiber $f_t$ of $f$ are  clopen, by Lemmas~\ref{inv:crit} and~\ref{conti:covar}, we only need to show that $f(A')$ is proper invariant, or equivalently, $f(A')$ is clopen. Since $\dom(f_t) \cap A'$ is clopen and $f$ is continuous, $\ran(f_t) \cap f(A')$ is clopen as well and hence, by Lemma~\ref{inv:crit} again, is proper $\alpha_t$-invariant for some $t$-definable $\alpha_t \in \absG$. By Lemma~\ref{db:to:db}, we may choose the same definable $\alpha$ for every $t$. The claim follows.
\end{proof}

\begin{conv}\label{conv:can}
We reiterate (a variation of) \cite[Convention~5.1]{Yin:tcon:I} here,  since this trivial-looking convention is actually quite crucial for understanding the discussion below, especially the parts that involve special covariant bijections. For each pseudo invariant or doubly bounded proper invariant set $A$, let
\[
\can(A) = \{(a, \rv(a), t) : (a, t) \in A \text{ and } a \in A_{\VF}\}.
\]
The natural bijection $\can : A \fun \can(A)$ is called the \emph{regularization} of $A$. We shall tacitly substitute $\can(A)$ for $A$  if it is necessary or is just more convenient. Whether this substitution has been performed or not should be clear in context (or rather, it is always performed).
\end{conv}


\begin{defn}[Special covariant transformations]\label{defn:special:bijection}
Let $A$ be an invariant set  or, more generally, a finite disjoint union of such sets, all doubly bounded if proper. Suppose that the first coordinate of $A$ is a $\VF$-coordinate (of course nothing is special about the first $\VF$-coordinate, we choose it simply for ease of notation).

Let $C \sub \RVH(A)$ be an $\RV$-pullback. Let
$\lambda: \pr_{>1}(C \cap A) \fun \VF$
be a definable function whose graph is contained in $C$, that is, for each $\RV$-polydisc $\gp \sub C$, $\lambda$ restricts to a function
\[
\lambda_{\gp} : \pr_{>1}(\gp \cap A) \fun \pr_{1}(\gp).
\]
Suppose that there is an $\rv(\gp)$-definable tuple $(\alpha_{\gp}, \beta_{\gp}) \in \absG_\infty$ such that
\begin{itemize}
  \item $\gp \cap A$ is $(\alpha_{\gp}, \beta_{\gp})$-invariant and $\lambda_{\gp}$ is  $(\alpha_{\gp}, \beta_{\gp})$-covariant,
  \item $(\alpha_{\gp}, \beta_{\gp}) = \infty$ for all $\gp$ if $A$ is pseudo and $(\alpha_{\gp}, \beta_{\gp}) \in \absG$ for all $\gp$ if $A$ is proper.
\end{itemize}

Let $\gamma \in \absG_\infty$ be a definable nonnegative element such that $\gamma = \infty$ if and only if $\beta_{\gp} = \infty$ for all $\gp$. We also assume that, for all $\gp$,
\[
\gamma_{\gp} \coloneqq \rad(\pr_{1}(\gp)) + \gamma \geq \beta_{\gp};
\]
the existence of such a $\gamma$ is guaranteed by Lemma~\ref{db:to:db}. Choose a definable point $t \in \gamma^\sharp$ and, for each $\gp$,  set
\[
t_{\gp} = t \cdot \rv(\pr_{1}(\gp)) \in \RV_0.
\]
Then the \emph{centripetal transformation $\eta$ on $A$ with respect to $\lambda$} is given by
\begin{equation*}\tag{CT}\label{centri}
\begin{cases}
  \eta (a, x) = (a - \lambda(x), x), & \text{if } (a, x) \in \gp \cap A \text{ and } \pi_{\gamma_{\gp}}(a) \neq \pi_{\gamma_{\gp}}(\lambda(x)),\\
  \eta (a, x) = (t_\gp, x), & \text{if } (a, x) \in \gp \cap A \text{ and } \pi_{\gamma_{\gp}}(a) = \pi_{\gamma_{\gp}}(\lambda(x)),\\
  \eta = \id, & \text{on } A \mi C.
\end{cases}
\end{equation*}
The function $\lambda$ is referred to as the \emph{focus} of $\eta$, the $\RV$-pullback $C$ as the \emph{locus} of $\lambda$ (or $\eta$), and the pair $(\gamma, t)$ as  the \emph{aperture} of $\lambda$ (or $\eta$).

Note that if $(\gamma, t)$ is the aperture of $\lambda$ then every other pair $(\gamma', t')$
of this form with $\gamma' \geq \gamma$ could be an aperture of $\lambda$ as well, so the aperture of $\lambda$ must be given as a part of $\lambda$ itself. Actually, all the data above should be regarded as part of $\lambda$, including the tuples $(\alpha_{\gp}, \beta_{\gp})$.

A \emph{special covariant transformation} $T$ on $A$ is an alternating composition of centripetal transformations and regularizations such that either all of the apertures are $(\infty, 0)$, in which case $T$ is called \emph{pseudo}, or none of the apertures is $(\infty, 0)$, in which case $T$ is called \emph{proper} (regularizations are usually not shown). The \emph{length} of such a special covariant transformation $T$, denoted by $\lh(T)$, is the number of centripetal transformations in it.



Choose a definable point $c \in t^\sharp$. If $(a, x) \in \gp \cap A$ and $\pi_{\gamma_{\gp}}(a) = \pi_{\gamma_{\gp}}(\lambda(x))$ for some $a \in \VF$ then $(\lambda(x), x) \in \gp \cap A$. Thus  the second clause of (\ref{centri}) may be ``lifted'' to
\[
\eta (a, x) = (a - \lambda(x)(1 - c), x).
\]
The images from the first two clauses of (\ref{centri}) may now overlap, but we take their disjoint union and thereby always assume that the resulting function $\eta^\flat$ is injective. In so doing, every special covariant transformation may be \emph{lifted} to a \emph{special covariant bijection} $T^\flat$ on $A$. This of course depends on the choice of the point $c$.
The image of $T^\flat$ is often denoted by $A^{\flat}$.
\end{defn}

This definition of a special covariant bijection is somewhat more complicated than  that of a special bijection in \cite[Definition~5.2]{Yin:tcon:I}. Clearly a special covariant bijection is a special bijection if all of the apertures are $(\infty, 0)$ or the second clause of (\ref{centri}) does not occur. The extra generality is needed to achieve better control of proper invariant sets.

\begin{rem}\label{spec:addendum}
Suppose that $A$ is proper $\delta$-invariant but is not doubly bounded. Then, by definition, there  are no special covariant transformations on $A$. But we can prepare $A$ as follows. Choose a definable point $t_\delta \in \delta^\sharp$. For each open polydisc $\gb_1 \times \ldots \times \gb_n \times t \sub A$ of radius $\delta$, replace each $\gb_i = \MM_\delta$ with $t_\delta$. The resulting definable set, or rather, the resulting disjoint union of definable sets, is doubly bounded and remains $\delta$-invariant. We may and do regard this operation as a centripetal transformation with respect to the constant focus map $0$.
\end{rem}

\begin{rem}\label{bigger:aper}
Let $\eta$ be a proper centripetal transformation on $A$ of aperture $(\gamma, t)$. Then the set $\eta(A)$ is indeed proper invariant and doubly bounded. If we change the aperture of $\eta$ to $(\gamma', t')$ with $\gamma' > \gamma$ and write the resulting  centripetal transformation as $\eta'$ then, in the notation of Definition~\ref{defn:special:bijection}, every open polydisc $\gb \sub \eta(A) \mi \eta'(A)$ of radius $\alpha_\gp$, where $\gp$ is the $\RV$-polydisc that contains $\eta^{-1}(\gb)$, has an extra $\RV$-coordinate $t_\gp$ that is contracted from open discs of radius $\gamma_\gp$ in the same $\VF$-coordinate of $A$. Each of these open polydiscs has a counterpart in $\eta'(A) \mi \eta(A)$, in which $t_\gp$ is replaced by  $(\MM_{\gamma_\gp} \mi \MM_{\gamma'_\gp}) \cup t'_{\gp}$. On the other hand, $\eta(A) = \eta'(A)$ if and only if  the second clause of (\ref{centri}) does not occur.
\end{rem}

Observe that if $(A,\omega)$, $(B, \sigma)$ are objects of $\mVF[k]$ and $F : A \fun B$ is a special covariant bijection then $\jcb_{\VF} F(x) = 1$ for all $x \in A$ outside a definable subset of $\VF$-dimension less than $k$, and hence $F$ represents an isomorphism if $\omega(x) = \sigma(F(x))$ for all $x \in A$ outside a definable subset of $\VF$-dimension less than $k$ (recall Remark~\ref{mor:equi}).

\begin{lem}[{\cite[Corollary~5.6]{Yin:tcon:I}}]\label{special:bi:term:constant}
Let $A \sub \VF^n$ be a definable set and $f : A \fun \RV^m$ a definable function. Then there is a special bijection $T$ on $A$ such that $A^\flat$ is an $\RV$-pullback and the function $f \circ T^{-1}$ is $\rv$-contractible.
\end{lem}

\begin{rem}\label{special:pseudo}
The proof of Lemma~\ref{special:bi:term:constant} in \cite{Yin:tcon:I} actually shows that for every definable set $A$, there is a special bijection $T$ on  $\RVH(A)$  such that $A^{\flat}$ is an $\RV$-pullback. In the present context, for reasons that will become clear, we would like to extend this result, using proper special covariant bijections  on proper invariant sets. This is not guaranteed by Lemma~\ref{special:bi:term:constant} since the focus maps in $T$ are not required to be (suitably) covariant within each $\RV$-polydisc, except when $A_{\VF} \sub \VF$, in which case  the covariance requirement is half vacuous and it is easy to see how to turn $T$ into a proper special covariant bijection whose components all have the same aperture (for more details see the FMT procedure in Terminology~\ref{FMT} below).
\end{rem}



\begin{prop}\label{simplex:with:hole:rvproduct}
Let $H$ be a proper invariant $\RV$-pullback and $(A_i)_i$ a definable finite partition of $H$ such that each $A_i$ is proper invariant. Then there is a proper special covariant bijection on $H$ such that
\begin{itemize}
  \item its focus maps and hence its components  are all continuous,
  \item its restriction to each $A_i$ is indeed a proper special covariant bijection,
  \item every $A_i^{\flat} \sub H^{\flat}$  is a doubly bounded $\RV$-pullback.
\end{itemize}
\end{prop}
\begin{proof}
To begin with, by Remark~\ref{spec:addendum}, we may assume that $H$ is doubly bounded. It is equivalent and less cumbersome to construct a proper special covariant transformation such that its lift is as required (see the last paragraph of this proof). To that end, we proceed by induction on $n$, where $H_{\VF} \sub \VF^n$. The base case $n=1$ follows from the discussion in Remark~\ref{special:pseudo}, and continuity is obvious.

For the inductive step, let $A_{i1} = \pr_1(A_i)$ and $H_1 = \pr_1(H)$. By Lemma~\ref{int:inv}, we may assume that the sets $A_{i1}$ form a partition of $H_1$. For each $a \in H_1$, the inductive hypothesis gives an $a$-definable proper special covariant transformation $T_a$ on $H_a$ with continuous focus maps such that each $T_a \rest A_{ia}$ is a proper special covariant transformation and each $T_a(A_{ia}) \sub T_a(H_a)$ is an $\RV$-pullback. Our goal then is to fuse together these transformations $T_a$ so to obtain one proper special covariant transformation on $H$ as desired. This is in general not possible without first modifying $H_1$ in a suitable way, which constitutes the bulk of the work below.

There is an element $\delta \in \absG$ such that every $A_i$ is proper $\delta$-invariant. Let $U_{ak}$ enumerate the loci of the components of $T_{a}$, $\lambda_{ak}$ the corresponding continuous focus maps, and $(\gamma_{ak}, t_{ak})$ their apertures; for each $\RV$-polydisc $\gq \sub U_{ak}$, the map $\lambda_{ak\gq}$ is $(\alpha_{ak\gq}, \beta_{ak\gq})$-covariant. By compactness, there is a definable set $V \sub \VF \times \RV^l$ such that $\pr_1(V) = H_1$ and, for each $a \in H_1$, the set $V_a$ contains the following $\RV$-data of $T_a$:
\begin{itemize}
  \item $\rv(T_{a} \rest A_{ia})$, $\rv(T_{a})$, and the sequence $\rv(U_{ak})$,
  \item the $\VF$-coordinates targeted by the focus maps $\lambda_{ak}$,
  \item the $a$-definable apertures $(\gamma_{ak}, t_{ak})$,
   \item the $(a, \rv(\gq))$-definable tuples $(\alpha_{ak\gq}, \beta_{ak\gq})$;
\end{itemize}
the set  $\rv(T_{a})$ is determined by other data in this list and hence is redundant, but we add it in anyway for clarity. Note that $V$ is not necessarily proper invariant.

Let $\phi(x, y)$ be a quantifier-free formula that defines $V$ and $G_{i}(x)$ enumerate its top $\lan{T}{}{}$-terms (recall Convention~\ref{topterm}). By Lemma~\ref{special:bi:term:constant}, there is a special bijection $R : H_1 \fun H_1^\flat$ such that each $A_{i1}^\flat \sub H_1^\flat$ is an $\RV$-pullback and every $G_{i} \circ R^{-1}$ is $\rv$-contractible. This implies that,  for every $\RV$-polydisc $\gp \sub H_1^\flat$, the $\RV$-data $V_a$ is constant over $a \in R^{-1}(\gp)$. Observe that, since $H_1$ is an $\RV$-pullback of $\RV$-fiber dimension $0$, by weak \omin-minimality and Lemma~\ref{RV:fiber:dim:same}, each focus map in $R$ consists of only finitely many points. Then there are finitely many definable open discs $\ga_j \sub H_1$ of radius $\delta$ such that the restriction of $R$ to $H_1 \mi \bigcup_j \ga_j$ is actually a proper special covariant bijection --- the reason simply being that, after deleting all the discs $\ga_j$, each focus map in $R$ lies outside the set in question. Each $\ga_j$ contains a definable point $a_j$ and, for all $a, a' \in \ga_j$ and every $A_i$, we have $A_{ia} = A_{ia'}$ (because  $A_i$ is $\delta$-invariant). It follows that, over each $\ga_j$, we can use the same  transformation $T_{a_j}$ to achieve the desired effect, and then adjust $R$ so that every $\ga_j$ is mapped to the same $\RV$-disc $t^\sharp_{\delta}$, where $t_{\delta} \in \delta^\sharp$ is definable.

\begin{ter}\label{FMT}
This operation just performed --- modifying a construction over finitely many definable open discs with a chosen definable point in each of them --- will be applied many more times below (even after this proof) and, to reduce verbosity, we shall refer to it by the acronym FMT. Which construction it is being applied to should be clear in context.
\end{ter}

Therefore, we may assume that $R$ is a (continuous) proper special covariant bijection whose components all have the same  aperture $(\delta, t_{\delta})$.

By compactness, there is  a definable finite partition of $H_1$ such that, over each piece, the  focus maps $\lambda_{ak}$ are uniformly defined by formulas $\lambda_k(a, y, z)$. By HNF, there are only finitely many open discs $\ga_i \sub H_1$ of radius $\delta$ that are split by this partition. Thus, by FMT, we may assume that the partition is indeed trivial. Since $R$ induces a continuous proper special covariant bijection on $H$, we may actually assume that $R$ is trivial as well.


Over each $t^\sharp \sub H_1$, we can now write $U_{ak}$ as $U_{tk}$, $\alpha_{ak\gq}$ as $\alpha_{k\gq}$ (the first $\RV$-coordinate of $\gq$ is actually $t$), and so on. We are almost ready to fuse together the transformations $T_a$ over $a \in t^\sharp$. The remaining problem  is that, for any $a, a' \in t^\sharp$, although the two focus maps $\lambda_{a1\gq}$, $\lambda_{a'1\gq}$ are both $(\alpha_{1\gq}, \beta_{1\gq})$-covariant,  the images of the same open polydisc of radius $\alpha_{1\gq}$ may lie in two distinct open discs of radius $\beta_{1\gq}$. To solve this problem, consider an open polydisc $\gp$ of radius $\alpha_{1\gq}$ that is contained in $\dom(\lambda_{a1\gq})$ for some (hence all) $a \in t^\sharp$. For each $b \in \gp$, let $\lambda_{1b} : t^\sharp \fun \VF$ be the function defined by $\lambda_1(x, b, z)$. By monotonicity and Lemma~\ref{open:pro}, there are a $b$-definable finite set $C_b \sub t^\sharp$ and, for any $a \in t^\sharp \mi C_b$, an open disc $\ga_a \sub t^\sharp \mi C_b$ around $a$ such that $\lambda_{1b}(\ga_a)$ lies in an open disc of radius $\beta_{1\gq}$. Since for any other $a' \in \ga_a$,  $\lambda_{a'1\gq}(\gp)$ also lies in an open disc of radius $\beta_{1\gq}$, we see that $\lambda_1(\ga_a \times \gp)$ lies in an open disc of radius $\beta_{1\gq}$, where $\lambda_1$ stands for the function defined by $\lambda_1(x, y, z)$. Therefore, we may assume that the finite set $C_b$ is actually $\code \gp$-definable. But then, by Lemma~\ref{ima:par:red}, it is even definable. By compactness and FMT, we may assume that there is an $\go$-partition $p: \dom(\lambda_1) \fun \absG$ such that, for each open polydisc $\gb$ in question,  $\lambda_1(\gb)$ lies in a disc of radius $\beta_{1\gq}$, where $\gq$ is related to $\gb$ in the obvious way. By Lemma~\ref{vol:par:bounded}, the image of $p$ is doubly bounded and, by Lemma~\ref{db:to:db}, there is a definable $\gamma_1 \in \absG$ with $\gamma_1 \geq \gamma_{t1}$ for all $t \in \rv(H_1)$, which means that $\lambda_1$ can serve as the focus map of a centripetal transformation $T_1$ on $H$ of aperture $(\gamma_{1}, t_{1})$ for some definable $t_1 \in \gamma^\sharp_1$. By Lemma~\ref{fiber:conti} and FMT, we may assume that $\lambda_1$ is continuous.

At this point the proof would be complete if we could repeat the procedure above for $\lambda_2$, and so on. We still have a small issue, namely some part of the locus $U_{a2}$ may have disappeared because the aperture of  $\lambda_{a1}$ is bumped up to $(\gamma_{1}, t_{1})$; see Remark~\ref{bigger:aper}. It is not hard to see that the inductive hypothesis may be applied to the $\RV$-pullback contained in $T_1(H)$ that corresponds to the missing locus, since it has one less $\VF$-coordinate and its intersection with each $T_1(A_i)$ is proper invariant.
\end{proof}


\begin{ques}\label{ques:prop}
Does Lemma~\ref{fiber:conti} still hold in some analogous form if we replace ``continuous'' with ``differentiable'' therein? If so, the same upgrade also applies to Proposition~\ref{simplex:with:hole:rvproduct}.
\end{ques}

\subsection{Lifting from $\RV$ to $\VF$}

\begin{defn}[Lifting maps]\label{def:L}
Let $U$ be a set in $\RV$ and $f : U \fun \RV^k$ a function. Set $U_f = \bigcup \{f(u)^\sharp \times u: u \in U\}$. The \emph{$k$th lifting map} $\mathbb{L}_k: \RV[k] \fun \VF[k]$ is given by $(U,f) \efun U_f$.
The lifting map $\mathbb{L}_{\leq k}: \RV[{\leq} k] \fun \VF[k]$ is given by $\bigoplus_{i} \bm U_i \efun \biguplus_{i} \bb L_i \bm U_i$.
Set $\mathbb{L} = \bigcup_k \mathbb{L}_{\leq k}$.

The \emph{$k$th lifting map} ${\mu}\mathbb{L}_k: \mRV[k] \fun \mVF[k]$ is given by $(\bm U, \omega) \efun (\bb L \bm U, \bb L \omega)$, where $\bb L \omega$ is the obvious function on $\bb L \bm U$ induced by $\omega$. Set ${\mu}\bb{L} = \bigoplus_k \mu\mathbb{L}_{k}$, and similarly for ${\mu_\Gamma}\bb{L}$.
\end{defn}

Note that if $\bm U \in \RV^{\db}[k]$ then $\bb L \bm U$ is a doubly bounded proper invariant set.

\begin{rem}\label{nota:form:down}
Let $A$ be a definable set and $\omega : A \fun \RV$ a definable function. Denote the set $\set{( a, \omega( a)) :  a \in A)}$ by $A_{\omega}$. For ease of notation, the function $A_{\omega} \fun \RV$ induced by $\omega$ is still denoted by $\omega$. The definable pair $(A, \omega)$ is called \emph{proper invariant} if $A$ is proper invariant and $\omega$ is proper covariant. In that case, by Lemmas~\ref{vol:par:bounded} and \ref{invar:db} (applied as in the middle of Terminology~\ref{prop:pseu}), the set $\omega(A)$ is doubly bounded and hence $(A_{\omega}, \omega)$ is proper invariant as well. If $(A, \omega) \in \mVF[k]$  then clearly the two objects $(A, \omega)$, $(A_{\omega}, \omega)$ are isomorphic and if, in addition, $(A, \omega)$ is proper invariant then the obvious isomorphism in question and its inverse are proper covariant.
\end{rem}

\begin{lem}\label{L:measure:surjective}
Let $A$ be an $\RV$-pullback and $(A, \omega) \in \mVF[k]$. Let $(A_i)_i$ be a definable finite partition of $A$ and $\bm A_i = (A_i, \omega \rest A_i) \in \mVF[k]$. Then there is a special covariant bijection $T$ on $A_{\omega}$ that induces morphisms of the form $\bm A_i \fun \mL(\bm U_i, \pi_i)$. Moreover, if every $\bm A_i$ is proper invariant then $T$ is as given in Proposition~\ref{simplex:with:hole:rvproduct} and hence $(\bm U_i, \pi_i) \in \mRV^{\db}[k]$.
\end{lem}
\begin{proof}
By Remark~\ref{special:pseudo}, Proposition~\ref{simplex:with:hole:rvproduct}, and Lemma~\ref{RV:fiber:dim:same}, there is a special covariant bijection $T$ on $A_{\omega}$ such that for every $i$, the restriction $T \rest A_{i \omega}$ is a special covariant bijection and $B_i \coloneqq A_{i\omega}^\flat$ is an $\RV$-pullback with $\rv(B_i)_{\leq k} \in \RV[{\leq}k]$ (recall Notation~\ref{0coor}). Moreover, if every $\bm A_i$ is proper invariant then $T$ is as given in Proposition~\ref{simplex:with:hole:rvproduct}, which implies that $\rv(B_i)_{\leq k} \in \RV^{\db}[k]$. Let $\bm U_i \in \RV[k]$ be the $k$th component of $\rv(B_i)_{\leq k}$ and $\sigma_i = (\omega \circ T^{-1}) \rest \mathbb{L} \bm U_i$. Then $\sigma_i$ is constant on every $\RV$-polydisc in $\mathbb{L} \bm U_i$ and hence induces a function $\pi_i : \bm U_i \fun \RV$. Since $\bm A_i$ is isomorphic to $(\mathbb{L} \bm U_i, \sigma_i)$, we see that $(\bm U_i, \pi_i)$ is as required.
\end{proof}

This lemma may be applied to any proper invariant object $(A, \omega) \in \mVF[k]$ without the partition of $\RVH(A)$ being explicitly given since, by  Lemma~\ref{int:inv},   $A' = \RVH(A) \mi A$ is proper invariant too, which means that  the two objects $(A, \omega)$,  $(A', 1)$ are indeed as assumed.

\begin{lem}\label{RVlift}
Let $f : U \epi V$ be a definable surjection between two sets in $\RV$. Then there is a definable differentiable function $f^\sharp : U^\sharp \fun V^\sharp$ that $\rv$-contracts to $f$. Furthermore, if $U$, $V$ are both subsets of $\RV^k$ and $f$ is a bijection then $f^\sharp$ is bijective as well.
\end{lem}
\begin{proof}
We do induction on $n = \dim_{\RV}(U)$. If $n=0$ then $U$ is finite and hence, for every $u \in U$, the $\RV$-polydisc $u^\sharp$ contains a definable point, similarly for $V$, in which case how to construct a definable differentiable function $f^{\sharp}$ as desired is obvious. Note that if $U$, $V$ do not sit inside the same ambient space then it is impossible for $f^{\sharp}$ to be a bijection.

For the inductive step, by an easy induction on $\dim_{\RV}(U) + \dim_{\RV}(V)$ and weak \omin-minimality, there are a definable finite partition $(U_i)_i$ of $U$ and injective coordinate projections
\[
\pi_i : U_i \fun \RV^{n_i}, \quad \pi'_i : V_i \fun \RV^{n'_i},
\]
where $V_i = f(U_i)$, $\dim_{\RV}(U_i) =  n_i$, and $\dim_{\RV}(V_i) = n'_i$; the obvious surjection $\pi_i(U_i) \fun \pi'_i(V_i)$ induced by $f$ is denoted by $f_i$. Moreover,  if $f$ is a bijection then $n_i = n'_i$ and  each $f_i$ is a bijection as well. Let $g_i$, $h_i$ be the obvious functions whose graphs are $U_i$, $V_i$, respectively. Observe that if all these functions $f_i$, $g_i$, $h_i$ can be lifted as desired then, by the construction in the base case above, $F$ can be lifted as desired as well. So, without loss of generality, we may assume that $U$ is a subset of $\RV^n$.

For all $u \in U$ and every $a \in u^\sharp$, the $\RV$-polydisc $f(u)^\sharp$ contains an $a$-definable point and hence, by compactness, there is a definable function $f^{\sharp} : U^\sharp \fun V^\sharp$ that $\rv$-contracts to $f$. There is a definable set $A \sub U^\sharp$ with $\dim_{\VF}(A) < n$ such that $f^{\sharp} \rest (U^\sharp \mi A)$ is differentiable.  By Lemma~\ref{RV:bou:dim}, $\dim_{\RV}(\rv(A)) < n$. So the first claim follows from the inductive hypothesis. For the second claim, since $V$  is also a subset of $\RV^n$, by Lemma~\ref{RV:bou:dim} again, $\dim_{\RV}(\partial_{\RV}f^{\sharp}(U^\sharp)) < n$ and hence, by the inductive hypothesis, we may assume that $f^{\sharp}$ is surjective. Then there is a definable function $g : V^\sharp \fun U^\sharp$ such that $f^{\sharp}(g(b)) = b$ for all $b \in V^\sharp$. It follows that we may further assume that $g$ is a surjection too, which just means that $f^{\sharp}$ is a bijection as desired.
\end{proof}

For the next few lemmas, let $F: (U, f) \fun (V, g)$ be an $\RV[k]$-morphism. For simplicity and without loss of generality, we assume that $f$, $g$ are both the identity functions and hence the finite-to-finite correspondence $F^\dag \sub U \times V$ is just $F$ itself (recall Remark~\ref{fin:dec}). If $F^{\sharp} : U^\sharp \fun V^\sharp$ is a definable bijection that $\rv$-contracts to $F$ then it is called a \emph{lift} of $F$. By Convention~\ref{conv:can}, we shall think of such a lift  as a definable bijection $\bb L U \fun \bb L V$ that $\rv$-contracts to $F$.

\begin{lem}\label{RV:iso:class:lifted:jcb}
Suppose that $F^{\sharp}$ is a lift of $F$. Then for all $u \in U$ away from a definable subset of $\RV$-dimension less than $k$ and all $a \in u^\sharp$,
\begin{equation}\label{jac:mat}
\rv(\jcb_{\VF} F^{\sharp}( a)) = \jcb_{\RV} F(u).
\end{equation}
\end{lem}
\begin{proof}
This is  immediate by Lemma~\ref{rv:op:comm} if $U$, $V$ are sets in  $\K^+$. The general case follows from this special case, the definitions of the Jacobians, Lemma~\ref{dim:cut:gam}, and compactness.
\end{proof}

\begin{defn}\label{rela:unary}
We say that $F$ is \emph{relatively unary} or more precisely, \emph{relatively unary in the $i$th coordinate}, if $(\pr_{\tilde{i}} \circ F)(u) = \pr_{\tilde{i}}(u)$ for all $u \in U$, where $i \in [k]$.

Similarly, for objects $A \sub \VF^{n} \times \RV^{m}$, $B \sub \VF^{n} \times \RV^{m'}$ of $\VF_*$ and a  morphism $G : A \fun B$, we say that $G$ is \emph{relatively unary in the $i$th $\VF$-coordinate} if $(\pr_{\tilde{i}} \circ G)(x) = \pr_{\tilde{i}}(x)$ for all $x \in A$, where $i \in [n]$.  If $G \rest A_a$ is also a (proper) special covariant bijection for every $a \in \pr_{\tilde{i}} (A)$ then we say that $G$ is \emph{relatively (proper) special covariant} in the $i$th $\VF$-coordinate.
\end{defn}

Since identity functions are relatively unary in any coordinate, if a morphism is piecewise a composition of relatively unary morphisms then it is indeed a composition of relatively unary morphisms.

Clearly every (proper) special covariant bijection on $A$ is a composition of relatively (proper) special covariant bijections (but not vice versa).

\begin{lem}\label{bijection:partitioned:unary}
Every morphism in $\RV[k]$ can be written as a composition of relatively unary morphisms, similarly in $\mRV[k]$, $\RV^{\db}[k]$, $\mRV^{\db}[k]$, $\VF[k]$, and $\mVF[k]$.
\end{lem}
\begin{proof}
It is enough to show this piecewise, which is an easy consequence of weak \omin-minimality. In detail, for each $t \in \pr_{< k}(U)$ there are a $t$-definable finite partition $(U_{ti})_i$ of $U_t$ and  injective coordinate projections $\pi_i : F(U_{ti}) \fun \RV$. So, by compactness, there are a definable finite partition $(U_{i})_i$ of $U$, definable injections $F_i : U_i \fun \RV^k$, and $j_i \in [k]$ such that for all $u \in U_i$, $\pr_{< k}(u) = \pr_{< k}(F_i(u))$ and $\pr_{k}(F_i(u)) = \pr_{j_i}(F(u))$.
The claim now follows from a routine induction on $k$. If $F$ is indeed a $\mRV[k]$-morphism then it is easy to see how to equip each $F_i(U_i)$ with a volume form so that $F_i$  becomes a $\mRV[k]$-morphism, and so on. The other cases are rather similar.
\end{proof}

For an arbitrary proper covariant $\mVF[k]$-morphism,  Lemma~\ref{bijection:partitioned:unary} cannot guarantee that the relatively unary morphisms in question are also proper covariant.  However, this can be arranged for certain lifts, as follows.

\begin{rem}\label{rvdisc:stre}
Let $t, s \in \RV$ be definable. So the $\RV$-discs $t^\sharp$, $s^\sharp$ contain definable points. It follows that for any definable $c \in \VF^\times$ with $\vv(c) = \vrv(s/t)$, there is a definable bijection $f : t^\sharp \fun s^\sharp$ such that $\ddx f = c$.
\end{rem}

\begin{lem}\label{L:measure:class:lift}
For any $\mRV[k]$-morphism $F : (U, \omega) \fun (V, \sigma)$ there is a lift $F^\sharp: {\mu}\bb{L}(U, \omega) \fun {\mu}\bb{L}(V, \sigma)$ of $F$ that may be written  piecewise as a composition of differentiable relatively unary $\mVF[k]$-morphisms. In addition, if $F$ is a $\mRV^{\db}[k]$-morphism then (every component of) $F^\sharp$ is proper covariant.
\end{lem}
\begin{proof}
We do induction on $n = \dim_{\RV}(U)$. For the base case $n=0$, we may assume that $U$ is just a singleton and $F$ is relatively unary. Since  every $\RV$-disc involved contains a definable point, it is easy to lift $F$  as desired by applying Remark~\ref{rvdisc:stre} in the coordinate in question.

For the inductive step,  we may assume that both $\pr_{\leq n} \rest U$ and $\pr_{\leq n} \rest V$ are injective. Let $F' : \pr_{\leq n} (U) \fun \pr_{\leq n}(V)$ be the bijection induced by $F$. Observe that, by Lemma~\ref{bijection:partitioned:unary} and the inductive hypothesis, we may actually assume that $F'$ is  relative unary in, say, the $n$th coordinate. By the construction in the proof of Lemma~\ref{RVlift}, $F'$ can be lifted to a differentiable relatively unary bijection $F'^\sharp$ outside a definable subset of $\RV$-dimension less than $n$. It follows that if $n=k$ then, by Lemma~\ref{RV:iso:class:lifted:jcb}, the whole situation is reduced to the inductive hypothesis. So, without loss of generality, $n < k$ and $F'^\sharp$ is a lift of $F'$ such that the condition (\ref{jac:mat}) is satisfied everywhere. Let
\[
U' = \{(F'(u'), u'') : (u', u'') \in U\}
\]
and $\omega' : U' \fun \RV$ be the function given by
\[
(F'(u'), u'') \efun \omega(u) / \jcb_{\RV} F'(u'),
\]
where $u = (u', u'') \in U$. Then $(F', \id)$ is a $\mRV[k]$-morphism between $(U, \omega)$ and $(U', \omega')$, and $(F'^\sharp, \id)$ is a differentiable relatively unary $\mVF[k]$-morphism that lifts $(F', \id)$. So we are further reduced to the case $\pr_{\leq n} (U) = \pr_{\leq n}(V) = W$. By Lemma~\ref{RVlift}, there are definable differentiable functions that $\rv$-contract to the obvious surjections $W \fun \pr_{> n}(U)$, $W \fun \pr_{> n}(V)$. This means that $F$ can be lifted by applying Remark~\ref{rvdisc:stre} fiberwise as in the base case above.

The second claim is a consequence of Lemma~\ref{conti:covar}.
\end{proof}

\begin{rem}\label{invert:compose}
Let $\bm A_i= (A_i, \omega_i) \in \mVF[k]$, where $i = 1,2,3$, and suppose that $F : \bm A_1 \fun \bm A_2$ and $G : \bm A_2 \fun \bm A_3$ are continuous proper covariant $\mVF[k]$-morphisms. Observe that $\dom(F)$, $\ran(F)$ must be the topological interiors of $A_1$, $A_2$, respectively; similarly for $G$, and hence $\ran(F) = \dom(G)$. Let $H$ be the composition of $F$ and $G$, considered as a $\mVF[k]$-morphism $\bm A_1 \fun \bm A_3$. So every $\VF$-fiber of $H$ is a continuous bijection between two open (and closed) subsets of $\VF^k$. By Lemma~\ref{conti:covar}, $H$ is  proper covariant.

The same holds if we replace ``continuous'' by ``differentiable.''
\end{rem}

\begin{defn}\label{defn:diamond}
The subcategory $\mVF^{\diamond}[k]$ of $\mVF[k]$ consists of the proper invariant objects and the morphisms that are compositions of continuous proper covariant relatively unary morphisms; similarly for the subcategory $\mgVF^{\diamond}[k]$ of $\mgVF[k]$.
\end{defn}

\begin{rem}\label{dia:is:cat}
Obviously the composition law holds in $\mVF^{\diamond}[k]$, $\mgVF^{\diamond}[k]$ and hence they are indeed categories.  By Remark~\ref{invert:compose}, every morphism in them is a continuous proper covariant bijection --- as opposed to merely an essential bijection --- and, by Corollary~\ref{conti:homeo}, admits an inverse.  So these categories are already groupoids and there is no need to pass to a quotient category as in Remark~\ref{mor:equi}. On the other hand, Proposition~\ref{simplex:with:hole:rvproduct} shows that they do have nontrivial morphisms.
\end{rem}

Although the preceding lemmas speak of differentiable morphisms, we do not consider  categories with just such morphisms between proper invariant objects, primarily because we are unable to construct a ``differentiable'' version of Proposition~\ref{simplex:with:hole:rvproduct}; see Question~\ref{ques:prop}.

The main reason that we need lifts in Lemma~\ref{L:measure:class:lift} to be in that particular form is that we are forced to work with explicit compositions of proper covariant relatively unary morphisms (explicitness is not an issue if proper covariance is not demanded, see Lemma~\ref{bijection:partitioned:unary} and the remark thereafter), due to the failure of generalizing Lemma~\ref{simul:special:dim:1} below to higher dimensions; see the opening discussion of \cite[\S~5.2]{Yin:tcon:I} for further explanation.

\begin{cor}\label{L:sur:c}
The lifting maps ${\mu}\bb L_{k}$ induce surjective homomorphisms, which are often simply denoted by ${\mu}\bb L$, between the Grothendieck semigroups
\[
\gsk \mRV[k] \epi \gsk \mVF[k], \quad \gsk \mRV^{\db}[k] \epi \gsk \mVF^{\diamond}[k].
\]
\end{cor}
\begin{proof}
By Lemma~\ref{L:measure:class:lift}, every $\mRV[k]$-morphism $F$ can be lifted, and the lift is piecewise a composition of continuous proper covariant relatively unary morphisms between proper invariant objects if  $F$ is in $\mRV^{\db}[k]$. So ${\mu}\bb L_{k}$ induces a map on the isomorphism classes, which is easily seen to be a semigroup homomorphism. By Lemma~\ref{L:measure:surjective} (see the remark thereafter), this  homomorphism is surjective.
\end{proof}

\begin{rem}\label{gL:sur:c}
There are surjective semigroup homomorphisms
\[
\bfig
  \Square(0,0)/->>`->>`->>`.>>/<400>[{\gsk \mRV[k]}`{\gsk \mVF[k]}`{\gsk \mgRV[k]}`{\gsk \mgVF[k]};
  {\mu}\bb L```\mgL]
 \efig
\]
where the vertical ones are induced by the obvious forgetful functors $\bb F$  and are also denoted by $\bb F$; similarly for the doubly bounded $\RV$-categories and the proper invariant $\VF$-categories. It follows from the construction in the proof of Lemma~\ref{L:measure:class:lift}   that there are surjective semigroup homomorphisms $\mgL$ that make the squares commute.
\end{rem}

The forgetful functor $\bb F : \mVF[k] \fun \mgVF[k]$ is faithful, but obviously not full. Something weaker holds, though. Let $\bm A = (A, \omega)$, $\bm B = (B, \sigma)$ be objects of $\mgVF[k]$ and $f : \bm A   \fun \bm B$ a morphism. For every $\bm A^* = (A, \omega^*) \in \mVF[k]$ with $\bb F \bm A^* = \bm A$, it is easy to construct a  $\mVF[k]$-morphism $f^* : \bm A^* \fun \bm B^*$ with $\bb F \bm B^* = \bm B$ (but $\bb F f^*$ may or may not be $f$). This is also the case for the proper invariant $\VF$-categories.

\begin{lem}\label{forg:ful}
If $f$ is in $\mgVF^{\diamond}[k]$ then  $f^*$ can be found in $\mVF^{\diamond}[k]$.
\end{lem}

\begin{proof}
By the definition of $\mgVF^{\diamond}[k]$ and compactness, we may assume that $f$ is relatively unary in, say, the $k$th $\VF$-coordinate and $A$, $B$ are subsets of $\VF^k$.

By a routine application of Lemma~\ref{vol:par:bounded} as we have done above, there is a definable $\alpha \in \absG$ such that $A$, $B$ are $\alpha$-invariant and $\omega \rest \ga$, $(\sigma \circ f) \rest \ga$, $\sigma \rest \gb$ are  constant for open polydiscs $\ga \sub A$, $\gb \sub B$ of radius $\alpha$. We claim that there is a definable continuous bijection $f^* = (f^*_i)_i :  A \fun  B$, relatively unary in the $k$th $\VF$-coordinate, such that $\pard k f^*_k$ is continuous and for all $a \in A$,
\[
\vv(\pard k f^*_k(a)) = \omega(a) / \sigma(f^*(a)).
\]
To see this, observe that if $k = 1$ then, since $f$ is a $\mgVF^{\diamond}[1]$-morphism, it satisfies the assumption of Lemma~\ref{diff:dtdp} and hence the claim follows from Lemma~\ref{adjust:C1}. In general, by induction on $k$, for each $a \in \pr_1(A)$ there is an $a$-definable bijection $f_a : A_a \fun B_a$ as desired. Then the existence of $f^*$ follows from  Lemma~\ref{fiber:conti} and FMT (Terminology~\ref{FMT}), where the finitely many definable open discs in question are all of radius $\alpha$.

Now, by Corollary~\ref{conti:homeo} and Lemma~\ref{conti:covar}, $f^*$ is a proper covariant homeomorphism. Since $\pard k f^*_k$ is continuous,  there is a definable $\beta \in \absG$ such that for every open polydisc $\gb \sub B$ of radius $\beta$,
\[
(\omega^* \circ (f^*)^{-1}) \rest \gb \dand (\rv \circ \pard k f^*_k \circ (f^*)^{-1}) \rest \gb
\]
are both constant. Thus there is a proper covariant function $\sigma^* : B \fun \RV$ such that
\[
\omega^* = (\sigma^* \circ f^*) \cdot (\rv \circ \pard k f^*_k) \dand \vrv \circ \sigma^*  = \sigma.
\]
So $f^*$ and $(B, \sigma^*)$ are as desired.
\end{proof}


\begin{lem}\label{simul:special:dim:1}
Suppose that $f : A \fun B$ is a $\VF[1]$-morphism between two objects that have exactly one $\VF$-coordinate each.  Then there are two special bijections $T_A : A \fun A^{\flat}$, $T_B : B \fun B^{\flat}$ such that $A^{\flat}$, $B^{\flat}$ are $\RV$-pullbacks and the function $f^{\flat}_{\downarrow}$ is bijective in
the commutative diagram
\[
\bfig
  \square(0,0)/->`->`->`->/<600,400>[A`A^{\flat}`B`B^{\flat};
  T_A`f``T_B]
 \square(600,0)/->`->`->`->/<600,400>[A^{\flat}`\rv(A^{\flat})`B^{\flat} `\rv(B^{\flat});  \rv`f^{\flat}`f^{\flat}_{\downarrow}`\rv]
 \efig
\]
Thus $f^{\flat}$ is a lift of $f^{\flat}_{\downarrow}$, where the latter is regarded as an $\RV[{\leq}1]$-morphism between $\rv(A^{\flat})_{1}$ and $\rv(B^{\flat})_1$ (recall Notation~\ref{0coor}).

If, in addition, $f$ is a $\mVF^{\diamond}[1]$- or $\mgVF^{\diamond}[1]$-morphism (the volume forms are not displayed) then we may take $T_A$, $T_B$ to be (continuous) proper covariant and $A^{\flat}$, $B^{\flat}$ doubly bounded.
\end{lem}
\begin{proof}
The first part is just \cite[Lemma~5.10]{Yin:tcon:I}. The second part is established by modifying two special bijections $T'_A$, $T'_B$ given by the first part as follows. By Proposition~\ref{simplex:with:hole:rvproduct}, we may assume that $A$, $B$ are doubly bounded $\RV$-pullbacks. By Lemma~\ref{RV:fiber:dim:same} and weak \omin-minimality, each focus map in $T'_A$ consists of, without loss of generality, a single point $c_i$; similarly, the focus maps in $T'_B$ are enumerated as points $d_i$. Let $a_i \in A$ enumerate the points corresponding to the points $c_i$ and $b_i \in B$ the points $d_i$.

As we have just observed in the proof of Lemma~\ref{forg:ful}, every $\VF$-fiber of $f$ satisfies the assumption of Lemma~\ref{diff:dtdp}. So there are pairwise disjoint open polydiscs $\ga_i \sub A$, $\gb_i \sub B$ of radii $\alpha_i$, $\beta_i$ containing $a_i$, $b_i$, respectively, such that $f \rest \ga_i$, $f^{-1} \rest \gb_i$ have dtdp. This implies that $f$ must restrict to a bijection between the two sequences $(a_i)_i$, $(b_i)_i$ (alternatively, this is indeed given by the construction in the proof of \cite[Lemma~5.10]{Yin:tcon:I}). So we may assume $f(a_i) = b_i$ and $f(\ga_i)= \gb_i$. Moreover, $A$ is $\alpha_1$-invariant, $B$ is $\beta_1$-invariant, and for all $i$, $\alpha_i \leq \alpha_{i+1}$ and $\beta_i \leq \beta_{i+1}$. Then the sequences $(a_i)_i$, $(b_i)_i$ may serve as focus maps, and thereby determine two proper special covariant bijections $T_A$, $T_B$ (continuity is automatic). These are as required.
\end{proof}

\section{Integration}\label{section:int}

\subsection{Blowups}
Recall that we write $\RV$ to mean the $\RV$-sort without the middle element $0$, and $\RV_0$ otherwise, etc., although quite often the difference is immaterial and it does not matter which set $\RV$ stands for. In the definition below, this difference does matter, in particular, $\RV^{\circ \circ}$ denotes the set of those nonzero elements $t \in \RV$ with $\abvrv(t) > 0$ (see Notation~\ref{nota:RV:short}).

\begin{nota}\label{punc:RVd}
More generally, recall from Notation~\ref{pgamma} that the set $\rv(\MM_{\gamma} \mi 0)$ is denoted by $\RV^{\circ \circ}_\gamma$. For $\gamma \in \absG^+$ and $t \in \gamma^\sharp$, write
\[
\RV^{\circ \circ}(t)_\gamma = (\RV^{\circ \circ} \mi \RV^{\circ \circ}_\gamma) \uplus t;
\]
if $\gamma = \infty$ then $\RV^{\circ \circ}(t)_\gamma = \RV^{\circ \circ}$ and if $\gamma = 0$ then $\RV^{\circ \circ}(t)_\gamma = 1 \in \RV$.
\end{nota}


\begin{defn}[Blowups]\label{defn:blowup:coa}
Let $\bm U = (U, f, \omega)$ be an object of $\mRV[k]$, where $k > 0$, such that the function $\pr_{\tilde j} \rest f(U)$ is finite-to-one for some $j \in [k]$. Write $f = (f_1, \ldots, f_k)$. Let $t \in \RV_0$ be a definable element with $\tau = \abvrv(t)$ nonnegative.  An \emph{elementary blowup} of $\bm U$ in the $j$th coordinate of \emph{aperture} $(\tau, t)$ is the triple $\bm U^{\flat} = (U^{\flat}, f^{\flat}, \omega^\flat)$, where $U^{\flat} = U \times \RV^{\circ \circ}(t)_{\tau}$ and for every $(u, s) \in U^{\flat}$,
\[
f^{\flat}_{i}(u, s) = f_{i}(u) \text{ for } i \neq j, \quad f^{\flat}_{j}(u, s) = s f_{j}(u), \quad \omega^\flat(u, s) = \omega(u).
\]
We say that $\bm U^{\flat}$ is \emph{annular} if $t \neq 0$.

Let $\bm V = (V, g, \sigma)$ be another object of $\mRV[k]$ and $C_i \sub V$ finitely many pairwise disjoint definable sets. Each triple
\[
\bm C_i = (C_i, g \rest C_i, \sigma \rest C_i) \in \mRV[k]
\]
is referred to as a \emph{subobject} of $\bm V$. Suppose that $F_i : \bm U_i \fun \bm C_i$ is a $\mRV[k]$-morphism and $\bm U^{\flat}_i$ is an elementary blowup of $\bm U_i$. Let $C = \bigcup_i C_i$, $\bm C = \bigcup_i \bm C_i$,  and $F = \biguplus_i F_i$. Then the object
$(\bm V \mi \bm C) \uplus \biguplus_i \bm U_i^{\flat}$
is  a \emph{blowup of $\bm V$ via $F$}, denoted by $\bm V^{\flat}_F$. The subscript $F$ is usually dropped. The object $\bm C$ (or the set $C$) is referred to as the \emph{locus} of $\bm V^{\flat}_F$.

A \emph{blowup of length $n$} is a composition of $n$ blowups.
\end{defn}

\begin{rem}
If there is an elementary blowup of $\bm U$ then, \textit{a posteriori}, $\dim_{\RV}(U) < k$. For any coordinate of $f(U)$, there is at most one elementary blowup of $\bm U$ of aperture $(\infty, 0)$, whereas there could be many annular elementary blowups of $\bm U$. We should have included the coordinate that is blown up as a part of the data. However, in context, either this is
clear or it does not need to be spelled out, and we shall suppress mention of it for ease of notation.

Clearly every blowup $\bm U^{\flat}$ of $\bm U$ is an object of $\mRV[k]$.  In discussing blowups and related topics below, we shall work either in $\mRV[k]$, where all elementary blowups are of aperture $(\infty, 0)$, or in $\mRV^{\db}[k]$, where all elementary blowups are annular. So, if $\bm U \in \mRV^{\db}[k]$  then $\bm U^{\flat} \in \mRV^{\db}[k]$ too. No blowups in other scenarios will be considered. This is in parallel with the distinction between proper and pseudo special covariant bijections set forth in Definition~\ref{defn:special:bijection} (also see Terminology~\ref{prop:pseu}). As a matter of fact, to make this analogy precise is essentially what is left to do for the rest of our main construction.
\end{rem}

For the next few lemmas, let $\bm U = (U, f, \omega)$ and $\bm V = (V, g, \sigma)$ be objects of $\mRV[k]$.


\begin{lem}\label{elementary:blowups:preserves:iso:vol}
Suppose that $[\bm U] = [\bm V]$. Let $\bm U^{\flat}$, $\bm V^{\flat}$ be elementary blowups of $\bm U$, $\bm V$. Then there are  blowups $\bm U^{\flat\flat}$, $\bm V^{\flat\flat}$ of $\bm U^{\flat}$, $\bm V^{\flat}$ of length $1$ such that $[\bm U^{\flat\flat}] = [\bm V^{\flat\flat}]$.
\end{lem}

Note that we have two cases here, one for $\mRV[k]$ and the other for $\mRV^{\db}[k]$. Also, for the former case, we will actually show $[\bm U^{\flat}] = [\bm V^{\flat}]$, which means that the blowups $\bm U^{\flat\flat}$, $\bm V^{\flat\flat}$ are trivial, that is, their loci are empty or their apertures are of the form $(0, t)$.

\begin{proof}
Let $F : \bm U \fun \bm V$ be a morphism. As before, without loss of generality, we may simply assume that $f$, $g$ are the identity functions and $\pr_{< k} \rest U$, $\pr_{< k} \rest V$ are both injections. Let $\pr_{< k}(U) = U'$ and $\pr_{< k}(V) = V'$. The bijection $U' \fun V'$ induced by $F$ is still denoted by $F$.

We first work in $\mRV[k]$. If $\dim_{\RV}(U') < k-1$ then the obvious bijection $U^\flat \fun V^\flat$ is clearly a morphism $\bm U^{\flat} \fun \bm V^{\flat}$ (the condition involving $\jcb_{\Gamma}$ is trivial to check, and the condition involving $\jcb_{\RV}$ is irrelevant). So we may further assume
$
\dim_{\RV}(U') = \dim_{\RV}(V') = k-1
$
and hence, for almost all $u' \in U'$, $\jcb_{\RV} F (u')$ is defined. Let $u = (u', u'') \in U$ and $F(u) = (v', v'') \in V$. Let $F_{u} : \RV^{\circ\circ} \fun \RV^{\circ\circ}$ be the $u$-definable  bijection given by
\[
t \efun \frac{\omega(u) \cdot u''}{\jcb_{\RV} F (u') \cdot \sigma(F(u)) \cdot v''} \cdot t
\]
if $\jcb_{\RV} F (u')$ is defined; otherwise set $F_{u} = \id$. Then a simple computation shows that the bijection $U^{\flat} \fun V^{\flat}$ given by $(u, t) \efun (F(u), F_{u}(t))$ is a morphism.

Now we work in $\mRV^{\db}[k]$. Let $(\rho, r)$, $(\tau, t)$ be the apertures of  $\bm U^{\flat}$, $\bm V^{\flat}$, respectively. First observe that if $\rho = \tau$ then  a similar construction as above yields $[\bm U^{\flat}] = [\bm V^{\flat}]$  directly. Thus, without loss of generality, we may assume $\rho' \coloneqq \tau - \rho > 0$. Choose a definable element $r' \in \rho'^\sharp$. Let $W = U \times r$ and $\bm W$ be the corresponding subobject of $\bm U^{\flat}$. Then it is easy to see that the blowup $\bm U^{\flat\flat}$  of $\bm U^{\flat}$ with locus $\bm W$ and aperture $(\rho', r')$ is isomorphic to $\bm V^{\flat}$.
\end{proof}


\begin{cor}\label{eblowup:same:loci:iso}
Let $F : \bm U \fun \bm V$ be a morphism and $\bm U^{\flat}$, $\bm V^{\flat}$ blowups of  $\bm U$, $\bm V$ of length $1$. Then there are blowups $\bm U^{\flat\flat}$, $\bm V^{\flat\flat}$ of $\bm U^{\flat}$, $\bm V^{\flat}$ of length $1$ such that $[\bm U^{\flat\flat}] = [\bm V^{\flat\flat}]$.
\end{cor}
\begin{proof}
Let  $C$, $D$ be the loci of $\bm U^{\flat}$, $\bm V^{\flat}$, respectively. By Lemma~\ref{elementary:blowups:preserves:iso:vol}, we can make the claim hold by restricting the loci of $\bm U^{\flat}$, $\bm V^{\flat}$ to $C \cap F^{-1}(D)$, $F(C) \cap D$. But in the meantime we can also blow up $\bm U^{\flat}$, $\bm V^{\flat}$ at the loci $F^{-1}(D) \mi C$, $F(C) \mi D$ using the apertures induced by those on $D \mi F(C)$, $C \mi F^{-1}(D)$, respectively. Then the resulting compounded blowups of $\bm U^{\flat}$, $\bm V^{\flat}$ of length $1$ are as desired.
\end{proof}

This corollary also holds in $\RV[k]$, which is essentially the only thing that the otherwise formal proof of \cite[Lemma~5.29]{Yin:tcon:I} depends on. Thus, the same proof yields the following analogue in the present context:

\begin{lem}\label{blowup:equi:class:coa}
If $[\bm U] = [\bm V]$ and $\bm U_1$, $\bm V_1$ are blowups of $\bm U$, $\bm V$ of lengths $m$, $n$, respectively, then there are blowups $\bm U_2$, $\bm V_2$ of $\bm U_1$, $\bm V_1$ of lengths $n$, $m$, respectively, such that $[\bm U_2] = [\bm V_2]$.
\end{lem}

\begin{cor}\label{blowup:equi:class}
Suppose that $[\bm U] = [\bm U']$ and $[\bm V] = [\bm V']$. If there are isomorphic blowups of $\bm U$, $\bm V$ then there are isomorphic blowups of $\bm U'$, $\bm V'$.
\end{cor}

\begin{defn}
Let $\misp[k]$ be the set of pairs $(\bm U, \bm V)$ of objects of $\mRV[k]$ such that there exist isomorphic blowups $\bm U^{\flat}$, $\bm V^{\flat}$. Let $\misp[*] = \bigoplus_{k} \misp[k]$. Similarly for the sets $\mispdb[k] \sub \mRV^{\db}[k] \times \mRV^{\db}[k]$ and $\mispdb[*]$.
\end{defn}

We will just write $\misp$, $\mispdb$ for all these sets when there is no danger of confusion. By Corollary~\ref{blowup:equi:class}, they may be regarded as binary relations on isomorphism classes.

\begin{lem}\label{isp:congruence:vol}
$\misp[k]$, $\mispdb[k]$ are semigroup congruence relations and $\misp[*]$, $\mispdb[*]$ are semiring congruence relations.
\end{lem}
\begin{proof}
This is an easy consequence of Lemma~\ref{blowup:equi:class:coa} (the proof of \cite[Lemma~5.32]{Yin:tcon:I}, which uses \cite[Lemma~5.29]{Yin:tcon:I}, contains more details).
\end{proof}

\begin{rem}\label{db:isp:rel}
The canonical embedding $\gsk \mRV^{\db}[k] \fun \gsk \mRV[k]$ and $\misp[k]$ together induce a semigroup congruence relation ${\mu}{\text{I}_{\text{sp}}^{\db}}[k]$ on $\gsk \mRV^{\db}[k]$. Observe that if $\bm V$ is a blowup of $\bm U$ in $\mRV^{\db}[k]$ then there are isomorphic blowups $\bm V^\flat$, $\bm U^\flat$ of $\bm V$, $\bm U$ in $\mRV[k]$. It follows easily from this and Corollary~\ref{blowup:equi:class} that $\mispdb[k]$ is contained in ${\mu}{\text{I}_{\text{sp}}^{\db}}[k]$.
\end{rem}

Let  $T$ be a special covariant bijection on $\bb L (U,f)$, proper if and only if $\bm U \in \mRV^{\db}[k]$. Let $A = \mathbb{L} (U,f)^\flat$. The $k$th component of $(A_{\RV})_{\leq k} \in \RV[{\leq}k]$ is denoted by $(U_{T})_{\leq k}$ and  the object $(U_{T}, \omega_T)_{\leq k} \in \mRV[k]$ by $\bm U_{T}$, where $\omega_T : U_{T} \fun \RV$ is the function induced by $\omega$.

\begin{lem}\label{special:to:blowup:coa}
The object $\bm U_T$ is isomorphic to a blowup of $\bm U$ of the same length as $T$.
\end{lem}
\begin{proof}
By induction on the length $\lh (T)$ of $T$ and Lemma~\ref{blowup:equi:class:coa}, this is immediately reduced to the case $\lh (T) = 1$. Then, without loss of generality, we may assume that the locus of $T$ is $\mathbb{L} (U,f)$. If $T$ is pseudo then it is clear how to construct an (elementary) blowup of $\bm U$ as desired. If $T$ is proper then the blowup needs to be annular and hence we need to first find a nontrivial aperture. But this is already provided by $T$.
\end{proof}

\begin{lem}\label{blowup:same:RV:coa}
Let $\bm U^{\flat}$ be a blowup of $\bm U$ of length $l$. Then ${\mL} \bm U^{\flat}$ is isomorphic to ${\mL} \bm U$; more precisely, the isomorphism is  in $\mVF[k]$ if $\bm U \in \mRV[k]$ and is in $\mVF^{\diamond}[k]$ if $\bm U \in \mRV^{\db}[k]$.
\end{lem}
\begin{proof}
By induction on $l$, this is immediately reduced to the case $l=1$. We may assume that $\pr_{\tilde 1} \rest f(U)$ is injective and $\bm U^{\flat}$ is an elementary blowup in the first coordinate. If $\bm U \in \mRV[k]$ then it is enough to construct a focus map into the first coordinate with locus $f(U)^\sharp$; this can be easily done, using once again the fact that every definable set contains a definable point (see Remark~\ref{rem:hyp}). If $\bm U \in \mRV^{\db}[k]$ then the aperture $(\tau, t)$ of $\bm U^{\flat}$ needs to be factored into the construction. Fix a definable point $c \in t^\sharp$. For $u \in U$ and $a \in  f(u)_{\tilde 1}^\sharp$, the $\RV$-disc $f(u)_{1}^\sharp$ contains an $a$-definable point $b_a$; moreover, by Lemma~\ref{RVlift}, these points may be chosen uniformly via a differentiable function. Thus there is a continuous (fiberwise) additive translation, with respect to the points $b_a$ and $b_a - c b_a$, between the two sets in question, which is proper covariant by Lemma~\ref{conti:covar}.
\end{proof}

\subsection{Standard contractions}

Suppose that $A \sub \VF^{n} \times \RV^{m}$ is an invariant set. Let $i \in [n]$ and $T_i$ be a definable bijection on $A$,  relatively special covariant in the $i$th $\VF$-coordinate, such that for every $a \in \pr_{\tilde i}(A)$, the set $T_i(A_a)$ is an $\RV$-pullback. Moreover, if $A$ is proper invariant then $T_i$ is continuous and proper covariant. Note that $T_i$ is not necessarily a special covariant bijection since, to begin with,  the various special covariant bijections in the $i$th $\VF$-coordinate  may not even be of the same length. Let
\[
 A_i = \bigcup_{a \in \pr_{\tilde i}(A)} a \times (T_i(A_a))_{\RV} \sub \VF^{n-1} \times \RV^{m_i}.
\]
Let $\hat T_i : A \fun A_i$ be the function induced by $T_i$. If $A$ is proper invariant then $A_i$ is proper invariant too, which is not guaranteed if we do not demand $T_i$ to be proper covariant (continuity is needed for other reasons, see Definition~\ref{defn:standard:contraction} and the proof of Lemma~\ref{exit:stand} below).

For any $j \in [n{-}1]$, we repeat the above procedure on $A_i$ with respect to the $j$th $\VF$-coordinate and thereby obtain a set $A_{j} \sub \VF^{n-2} \times \RV^{m_j}$ and a function $\hat T_{j} : A_i \fun A_{j}$. The relatively special covariant bijection on $T_i(A)$ induced by $\hat T_{j}$ is denoted by $T_j$. Continuing thus, we obtain a sequence of  bijections $T_{\sigma(1)}, \ldots, T_{\sigma(n)}$ and a corresponding function $\hat T_{\sigma} : A \fun \RV^{l}$, where $\sigma$ is the permutation of $[n]$ in question. The composition $T_{\sigma(n)} \circ \ldots \circ T_{\sigma(1)}$, which is referred to as the \emph{lift} of $\hat T_{\sigma}$, is denoted by $T_{\sigma}$.

Suppose that there is a $k \in 0 \cup [m]$ such that  $(A_a)_{\leq k} \in \RV[k]$  for  every $a \in A_{\VF}$. In particular, if $k=0$ then $A \in \VF_*$. By Lemma~\ref{RV:fiber:dim:same}, $\hat T_{\sigma}(A)_{\leq n+k}$ is an object of $\RV[{\leq} l{+}k]$, where $\dim_{\VF}(A) = l$, or even an object of $\RV^{\db}[n{+}k]$ if $A$ is proper invariant.


\begin{defn}\label{defn:standard:contraction}
The function $\hat T_{\sigma}$, or the object $\hat T_{\sigma}(A)_{\leq n+k}$, is referred to as a \emph{standard contraction} of the invariant set $A$ with the \emph{head start} $k$.

Suppose that $\omega : A \fun \RV$ is a definable function, proper covariant if $A$ is proper invariant, and $\hat T_{\sigma}$ is a standard contraction of $A_\omega$ (recall Notation~\ref{nota:form:down}). Let $\hat T^\omega_{\sigma}(A)_{\leq n+k}$ be the $(n+k)$th component of $\hat T_{\sigma}(A_\omega)_{\leq n+k}$, which is not the empty object if and only if $\dim_{\VF}(A) = n$.  The definable function on $\hat T^\omega_{\sigma}(A)_{\leq n+k}$ induced by $\omega$ is denoted by $\omega_{\hat T_{\sigma}}$. The function $\hat T_{\sigma}$, or the object
\[
\hat T_{\sigma}(\bm A)_{\leq n+k} \coloneqq (\hat T^\omega_{\sigma}(A), \omega_{\hat T_{\sigma}})_{\leq n+k},
\]
is referred to as a \emph{standard contraction} of the pair $\bm A = (A, \omega)$.
\end{defn}

Let $A^\sharp = \bigcup_{a \in A_{\VF}} a \times \bb L (A_a)_{\leq k}$, which is an object of $\VF_*$. The volume form on $A^\sharp$ induced by $\omega$ is still denoted by $\omega$. The  bijection $A^\sharp_\omega \fun \mathbb{L} \hat T_{\sigma}(A_\omega)_{\leq n+k}$ induced by $T_{\sigma}$, which is still denoted by $T_{\sigma}$, is indeed a $\mVF[n{+}k]$-morphism or a $\mVF^{\diamond}[n{+}k]$-morphism if $(A, \omega)$ is proper invariant.

The head start of a standard contraction is usually implicit. In fact, it is always $0$ except in Lemma~\ref{isp:VF:fiberwise:contract}, and can be circumvented even there. This seemingly needless gadget only serves to make the above definition more streamlined: If $A \in \VF_*$ then the intermediate steps of a standard contraction of $A$ may or may not result in objects of $\VF_*$ and hence the definition cannot be formulated entirely within $\VF_*$.

By Remark~\ref{special:pseudo} and compactness, every pseudo invariant pair $(A, \omega)$ of the above form admits a standard contraction with respect to any permutation $\sigma$ of $[n]$ and any head start $k \in [m]$. This also holds for the other case:

\begin{lem}\label{exit:stand}
If $(A, \omega)$ is proper invariant then it also admits such a standard contraction.
\end{lem}

Before we start the proof, first observe that $(A, \omega)$ may be viewed as the obvious intermediate result of a  standard contraction of $(A^\sharp, \omega) \in \mVF^{\diamond}[n{+}k]$. Thus we may and do assume $k=0$, that is, $A \in \VF_*$.

\begin{proof}
We do induction on $n$. The base case $n=1$ follows from Remark~\ref{special:pseudo}. For the inductive step, we may assume $\sigma = \id$ and, by compactness, $\omega = 1$. By the inductive hypothesis and compactness again,  there is a definable bijection $T$ on $A$ such that, for every $a \in A' \coloneqq \pr_n(A) \sub \VF$, the induced bijection $T_a$ on $A_a$ is the lift of a standard contraction $\hat T_a$ of $A_a$; in particular, $T$ may be written as a composition of relatively proper special covariant bijections $T_i : A_{i-1} \fun A_{i}$. Note that each $T^{-1}_i$ is $\rv$-contractible.

Suppose that $A$ is proper $\delta$-invariant. Let $\ga$ range over open discs of radius $\delta$ contained in $A'$. Applying Lemma~\ref{special:bi:term:constant} in the $n$th $\VF$-coordinate as in the proof of Proposition~\ref{simplex:with:hole:rvproduct}, we deduce that over all but finitely many definable open discs $\ga$, the $\RV$-data of $T_a$ is constant over $\ga$; it will become clear what $\RV$-data is actually needed below. By FMT (Terminology~\ref{FMT}), we may assume that there are no exceptional discs and in light of Lemma~\ref{fiber:conti},  every $T_i$ is continuous.

By the constancy of $\RV$-data over each $\ga$  and Lemma~\ref{invar:db}, the $\VF$-coordinates of $A_i$  are bounded and the $\RV$-coordinates of $A_i$ are doubly bounded. Pulling back open polydiscs of the form $\ga \times \gp$ along $T^{-1}$, where  $\gp$ is an $\RV$-polydisc, we see that every $\VF$-fiber of $A_i$ is a union  of open polydiscs of the form $\gb \times \gq$, where $\gb \sub \ga$ for some $\ga$ and the radius of $\gq$ only depends on $\ga$ and the $\VF$-fiber in question. These open discs $\gb$ may be taken to be maximal in the obvious sense and hence form an $\go$-partition. By Lemma~\ref{vol:par:bounded}, we conclude that each $A_i$ is indeed proper invariant. Then, by Lemma~\ref{conti:covar}, every $T_i$ is proper covariant. Finally, applying Remark~\ref{special:pseudo} to the set $\bigcup_{a \in A'} a \times \hat T_a(A_a)$, the lemma follows.
\end{proof}

We remark that  standard contractions bear marked similarities to special covariant bijections and may indeed be used to deduce, say, Lemma~\ref{L:measure:surjective}, etc.


\begin{lem}\label{kernel:dim:1:coa}
Suppose that $[(A, \omega)] = [(B, \sigma)]$ in $\gsk \mVF[1]$ (respectively,  $\gsk \mVF^{\diamond}[1]$). Let $(\bm U, \omega')$, $(\bm V, \sigma')$ be, respectively, standard contractions of $(A, \omega)$, $(B, \sigma)$ in $\mRV[1]$ (respectively,  $\mRV^{\db}[1]$). Then
\[
([(\bm U, \omega')], [(\bm V, \sigma')]) \in \misp \quad (\text{respectively, } \mispdb).
\]
\end{lem}
\begin{proof}
Let $F: (A, \omega) \fun (B, \sigma)$ be a $\mVF[1]$-morphism. The essential bijection $\bb L \bm U \fun \bb L \bm V$ induced by $F$ is denoted by $F'$. Let $A' = \dom(F')$ and $B' = \ran(F')$. By Lemma~\ref{simul:special:dim:1}, deleting finitely many points from $A'$, $B'$ if necessary, we may assume that there are special bijections $T$, $R$ on $\bb L \bm U$, $\bb L \bm V$ such that $A'^{\flat}$, $B'^{\flat}$ are $\RV$-pullbacks and $F'^{\flat}_{\downarrow}$ is an $\RV[1]$-morphism
\[
\bm A' \coloneqq (A'^{\flat}_{\RV})_1 \fun \bm B' \coloneqq (B'^{\flat}_{\RV})_1.
\]
By Lemma~\ref{special:to:blowup:coa}, $(\bm A', \omega'_T)$, $(\bm B', \sigma'_R)$ are isomorphic to blowups of $(\bm U, \omega')$, $(\bm V, \sigma')$, respectively. By Lemma~\ref{RV:iso:class:lifted:jcb}, the bijection $F'^{\flat}_{\downarrow}$ is indeed a $\mRV[1]$-morphism $(\bm A', \omega'_T) \fun (\bm B', \sigma'_R)$ except at finitely many points (the condition involving $\jcb_{\Gamma}$ might fail). At this stage, without loss of generality, we may retroactively assume that $A'$, $B'$ are of the forms $t^\sharp$, $s^\sharp$, where $t, s \in \RV$, and $T$, $R$ are both centripetal transformations, and $F'$ has dtdp (after deleting finitely many points again, see  Lemma~\ref{open:pro} and the proof of \cite[Lemma~5.10]{Yin:tcon:I} for more details). Since the volume forms $\omega'_T$, $\sigma'_R$ are constant, by \omin-minimality in the $\K$-sort and Lemma~\ref{RV:iso:class:lifted:jcb} again, we must have that $F'^{\flat}_{\downarrow}$ is differentiable everywhere and indeed,  for all $u \in \rv(A'^{\flat})$ and all $a \in u^\sharp$,
\[
\rv(\jcb_{\VF} F'^{\flat}( a)) = \jcb_{\RV} F'^{\flat}_{\downarrow}(u) = \omega'_T / \sigma'_R.
\]
It follows that $F'^{\flat}_{\downarrow}$ is a $\mRV[1]$-morphism.

The argument above also works for the other case. Note that  $F$ and hence $F'$ are then already  bijections. Also, we are not free to delete finitely many points to enforce dtdp. Instead, Lemma~\ref{diff:dtdp} is applied as in the proof of Lemma~\ref{simul:special:dim:1} to ensure that the condition involving $\jcb_{\Gamma}$ holds at these exceptional points.
\end{proof}


\begin{lem}\label{isp:VF:fiberwise:contract}
Let $A'$, $A''$ be definable sets with $A'_{\VF} = A''_{\VF} \eqqcolon A \sub \VF^n$ and $\omega'$, $\omega''$ definable functions from $A'$, $A''$ into $\RV$, respectively. Write $\bm A' = (A', \omega')$ and $\bm A'' = (A'', \omega'')$. Suppose that there is a $k \in \N$ such that for  every $a \in A$,
\begin{equation}\label{condi:fiber}
([(A'_a, \omega')]_{\leq k}, [(A''_a, \omega'')]_{\leq k}) \in \misp.
\end{equation}
Let $\hat T_{\sigma}$, $\hat R_{\sigma}$ be standard contractions of $\bm A'$, $\bm A''$, respectively. Then
\[
([\hat T_{\sigma}(\bm A')]_{\leq n+k}, [\hat R_{\sigma}(\bm A'')]_{\leq n+k}) \in \misp.
\]
If $\bm A'$, $\bm A''$ are proper invariant then the same holds with respect to $\mispdb$.
\end{lem}

Note that condition (\ref{condi:fiber}) makes sense only  over the substructure $\mdl S \la a \ra$.

\begin{proof}
By induction on $n$, this is immediately reduced to the case $n=1$. So assume $A \sub \VF$. After replacing $A'$, $A''$ with $A'_{\omega'}$, $A''_{\omega''}$, we see that the special bijection $F : A \fun A^\sharp$ constructed in the proof of \cite[Lemma~5.36]{Yin:tcon:I} also satisfies the condition that for all $\RV$-polydisc $\gp \sub A^{\sharp}$ and all $a_1, a_2 \in F^{-1}(\gp)$,
\[
\omega' \rest A'_{a_1} = \omega' \rest A'_{a_2} \dand \omega'' \rest A''_{a_1} = \omega'' \rest A''_{a_2}.
\]
Therefore, using Lemma~\ref{kernel:dim:1:coa} in place of \cite[Lemma~5.34]{Yin:tcon:I}, that proof goes through here with virtually no changes (the additional computations involving $\jcb_{\RV}$ and $\jcb_{\Gamma}$ are all straightforward).

The second claim is similar, except that $F$ must be a proper special covariant bijection with a sufficiently large aperture. As before, this may be taken care of by  FMT (also see the discussion in Remark~\ref{special:pseudo}) .
\end{proof}

\begin{cor}\label{contraction:same:perm:isp}
Let $\bm A'$, $\bm A''$ be as above with $k=0$, and suppose that there is a morphism $F : \bm A' \fun \bm A''$ that is relatively unary in the $i$th $\VF$-coordinate. Then for any permutation $\sigma$ of $[n]$ with $\sigma(1) = i$ and any standard contractions $\hat T_{\sigma}$, $\hat R_{\sigma}$ of $\bm A'$, $\bm A''$,
\[
([\hat T_{\sigma}(\bm A')]_{\leq n}, [\hat R_{\sigma}(\bm A'')]_{\leq n}) \in \misp \quad (\text{respectively, } \mispdb).
\]
\end{cor}

\begin{proof}
First suppose that $F$ is a $\mVF[n]$-morphism. Note that since $F$ is partially differentiable outside a definable subset of $\VF$-dimension less than $n$, it may not induce a morphism $\bm A'_a \fun \bm A''_a$ for every $a \in \pr_{\tilde i}(A)$. By Lemma~\ref{L:measure:surjective}, there is a standard contraction $\hat T'_{\sigma}$ of $\bm A'$ such that the subset of $A'$ that corresponds to $\hat T'_{\sigma}(\bm A')_{\leq n}$ via $T'_{\sigma}$ is contained in the partial differential locus of $F$. Thus, by Lemma~\ref{isp:VF:fiberwise:contract}, we may assume that $F$ is indeed differentiable. But then the claim follows immediately from Lemmas~\ref{kernel:dim:1:coa} and \ref{isp:VF:fiberwise:contract}.

If $F$ is a $\mVF^{\diamond}[n]$-morphism then there is a $\delta \in \absG^+$ such that $A'$ and $\hat T_{\sigma(1)}(A')$ are both $\delta$-invariant. This means that for all $a \in \pr_{\tilde i}(A)$ there is a $b \in \pr_{\tilde i}(A)$ such that they are contained in the same open polydisc of radius $\delta$, $F$ induces a morphism $\bm A'_b \fun \bm A''_b$, and $\hat T_{\sigma(1)}(A')_a = \hat T_{\sigma(1)}(A')_b$. Thus the claim follows immediately from Lemmas~\ref{kernel:dim:1:coa} and \ref{isp:VF:fiberwise:contract} again.
\end{proof}

\begin{defn}[$\vv$-affine and $\rv$-affine]\label{rvaffine}
Let $\ga$ be an open disc and $f : \ga \fun \VF$ an injection.
We say that $f$ is \emph{$\vv$-affine} if there is a (necessarily unique) $\gamma \in \Gamma$, called the \emph{shift} of $f$, such that, for all $a, a' \in \ga$,
\[
\abval(f(a) - f(a')) = \gamma + \abval(a - a').
\]
We say that $f$ is \emph{$\rv$-affine} if there is a (necessarily unique) $t \in \RV$, called the \emph{slope} of $f$, such that, for all $a, a' \in \ga$,
\[
\rv(f(a) - f(a')) = t \rv(a - a').
\]
\end{defn}

\begin{defn}\label{defn:balance}
Let $A \sub \VF^2$ be a definable set such that $\ga_1 \coloneqq \pr_1(A)$ and $\ga_2 \coloneqq \pr_2(A)$ are both open discs. Let $f : \ga_1 \fun \ga_2$ be a definable bijection that has dtdp.  We say that $f$ is \emph{balanced in $A$} if $f$ is actually $\rv$-affine and there are $t_1, t_2 \in \RVV$, called the \emph{paradigms} of $f$, such that, for every $a \in \ga_1$,
\[
A_a = t_2^\sharp + f(a) \dand f^{-1}(A_a) = a - t_1^\sharp.
\]
\end{defn}

\begin{defn}[$2$-cell]\label{def:units}
We say that a set $A$ is a \emph{$1$-cell} if it is either an open disc contained in a single $\RV$-disc or a point in $\VF$. We say
that $A$ is a \emph{$2$-cell} if
\begin{itemize}
 \item $A$ is a subset of $\VF^2$ contained in a single $\RV$-polydisc and $\pr_1(A)$ is a $1$-cell,
 \item there is a function $\epsilon : A_1 \coloneqq \pr_1 (A) \fun \VF$ and a $t \in \RV$ such that, for every $a \in A_1$, $A_a = t^\sharp + \epsilon(a)$,
 \item one of the following three possibilities occurs:
  \begin{itemize}
   \item $\epsilon$ is constant,
   \item $\epsilon$ is injective, has dtdp, and $\rad(\epsilon(A_1)) \geq \abs{\vrv(t)}$,
   \item $\epsilon$ is balanced in $A$.
  \end{itemize}
\end{itemize}
The function $\epsilon$ is called the \emph{positioning function} of $A$ and the element $t$ the \emph{paradigm} of $A$.

More generally, a set $A$ with exactly one $\VF$-coordinate is a \emph{$1$-cell} if, for each $t \in A_{\RV}$, $A_t$ is a $1$-cell in the above sense; the parameterized version of the notion of a $2$-cell is formulated in the same way.
\end{defn}

\begin{lem}\label{decom:into:2:units}
Let  $A \sub \VF^2$ be an invariant set. Then there is a standard contraction $\hat T_\sigma$ of $A$ such that $A_s$ is a $2$-cell for every $s \in \hat T_\sigma(A)$.
\end{lem}
\begin{proof}
This is just \cite[Lemma~5.20]{Yin:tcon:I} if $A$ is pseudo invariant, even though it is not exactly stated in this form. The proof of \cite[Lemma~5.20]{Yin:tcon:I} proceeds by constructing a positioning function $\epsilon_{(t,s)}$ in each $\VF$-fiber, which heavily relies on the use of standard contractions in the preceding auxiliary results, namely \cite[Lemmas~5.15, 5.16]{Yin:tcon:I}.  It is not hard to see that if $A$ is proper invariant  then that proof still goes through, provided that FMT is applied at suitable places to cut out finitely many exceptional open discs in the first coordinate at which the desired properties of $\epsilon_{(t,s)}$ do not hold, as we have done above, say, in the proof of Lemma~\ref{exit:stand}.
\end{proof}

We also remark that Lemma~\ref{decom:into:2:units} holds fiberwise for invariant sets $A \sub \VF^n$ with $n \geq 2$, that is, there is a standard contraction $\hat T_\sigma$ of $A$ such that for every $(a, s) \in \hat T_{\sigma(2)}(A)$, $\hat T_{\sigma(2)}^{-1}(a, s)$ is of the form $a \times C$, where $C$ is a $2$-cell. This is immediate by compactness if $A$ is pseudo invariant. If $A$ is proper invariant then it follows from Lemma~\ref{decom:into:2:units} and the construction in the proof of  Lemma~\ref{exit:stand} (Lemma~\ref{decom:into:2:units} serves as the base case and hence, in the inductive step, we can assume that $\hat T_{\sigma(2)}$ is already as desired fiberwise).

From now on, the pain of treating two cases (pseudo versus proper) in parallel is over, as the proofs will be formal and work for either one.

For the next two lemmas, let $12$, $21$ denote the permutations of $[2]$ and $\bm A = (A, \omega)$ be an object of either $\mVF[2]$ or $\mVF^\diamond[2]$.

\begin{lem}\label{2:unit:contracted}
Suppose that $A \sub \VF^2$ is a $2$-cell and $\omega$ is constant. Then there are standard contractions $\hat T_{12}$, $\hat R_{21}$ of $\bm A$ such that $[\hat T_{12}(\bm A)]_{\leq 2} = [\hat R_{21}(\bm A)]_{\leq 2}$.
\end{lem}
\begin{proof}
We may of course assume $\dim_{\VF}(A) = 2$. All we need to do is to check that the maps constructed in the proof of  \cite[Lemma~5.25]{Yin:tcon:I} are indeed $\mRV[2]$- or $\mRV^{\db}[2]$-morphisms. By inspection of that proof, we see that there are two cases: $A$ is a product of two open discs or the positioning function $\epsilon$ in question is balanced in $A$ with nonzero paradigms $t_1$, $t_2$. The first case is obvious since we can simply use the identity map. In the second case, observe that there is no requirement on $\jcb_{\RV}$ since the resulting standard contractions are of $\RV$-dimension $1$. On the other hand, a morphism between the standard contractions can be easily constructed using  $\epsilon$  (we could also cite Lemma~\ref{simul:special:dim:1}, but the situation here is much simpler), and  the requirement on $\jcb_{\Gamma}$ is satisfied since the slope of the $\rv$-affine function $\epsilon$ is $-t_2/t_1$ (see the last paragraph of \cite[Remark~5.18]{Yin:tcon:I} for further explanation).
\end{proof}

\begin{lem}\label{subset:partitioned:2:unit:contracted}
There are a morphism $\bm A \fun \bm A^*$, relatively unary in both coordinates, and two standard contractions $\hat T_{12}$, $\hat R_{21}$ of $\bm A^*$ such that $[\hat T_{12}(\bm A^*)]_{\leq 2} = [\hat R_{21}(\bm A^*)]_{\leq 2}$.
\end{lem}
\begin{proof}
This follows from  Lemmas~\ref{decom:into:2:units}, \ref{2:unit:contracted}, and compactness (see the proof of \cite[Lemma~5.26]{Yin:tcon:I} for a bit more details).
\end{proof}

\begin{lem}\label{contraction:perm:pair:isp}
Let $\bm A = (A, \omega)$ be an object of $\mVF[n]$ (respectively, $\mVF^\diamond[n]$). Suppose that $i, j \in [n]$ are distinct and $\sigma_1$, $\sigma_2$ are  permutations of $[n]$ such that
\[
\sigma_1(1) = \sigma_2(2) = i, \quad \sigma_1(2) = \sigma_2(1) = j, \quad \sigma_1
\rest \set{3, \ldots, n} = \sigma_2 \rest \set{3, \ldots, n}.
\]
Then, for any standard contractions $\hat T_{\sigma_1}$, $\hat T_{\sigma_2}$ of $\bm A$,
\[
([\hat T_{\sigma_1}(\bm A)]_{\leq n}, [\hat T_{\sigma_2}(\bm A)]_{\leq n}) \in \misp  \quad (\text{respectively, } \mispdb).
\]
\end{lem}
\begin{proof}
Let $ij$, $ji$ denote the permutations of $\{i, j\}$ and $E = [n] \mi \{i, j\}$. By compactness and
Lemma~\ref{isp:VF:fiberwise:contract}, it is enough to show
that for $a \in A_E$ and any standard
contractions $\hat T_{ij}$, $\hat T_{ji}$ of $\bm A_a \coloneqq (A_a, \omega \rest A_a)$,
\[
([\hat T_{ij}(\bm A_a)]_{\leq 2}, [\hat T_{ji}(\bm A_a)]_{\leq 2})
\in \misp \quad (\text{respectively, }\mispdb).
\]
By Corollary~\ref{contraction:same:perm:isp} and Lemma~\ref{isp:VF:fiberwise:contract}, it is enough to find a morphism $\bm A_a \fun \bm B$, relatively unary in both coordinates, and  standard contractions $\hat R_{ij}$, $\hat R_{ji}$ of $\bm B$ such that $[\hat R_{ij}(\bm B)]_{\leq 2} = [\hat R_{ji}(\bm B)]_{\leq 2}$. This is just Lemma~\ref{subset:partitioned:2:unit:contracted}.
\end{proof}

\subsection{The kernel of ${\mu}\mathbb{L}$ and the main theorems}

The following proposition is the culmination of the preceding technicalities, which identifies the
congruence relations $\misp$ and $\mispdb$ with those induced by
$\mu\bb L$.


\begin{prop}\label{kernel:L}
For $\bm U, \bm V \in \mRV[k]$,
\[
[\mu\bb L \bm U] = [\mu\bb L \bm V] \text{ in }\gsk \mVF[k] \quad \text{if and only if} \quad ([\bm U], [\bm V]) \in \misp.
\]
The same holds with respect to $\mRV^{\db}[k]$, $\gsk \mVF^\diamond[k]$, and $\mispdb$.
\end{prop}
\begin{proof}
The ``if'' direction simply follows from Lemma~\ref{blowup:same:RV:coa} and
Proposition~\ref{L:sur:c}.

For the ``only if'' direction, we proceed by induction on $k$. The base case $k = 1$ is of course Lemma~\ref{kernel:dim:1:coa}. For the inductive step, let
\[
\mu\bb L \bm U = \bm B_1 \to^{G_1} \bm B_2 \cdots \bm B_l \to^{G_l} \bm B_{l+1} = \mu\bb L \bm V
\]
be relatively unary $\mVF[k]$-morphisms, which exist by Lemma~\ref{bijection:partitioned:unary}. For each $j \leq l - 2$, we can choose five standard contractions
\[
[\bm U_j]_{\leq k}, \quad [\bm U_{j+1}]_{\leq k}, \quad [\bm U'_{j+1}]_{\leq k}, \quad [\bm U''_{j+1}]_{\leq k}, \quad [\bm U_{j+2}]_{\leq k}
\]
of $\bm B_j$, $\bm B_{j+1}$, $\bm B_{j+1}$, $\bm B_{j+1}$, $\bm B_{j+2}$ with the permutations $\sigma_{j}$, $\sigma_{j+1}$, $\sigma'_{j+1}$, $\sigma''_{j+1}$, $\sigma_{j+2}$ of $[k]$, respectively, such that
\begin{itemize}
  \item $\sigma_{j+1}(1)$ and $\sigma_{j+1}(2)$ are the $\VF$-coordinates targeted by $G_{j}$ and $G_{j+1}$, respectively,
  \item $\sigma''_{j+1}(1)$ and $\sigma''_{j+1}(2)$ are the $\VF$-coordinates targeted by $G_{j+1}$ and $G_{j+2}$, respectively,
  \item $\sigma_{j} = \sigma_{j+1}$, $\sigma''_{j+1} =  \sigma_{j+2}$,  and $\sigma'_{j+1}(1) = \sigma''_{j+1}(1)$,
  \item the relation between $\sigma_{j+1}$ and $\sigma'_{j+1}$ is as described in Lemma~\ref{contraction:perm:pair:isp}.
\end{itemize}
By Corollary~\ref{contraction:same:perm:isp} and Lemma~\ref{contraction:perm:pair:isp}, all the adjacent pairs of these standard contractions are indeed $\misp$-congruent, except $([\bm U'_{j+1}]_{\leq k}, [\bm U''_{j+1}]_{\leq k})$.  Since we can choose $[\bm U'_{j+1}]_{\leq k}$, $[\bm U''_{j+1}]_{\leq k}$ so that they start with the same contraction in the first targeted $\VF$-coordinate of $\bm B_{j+1}$, the resulting sets from this step are the same. Therefore, applying  the inductive hypothesis in each fiber over the just contracted coordinate, we see that this last pair is also $\misp$-congruent. This completes the ``only if'' direction.

The same argument works for the other case. Notice that Lemma~\ref{bijection:partitioned:unary} is no longer needed since it is built into the definition of $\mVF^\diamond[k]$.
\end{proof}

\begin{rem}\label{kernel:gL}
Recall from Remark~\ref{gL:sur:c} the obvious surjective semiring homomorphism
$\bb F : \gsk \mRV[*] \epi \gsk \mgRV[*]$. The congruence relation $\misp$ induces a semiring congruence relation $\mgisp$ on $\gsk \mgRV[*]$ via $\bb F$, which may also be regarded as the coarser congruence relation on $\gsk \mRV[*]$ generated by $\misp$ and the kernel of $\bb F$; similarly for the congruence relation $\mgispdb$ on $\gsk \mRV^{\db}[*]$.

By Lemma~\ref{blowup:same:RV:coa} and Remark~\ref{gL:sur:c}, if $([\bm U], [\bm V]) \in \mgisp$ then  $[{\mgL}{\bb F\bm U}] = [{\mgL} \bb F\bm V]$. Conversely, suppose that $[{\mgL}{\bb F\bm U}] = [{\mgL} \bb F\bm V]$ in  $\gsk \mgVF[*]$. By (the remark before) Lemma~\ref{forg:ful}, there is an $\bm A \in \mVF[*]$ such that $[{\mL}{\bm U}] = [\bm A]$ in  $\gsk \mVF[*]$ and $\bb F \bm A = {\mgL} \bb F\bm V$. It is easy to see that, after attaching additional $\RV$-coordinates to account for each other's volume forms, there are standard contractions $\bm U'$, $\bm V'$ of $\bm A$, ${\mL} {\bm V}$ such that the identity map is a morphism $\bb F \bm U' \fun \bb F \bm V'$. So we may deduce $([\bm U], [\bm V]) \in \mgisp$ from Proposition~\ref{kernel:L}.

By the same reasoning, we have
\[
[{\mgL}{\bb F\bm U}] = [{\mgL} \bb F\bm V] \text{ in } \gsk \mgVF^\diamond[*] \quad \text{if and only if} \quad  ([\bm U], [\bm V]) \in \mgispdb.
\]
\end{rem}

%
%

\begin{rem}\label{thekernel}
Proposition~\ref{kernel:L} shows that the kernel of $\mL$ in $\gsk \mRV[*]$ is indeed generated by the pair $([1],  [\RV^{\circ \circ}])$ and hence the corresponding ideal of the graded ring $\ggk \mRV[*]$ is generated by the element $\bm P$ (Notation~\ref{nota:RV:short}). On the other hand, the kernel  of $\mL$ in $\gsk \mRV^{\db}[*]$ is generated by the pairs $([1], [\RV^{\circ \circ}(t)_\gamma])$, where $\gamma \in \absG$ and $t \in \gamma^\sharp$ are definable (Notation~\ref{punc:RVd}), and hence  the corresponding ideal of the graded ring $\ggk \mRV^{\db}[*]$  is generated by the set $\bm P_\Gamma$ of elements $\bm P_\gamma$, $\abs \gamma > 0$ (Notation~\ref{pgamma}); note that the class $\bm P_\gamma$ does not depend on the choice of $t_\gamma \in \gamma^\sharp$.

Remark~\ref{kernel:gL} shows the same thing for the ideals of $\ggk \mgRV[*]$, $\ggk \mgRV^{\db}[*]$ induced by $\mgL$.
\end{rem}

\begin{thm}\label{main:prop}
For each $k \geq 0$ there is a canonical isomorphism of Grothendieck semigroup
\[
\int_{+} : \gsk  \mVF[k] \fun \gsk  \mRV[k] /  \misp
\]
such that
\[
 \int_{+} [\bm A] = [\bm U]/  \misp \quad \text{if and only if} \quad  [\bm A] = [\mu\bb L\bm U].
\]
Putting these together, we obtain a canonical isomorphism of graded Grothendieck semirings
\[
\int_{+} : \gsk \mVF[*] \fun \gsk  \mRV[*] /  \misp.
\]
Similarly, there are three other such isomorphisms
\begin{gather*}
 \int_{+}: \gsk \mgVF[*] \fun \gsk  \mgRV[*] /  \mgisp, \\
 \int^{\diamond}_+ : \gsk \mVF^\diamond[*] \fun \gsk  \mRV^{\db}[*] /  \mispdb,\\
   \int^{\diamond}_+  : \gsk \mgVF^\diamond[*] \fun \gsk  \mgRV^{\db}[*] /  \mgispdb.
\end{gather*}
\end{thm}
\begin{proof}
This is immediate by Corollary~\ref{L:sur:c}, Proposition~\ref{kernel:L}, Remarks~\ref{gL:sur:c} and \ref{kernel:gL}.
\end{proof}

\begin{thm}\label{thm:ring}
The  isomorphism $\int_+$ induces  two graded ring homomorphisms
\[
\Xint{\textup{e}}^g, \Xint{\textup{e}}^b: \ggk \mgVF[*] \to^{\int} \ggk \mgRV[*] / (\bm P) \two^{\mgE_{g}}_{\mgE_{b}} \F_3\abs\Gamma[X].
\]
In contrast, the  isomorphism $\int_{+}^\diamond$ induces only one graded ring homomorphism
\[
\Xint{\textup{e}}^\diamond : \ggk \mgVF^\diamond[*] \to^{\int^\diamond} \ggk \mgRV^{\db}[*] / (\bm P_\Gamma) \to^{\mgE^{\db}} \Z[X].
\]
\end{thm}
\begin{proof}
This is of course just a combination of Theorem~\ref{main:prop}, Remark~\ref{thekernel}, and Propositions~\ref{prop:eu:retr:k} and \ref{prop:eu:retr:k:db}.
\end{proof}

Let $F$ be a definable set with $A \coloneqq F_{\VF} \sub \VF^n$ and $\omega : F \fun \RV$ a definable function on $F$. Note that we are not assuming $F \in \VF_*$.

\begin{defn}
The pair $(F, \omega)$ is viewed as a representative of a \emph{definable} function
\[
(\bm F, \bm \omega) : A \fun \gsk \mRV[*] / \misp
\]
given by $a \efun [(F_a, \omega_a)] / \misp$, where $\omega_a$ is the function on $F_a$ induced by $\omega$.

Note that $[(F_a, \omega_a)]$ depends on the parameter $a$ and, for distinct $a, a' \in A$, there is a priori no way to compare $[(F_a, \omega_a)]$ and $[(F_{a'}, \omega_{a'})]$ unless we work over the substructure $\mdl S \la a, a' \ra$. Similarly, given another definable pair $(G, \sigma)$ of this form with $A = G_{\VF}$, we say that the corresponding definable function $(\bm G, \bm \sigma)$ is \emph{equivalent} to $(\bm F, \bm \omega)$, written as $(\bm G, \bm \sigma) \sim (\bm F, \bm \omega)$, if $(\bm G, \bm \sigma)(a) = (\bm F, \bm \omega)(a)$ over $\mdl S \la a \ra$ for all $a \in A$ outside a definable subset of $\VF$-dimension less than $n$. Thus, if $\dim_{\VF}(A) < n$ then there is only one such function on $A$, namely $0$. The set of all such functions, or rather the equivalence classes of such functions, is denoted by $\mfn_+(A)$, which is a  $\gsk  \mRV[*] / \misp$-semimodule. If $E \sub [n]$ is a nonempty set then, for each $a \in \pr_{E}(A)$, the definable function in $\mfn_+(A_a)$ represented by $(F_a, \omega_a)$ is denoted by $(\bm F_a, \bm \omega_a)$.

We say that  $(\bm F, \bm \omega)$ is \emph{proper invariant} if $(F, \omega)$ is proper invariant (recall Remark~\ref{nota:form:down}). Notice that if $(\bm G, \bm \sigma)$ is also proper invariant and $(\bm G, \bm \sigma) \sim (\bm F, \bm \omega)$ then $(\bm G, \bm \sigma)(a) = (\bm F, \bm \omega)(a)$ over $\mdl S \la a \ra$ for all $a \in A$. The set of proper invariant functions is denoted by $\mfn^\diamond_+(A)$, which is a  $\gsk  \mRV^{\db}[*] / \mispdb$-semimodule.


The semimodules $\mgfn_+(A)$, $\mgfn^\diamond_+(A)$ with $\Gamma$-volume forms are constructed likewise.
\end{defn}

\begin{defn}
Let $\bb L F = \bigcup_{a \in A} a \times F_a^\sharp$ and $\bb L \omega$ be the function on $\bb L F$ induced by $\omega$. Set $\int_{+A} (\bm F, \bm \omega) = \int_+ [(\bb L F, \bb L \omega)]$,
which, by Theorem~\ref{main:prop} and compactness, does not depend on the representative $(F, \omega)$ (alternatively, this also follows from Lemma~\ref{isp:VF:fiberwise:contract}). If $(\bm F, \bm \omega)$ is proper invariant then we replace $\int_+$ with $\int^\diamond_+$. Thus there are canonical homomorphisms of semimodules:
\begin{gather*}
 \int_{+A} : \mfn_+(A) \fun \gsk \mRV[*] /  \misp, \\
  \int_{+A}^\diamond : \mfn^\diamond_+(A) \fun \gsk \mRV^{\db}[*] /  \mispdb;
\end{gather*}
similarly for the semimodules  with $\Gamma$-volume forms.
\end{defn}

\begin{thm}\label{semi:fubini}
For all definable functions $(\bm F, \bm \omega)$ on $A$ and all nonempty sets $E, E' \sub [n]$,
\[
\int_{+ a \in \pr_{E}(A)} \int_{+ A_a}  (\bm F_a, \bm \omega_a) = \int_{+ a \in \pr_{E'}(A)} \int_{+ A_a} (\bm F_a, \bm \omega_a).
\]
\end{thm}
\begin{proof}
This is clear since both sides equal $\int_{+A} (\bm F, \bm \omega)$.
\end{proof}

Let $B \sub \VF^n$ be a definable set and $\phi : B \fun A$ a definable  bijection. Suppose that $\phi$ is a $\mVF^\diamond[n]$-morphism (volume forms are ignored) if $(F, \omega)$ is proper invariant.

Let $F_B = \bigcup_{b \in B} b \times F_{\phi(b)}$. For simplicity, the bijection $F_B \fun F$ induced by $\phi$ is still denoted by $\phi$. By Corollary~\ref{cov:rest}, if $(F, \omega)$ is proper invariant then $(F_B, \omega \circ \phi)$ is proper invariant as well. So $(F_B, \omega \circ \phi)$ represents a function in $\mfn_+(B)$ or $\mfn^\diamond_+(B)$, which does not depend on the representative $(F, \omega)$ and is denoted by $( \bm F \circ \phi, \bm \omega \circ \phi)$. The corresponding functions
\[
\phi_* : \mfn_+(A) \fun \mfn_+(B), \quad \phi_* : \mfn^\diamond_+(A) \fun \mfn^\diamond_+(B);
\]
are indeed isomorphisms of semimodules; similarly for the semimodules  with $\Gamma$-volume forms. However, in general, these do not preserve integrals.


Let $\sigma : B \fun \RV$ be a definable function such that $\sigma = \rv \circ (\jcb_{\VF} \phi)$ outside a definable subset of $\VF$-dimension less than $n$ and,  if $(F, \omega)$ is proper invariant, then $\sigma$ is proper covariant. Let  $\omega  \sigma$ be the function on $F_{B}$ given by $(b, t) \efun (\omega \circ \phi)( b, t) \cdot \sigma( b)$. Then  $(F_B, \omega  \sigma)$ represents another function in $\mfn_+(B)$ or $\mfn^\diamond_+(B)$, which is denoted by $( \bm F \circ \phi, \bm {\omega \sigma})$.


\begin{defn}
The \emph{Jacobian transformation} of $(\bm F, \bm \omega)$ with respect to $\phi$ and $\sigma$ is given by
$\phi^{\sigma}_{*}(\bm F, \bm \omega) = ( \bm F \circ \phi, \bm {\omega \sigma})$.
\end{defn}

The Jacobian transformation itself does depend on the choice of $\sigma$. However, the point is that the integral does not:

\begin{thm}\label{semi:change:variables}
$\int_{+A} ( \bm F, \bm \omega) = \int_{+B} \phi^{\sigma}_{*}(\bm F, \bm \omega)$.
\end{thm}
\begin{proof}
Let $\bb L \phi$ be the bijection $\bb L F_B \fun \bb L F$ induced by $\phi$.
For all $(b, b') \in \bb L F_B$ outside a definable subset of $\VF$-dimension less than $n+m$, where $F_{\RV} \sub \RV^m$, we have $\jcb_{\VF} \bb L \phi (b, b') = \jcb_{\VF} \phi(b)$. So  $\bb L \phi$ is a $\mVF[*]$- or $\mVF^\diamond[*]$-morphism $(\bb L F_B, \bb L \omega\sigma) \fun (\bb L F, \bb L \omega)$. The equality then follows from Theorem~\ref{main:prop}.
\end{proof}

\providecommand{\bysame}{\leavevmode\hbox to3em{\hrulefill}\thinspace}
\providecommand{\MR}{\relax\ifhmode\unskip\space\fi MR }
\providecommand{\MRhref}[2]{%
  \href{http://www.ams.org/mathscinet-getitem?mr=#1}{#2}
}
\providecommand{\href}[2]{#2}

%

\enddocument